\documentclass[11pt]{article}

\usepackage{amsmath,amssymb,amsthm}
\usepackage{graphicx,verbatim}
\usepackage{algorithm,algorithmic}
\usepackage[small,bf]{caption}
\usepackage{subcaption} 
\usepackage{cases}
\usepackage[table,dvipsnames]{xcolor}
\usepackage[shortlabels]{enumitem}
\usepackage{hyperref,url}
\usepackage{cite}
\usepackage{booktabs,adjustbox,multirow}  
\usepackage{makecell}
\usepackage{import} 
\usepackage{aligned-overset} 
\usepackage[toc,page]{appendix}
\usepackage[normalem]{ulem}
\hypersetup{
	colorlinks=true,
	linkcolor=red,     
	urlcolor=magenta,
	citecolor={blue},
}
\usepackage{tikz,psfrag,pgfplots}
\pgfplotsset{compat=newest}
\usetikzlibrary{plotmarks,arrows.meta}
\usepgfplotslibrary{patchplots,external}
\usepackage{lmodern}
\usepackage{dsfont} 
\usepackage{bm,bbm} 
\usepackage{upgreek} 
\usepackage{pifont}
\newcommand{\xmark}{\ding{55}}
\newtheorem{lemma}{{Lemma}}
\newtheorem{theorem}{{Theorem}}
\newtheorem{corollary}{{Corollary}}
\newtheorem{proposition}{{Proposition}}
\newtheorem{assumption}{{Assumption}}
\newtheorem{remark}{{Remark}}

\newcommand{\grad}{{\nabla}}   
\newcommand{\one}{\mathbf{1}}   
\newcommand{\real}{\mathbb{R}}  

\newcommand{\diag}{{\rm diag}}  
\newcommand{\bdiag}{{\rm blkdiag}}  
\newcommand{\col}{\mathrm{col}}     

\newcommand{\Tr}{\mathop{\rm tr}}    

\newcommand{\Ex}{\mathop{\mathds{E}{}}}

\newcommand{\argmin}{\mathop{\rm argmin}}

\newcommand{\ie}{{\it i.e.}}

\renewcommand{\top}{\textit{\footnotesize \texttt{T}}} 
\newcommand{\define}{\triangleq} 

\def\bxi{\boldsymbol{\xi}}

\def\B{{\mathbf{B}}}

\def\D{{\mathbf{D}}}
\def\E{{\mathbf{E}}}
\def\F{{\mathbf{F}}}

\def\I{{\mathbf{I}}}

\def\P{{\mathbf{P}}}
\def\Q{{\mathbf{Q}}}

\def\V{{\mathbf{V}}}
\def\W{{\mathbf{W}}}

\def\d{{\mathbf{d}}}
\def\e{{\mathbf{e}}}
\def\f{{\mathbf{f}}}
\def\g{{\mathbf{g}}}

\def\m{{\mathbf{m}}}

\def\s{{\mathbf{s}}}

\def\u{{\mathbf{u}}}
\def\v{{\mathbf{v}}}

\def\x{{\mathbf{x}}}
\def\y{{\mathbf{y}}}
\def\z{{\mathbf{z}}}

\newcommand{\cC}{{\mathcal{C}}}
\newcommand{\cD}{{\mathcal{D}}}
\newcommand{\cE}{{\mathcal{E}}}
\newcommand{\cF}{{\mathcal{F}}}
\newcommand{\cG}{{\mathcal{G}}}
\newcommand{\cH}{{\mathcal{H}}}

\newcommand{\cL}{{\mathcal{L}}}

\newcommand{\cN}{{\mathcal{N}}}
\newcommand{\cO}{{\mathcal{O}}}

\newcommand{\cQ}{{\mathcal{Q}}}
\newcommand{\cR}{{\mathcal{R}}}

\def\evdots{\vbox{\baselineskip=2pt \lineskiplimit=0pt 
		\kern6pt \hbox{$.$}\hbox{$.$}\hbox{$.$}}}  
\usepackage[margin=1in,
includefoot,headsep=0pt 
]{geometry}
\usepackage{sectsty}
\allsectionsfont{\sffamily}
\usepackage{fancyhdr}

\begin{document}
			
\title{\LARGE\bfseries CED-EF: Compressed Exact Diffusion with Error Feedback for Multi-Agent Learning%
	\thanks{This work was supported and funded by Kuwait University Research Grant No.~EE02/23.}}

\author{%
	\begin{tabular}{cc}
		\textsf{Sulaiman Alghunaim}
		&
		\textsf{Kun Yuan}
		\\[0.4em]
		\small Department of Electrical Engineering
		&
		\small Center for Machine Learning Research
		\\
		\small Kuwait University, Kuwait
		&
		\small Peking University, China
		\\[0.2em]
		\small \texttt{sulaiman.alghunaim@ku.edu.kw}
		&
		\small \texttt{kunyuan@pku.edu.cn}
	\end{tabular}
}

\date{}
\maketitle

\begin{abstract}
		We study decentralized stochastic optimization over a network of \(N\) agents
		under compressed communication. We propose CED-EF, an exact diffusion-based
		method with error feedback that directly accommodates biased
		\(\delta\)-contractive compressors while communicating one compressed
		model-sized vector per node per iteration. For smooth nonconvex objectives
		with unbiased stochastic gradients whose variance is bounded by
		$\sigma^2$, where $\sigma\geq0$, we establish
		a convergence rate whose leading stochastic term is
		\(\mathcal O(\sigma/\sqrt{NK})\). For \(\sigma>0\), the dominant dependence of
		the corresponding transient time on the number of agents, compression level,
		and spectral gap $\Delta_\lambda$ is
		\(\mathcal O(N^3/(\delta^4\Delta_\lambda^4))\), with fixed
		problem-dependent factors suppressed. Under the Polyak--\L{}ojasiewicz
		condition, CED-EF attains a leading stochastic term
		\(\widetilde{\mathcal O}(\sigma^2/(NK))\) with transient time on the order of
		\(\widetilde{\mathcal O}(N/(\delta^2\Delta_\lambda^2))\).
		These dependencies improve the compression and/or network dependence of existing results. Numerical experiments on
		least-squares and logistic-regression problems illustrate the performance
		advantages of CED-EF.
\end{abstract}		
		

		\section{Introduction}
		\subsection{Problem setup}
		We consider a decentralized stochastic consensus optimization problem over a network of $N$ nodes, also referred to as agents, clients, or workers. Each node $i\in\{1,\ldots,N\}$ has access only to its local data distribution $\mathcal D_i$ and the associated local risk function $f_i:\mathbb R^n\to\mathbb R$. The goal is for the nodes to collaboratively minimize the global objective, defined as the average of the local risks:
		\begin{align}\label{min_problem}
			x^\star
			\in
			\argmin_{x\in\mathbb R^n}
			f(x)
			\define
			\frac1N\sum_{i=1}^N f_i(x),
			\quad
			f_i(x)
			\define
			\mathbb E_{\xi_i\sim\mathcal D_i}
			\bigl[
			F_i(x;\xi_i)
			\bigr].
		\end{align}
		Here, $F_i(x;\xi_i)$ denotes a stochastic realization of the loss at node
		$i$, and $\xi_i$ represents a random sample drawn from the local distribution
		$\mathcal D_i$. The objective
		is to minimize the global risk $f$ without requiring any node to access the
		private data or loss functions of the other nodes. We consider {\em decentralized} (fully distributed) settings, where the nodes
		are connected through a communication graph and may exchange information only
		with their immediate neighbors \cite{nedic2009distributed,sayed2014nowbook}. This differs from the {\em centralized} and
		{\em federated} setups, in which every node communicates with a central server
		that coordinates the updates \cite{boyd2011admm,konevcny2016federated,mcmahan2017communication}.
		
		Problem~\eqref{min_problem} is a standard model for decentralized machine learning, where data are naturally distributed across devices,
		workers, or institutions and communication is often the main
		bottleneck~\cite{mcmahan2017communication,konevcny2016federated,lian2017can,lian2018asynchronous}. It also has applications in control as it arises in multi-agent and multi-robot coordination, formation	control, and flocking~\cite{jadbabaie2003coordination,olfatisaber2006flocking,shorinwa2024distributed},
		and in distributed model-predictive control~\cite{mota2012distributedadmm}. Further applications include distributed estimation and in-network adaptive filtering over sensor networks~\cite{schizas2007consensus,sayed2014nowbook},
		as well as distributed spectrum sensing in cognitive-radio
		networks~\cite{bazerque2010distributed}. 
		In all of these settings, decentralization is desirable because it avoids a
		single point of failure, preserves data locality, and lets computation scale
		with the number of participating agents; in particular, for large data sets it
		becomes inefficient---or infeasible---to solve problem~\eqref{min_problem} on a
		single agent.

		In distributed optimization, communication can become a major bottleneck
		when high-dimensional model updates must be exchanged among the participating
		agents. Two common strategies are used to reduce this cost: allowing agents to perform multiple local updates between communication rounds \cite{mcmahan2017communication,karimireddy2019scaffold,koloskova2020unified,wang2021cooperative,alghunaim2023led}, and compressing the messages transmitted during communication  \cite{alistarh2018convergence,richtarik2021ef21,fatkhullin2023momentum,gao2023econtrol,tang2018communication,koloskova2019decentralized,koloskova2020Decentralized,tang2019deepsqueeze,carpentiero2021distributed,singh2022sparq,nassif2025differential,kajiyama2020linear,magnusson2020maintaining,kovalev2021linearly,liu2021linear,zhang2023innovation,huang2023cedas,yi2022communication,xie2024communication,liao2022compressed,zhao2022beer,yan2023compressed,islamov2024towards,faia2025communication,xu2025compressed}.
		
		 We focus on decentralized optimization with compressed communication and
		 develop a compressed variant of exact diffusion (ED)~\cite{yuan2019exactdiffI},
		 also known as NIDS~\cite{li2017nids}, equipped with error feedback. The method
		 handles biased contractive compressors and is analyzed for both smooth
		 nonconvex objectives and objectives satisfying the
		 Polyak--\L{}ojasiewicz (P\L) condition. We first review the most closely
		 related compressed decentralized methods and then summarize our contributions.

		\subsection{Related work}
		Decentralized optimization has been extensively studied in recent years, leading to a wide range of algorithms. Among the most prominent are consensus-based methods \cite{nedic2009distributed} and diffusion strategies \cite{lopes2008diffusion,chen2013distributed,sayed2014nowbook}, which are often collectively referred to as decentralized stochastic gradient descent (DSGD) in the recent literature \cite{koloskova2020unified}. To reduce communication costs, DSGD has been combined with compression in several works \cite{tang2018communication,koloskova2019decentralized,koloskova2020Decentralized,tang2019deepsqueeze,carpentiero2021distributed,singh2022sparq,nassif2025differential}. However, DSGD and its variants are generally not robust to data heterogeneity: when the local objective functions differ across nodes, these methods may converge to a biased solution under constant stepsizes \cite{chen2013distributed,yuan2016convergence,lian2017can}. To remove this bias, several bias-correction mechanisms have been developed. Two widely adopted classes are gradient-tracking methods \cite{di2016next,xu2015augmented} and primal-dual methods, including EXTRA \cite{shi2015extra} and exact diffusion (ED) \cite{yuan2019exactdiffI,li2017nids}. A unified treatment of these bias-corrected decentralized methods can be found in \cite{alghunaim2019decentralized,alghunaim2021unified}.

		 Several bias-corrected decentralized methods have also been extended to
		 compressed communication settings. The work \cite{kovalev2021linearly} develops a compressed primal-dual method with linear convergence for strongly convex objectives under unbiased compressors. Closely related compressed ED-type schemes are proposed in \cite{liu2021linear,zhang2023innovation}, where linear convergence is established in strongly convex settings for unbiased compressors. In addition, \cite{zhang2023innovation} also treats biased compressors, although its guarantees do not cover aggressive compression regimes. The work \cite{faia2025communication} studies compressed ED with error feedback, focusing again on unbiased compression and strongly convex costs. On the gradient-tracking side, \cite{liao2022compressed} proposes compressed GT-based algorithms with linear convergence for strongly convex objectives under biased compressors. In contrast to these works, which focus primarily on strongly convex optimization, our work studies the nonconvex stochastic setting.

		The work~\cite{huang2023cedas} studies CEDAS, an ED-based method closely
		related to~\cite{liu2021linear}, in the stochastic nonconvex setting. Its main
		convergence analysis assumes unbiased compression. The same work discusses a
		biased-to-unbiased conversion mechanism that introduces an additional
		compression operation and increases the communication cost. The
		work~\cite{xie2024communication} handles a general class of compressors
		directly and establishes the linear-speedup rate
		\(\mathcal O(1/\sqrt{NK})\) for smooth nonconvex objectives where $K$ is the number of iteration. Under the
		P\L~condition, however, its constant-stepsize result gives linear convergence
		to a noise neighborhood rather than a \(1/(NK)\) stochastic decay.
		Among gradient-tracking methods, BEER~\cite{zhao2022beer} supports biased
		compression but requires large mini-batches in its stochastic guarantee,
		whereas MoTEF~\cite{islamov2024towards} combines momentum tracking with error
		feedback and obtains linear speedup under arbitrary data heterogeneity.

		In this work, we focus on an ED-based construction for two reasons. First,
		CED-EF communicates one compressed model-sized vector per node per iteration,
		whereas the GT baselines BEER and MoTEF communicate two compressed
		model-sized quantities per iteration. Second, ED-type
		methods can exhibit a more favorable dependence on network connectivity than
		standard GT methods~\cite{alghunaim2021unified}. This motivates the development of an ED-based method that directly handles
		biased contractive compression while retaining favorable dependence on the
		network connectivity.

\subsection{Contributions}
Our main contributions are as follows. Here, \(N\) denotes the number of agents,
\(K\) the number of iteration, \(\sigma^2\) the uniform upper bound on the
stochastic-gradient variance, \(\delta\in(0,1]\) the contraction parameter of
the compressor, and \(\Delta_\lambda\) the spectral gap of the mixing matrix
defined in \eqref{lambda_Deltalam}.

\begin{enumerate}
	\item {\bfseries\small Algorithm.}
	We propose \emph{CED-EF}, a compressed exact diffusion method equipped
	with error feedback that directly accommodates biased
	\(\delta\)-contractive compressors. CED-EF communicates one compressed model-sized vector per node per iteration
	and requires neither unbiasedness of the compressor nor a lower bound on
	$\delta$ away from zero.
	
	\item {\bfseries\small Nonconvex convergence with linear speedup.}
	For smooth nonconvex objectives with unbiased stochastic gradients of
	bounded variance, we establish an ergodic stationarity rate with leading
	term $
	\mathcal O(\sigma/\sqrt{NK})$.
	The analysis requires neither bounded stochastic gradients nor a uniform
	bound on data heterogeneity. For \(\sigma>0\), the dominant
	network- and compression-dependent transient time is $
	\mathcal O(N^3/\delta^4\Delta_\lambda^4)$. This improves over existing methods as summarized in
	Table~\ref{tab:comparison_transient}. To the best of our knowledge, among the ED-type methods, this is the first
	linear-speedup guarantee under biased contractive compression in the nonconvex setting.
	
	\item {\bfseries\small P\L~convergence with linear speedup.}
	Under the P\L~condition, we establish a leading
	stochastic term $
	\widetilde{\mathcal O}(\sigma^2/NK)$. The dominant network- and
	compression-dependent transient time is $
	\widetilde{\mathcal O}(N/\delta^2\Delta_\lambda^2)$. This improves over existing methods as summarized in
	Table~\ref{tab:comparison_transient}.

	\item {\bfseries\small Relation to unbiased compression.}
	For an unbiased compressor with variance parameter \(C\ge0\), the standard
	scaling yields a contractive compressor with
	\(\delta^{-1}=1+C\). Consequently, for \(C\ge1\), the dominant network- and
	compression-dependent factors scale as \(C^4\) and \(C^2\) in the nonconvex
	and P\L~settings, respectively, compared with the \(C^6\) and \(C^3\)
	dependencies reported for CEDAS.
\end{enumerate}

		\section{Proposed algorithm}
		\subsection{Network combination matrix}
		We let $w_{ij} \ge 0$ denote the weight that agent $i$ assigns to the information received from agent $j$.  
		The weights are assumed to satisfy $w_{ij} = 0$ whenever $j \notin \mathcal{N}_i$, where $\mathcal{N}_i$ 
		denotes the set of agents directly connected to agent $i$ including agent $i$ so that \(w_{ii}\) denotes the self-weight of agent
		\(i\). In other words, each agent only assigns nonnegative weights to itself and direct neighbors. We further define the combination (weight or mixing) matrix $W = [w_{ij}] \in \mathbb{R}^{N \times N}$, 
		whose $(i,j)$th entry is given by $w_{ij}$. We assume that $W$ satisfies Assumption \ref{assump:network}.

		\subsection{CED-EF}
	\paragraph{Exact Diffusion.}
	Exact Diffusion (ED) is a decentralized method originally proposed in \cite{yuan2019exactdiffI}. Given $x_i^0$, the primal-dual form of the ED recursion can be written as \cite{alghunaim2019decentralized}:
	\begin{subequations}\label{ed_updates}
		\begin{align}
			\hat{x}_i^{k}
			&=
			x_i^{k}
			-\alpha \grad F_i(x_i^k;\xi_i^k)
			-y_i^k,
			\\
			x_i^{k+1}
			&=
			\sum_{j\in\cN_i}w_{ij}\hat{x}_j^k,
			\\
			y_i^{k+1}
			&=
			y_i^k+\hat{x}_i^k-x_i^{k+1}.
		\end{align}
	\end{subequations}
	Here $y_i^k$ is a correction dual variable, also called a control variate, and $\alpha>0$ is the primal stepsize.
	
\paragraph{CED-EF}
We propose a compressed variant of ED with error-feedback, named CED-EF, listed in Algorithm~\ref{alg:ced-ef}. Here $\cC_i:\real^n\to\real^n$ is a (possibly biased) contractive compression operator satisfying Assumption~\ref{assump:comp}. The algorithm uses five scalar parameters: the gradient stepsize
\(\alpha>0\), the consensus parameter \(\gamma\in(0,1]\), the primal and dual
mixing weights \(\beta,\theta>0\), and the error-feedback leakage parameter
\(\eta\in(0,1]\).
\begin{algorithm}[h]
	\caption{\sc \small CED-EF Update for Agent $i$}
	\label{alg:ced-ef}
	\textbf{given} positive stepsizes $\alpha,\beta,\gamma,\theta$, leakage $\eta\in(0,1]$, mixing matrix $W$, and initial points  $\phi_i^0$,
	$\psi_i^0=\sum_{j\in\mathcal N_i}w_{ij}\phi_j^0$,
	$x_i^0=(1-\gamma)\phi_i^0+\gamma\psi_i^0$,
	$y_i^0=0$, and $e_i^0=0$.
	\medskip
	For $k=0,1,2,\ldots$, agent $i$ performs the following updates:
	\begin{subequations}\label{ced-ef}
		\begin{enumerate}
			\item Compute:
			\begin{align}
				\hat{x}_i^{k} &= x_i^{k}-\alpha\nabla F_i(x_i^{k};\xi_i^{k})-y_i^{k}. \label{hat_x_ced}
			\end{align}
			\item Create the compressed message
			$q_i^{k}=\cC_i(\hat{x}_i^{k}-\phi_i^{k}+\eta e_i^{k})$  and update:
			\begin{align}
				e_i^{k+1} &= \hat{x}_i^{k}-\phi_i^{k}+e_i^{k}-q_i^{k} \label{e_ced} \\
				\phi_i^{k+1} &= \phi_i^{k}+\beta q_i^{k}.
			\end{align}
			\item Send $q_i^k$ to neighbors, receive $q_j^{k}$ from
			$j\in\cN_i$, and update
			\begin{align}
				\psi_i^{k+1} &= \psi_i^{k}+\beta \sum_{j\in\cN_i} w_{ij} q_j^{k}\label{surr_ced} \\
				x_i^{k+1} &= (1-\gamma) \phi_i^{k+1}+\gamma \psi_i^{k+1}.\label{x_ced}
			\end{align}
			
		\item Construct $\hat{\phi}_i^{k} = \phi_i^{k}+\theta  q_i^{k}$ and $\hat{\psi}_i^{k} = \psi_i^{k}+\theta \sum_{j\in\cN_i} w_{ij} q_j^{k}$ and update the  dual variables
		\begin{align}
			y_i^{k+1} &= y_i^{k}+\gamma\big(\hat{\phi}_i^{k}-\hat{\psi}_i^{k}\big).\label{y_ced}
		\end{align}	
		\end{enumerate}
	\end{subequations}
\end{algorithm}
	
\noindent
In Algorithm~\ref{alg:ced-ef}, communication occurs only in Step~3. Each agent
broadcasts a single compressed vector $q_i^k$, receives the corresponding
vectors from its neighbors, and forms
$\sum_{j\in\cN_i}w_{ij}q_j^k$. In addition to $x_i^k$, $y_i^k$, and $e_i^k$,
each agent stores the two surrogate variables $\phi_i^k$ and $\psi_i^k$.
The variable $e_i^k$ accumulates the compression error leakage and feeds it back into
the subsequent compressed message.

In the uncompressed case, \(\mathcal C_i=I\). If \(e_i^0=0\), then
\(e_i^k=0\) for all \(k\), for any \(\eta\in(0,1]\). Moreover, with
\(\beta=\theta=1\), CED-EF reduces to exact diffusion with mixing matrix $
W_\gamma=(1-\gamma)I_N+\gamma W$.
		
	\section{Convergence result}
	\subsection{Main assumptions}
			\begin{assumption}[\sc \small Weight matrix] \label{assump:network} \rm
			The matrix \(W=[w_{ij}]\) is nonnegative, symmetric, doubly stochastic,
			primitive, and compatible with the communication graph, \ie,
			\(w_{ij}=0\) whenever \(j\notin\mathcal N_i\).
		\end{assumption} 
		\noindent   Under Assumption~\ref{assump:network}, the eigenvalues of $W$ can be ordered as $
		1=\lambda_1>\lambda_2\ge\cdots\ge\lambda_N>-1$.
		We define
		\begin{subequations} \label{lambda_Deltalam}
			\begin{align}
			\lambda&\define\lambda_2=\max_{i\ge2}\lambda_i, \quad \Delta_\lambda\define1-\lambda \\
			\underline\lambda&\define\lambda_N=\min_{i\ge2}\lambda_i, \quad \underline\Delta_\lambda\define1-\underline\lambda
		\end{align}
		\end{subequations}
			Here $\Delta_\lambda$ is the spectral gap associated with the second-largest eigenvalue, while $\underline\Delta_\lambda$ captures the dependence on the 	lower end of the nontrivial spectrum. Note that $0<\Delta_\lambda
		\le
		\underline\Delta_\lambda
		<2$.
		
		\begin{assumption}[\sc \small Smoothness and Lower Boundedness]
			\label{assump:smoothness} \rm
			Each local objective function $f_i : \mathbb{R}^n \to \mathbb{R}$ has $L$-Lipschitz continuous gradients ($L$-smooth). 
			That is, for all $x, x' \in \mathbb{R}^n$ and for all agents $i$, 
			\begin{equation}
				\|\nabla f_i(x) - \nabla f_i(x')\| \le L \|x - x'\|.
			\end{equation}
			In addition,  $f^\star \define \inf_{x \in \mathbb{R}^n} f(x)>-\infty$. 	
		\end{assumption}
		Both the stochastic gradient oracle and the compression operator may be randomized.
		We collect all of this randomness in a single filtration $\{\bm{\cF}^k\}_{k\ge0}$, where $\bm{\cF}^k$
		is generated by the gradient samples and the compressor's internal randomness drawn in
		iterations $0,1,\dots,k-1$, together with the deterministic initialization. Consequently the
		iterates $x_i^k,\psi_i^k,y_i^k,e_i^k$ are $\bm{\cF}^k$-measurable, whereas the
		iteration-$k$ gradient samples $\{\xi_i^k\}_{i=1}^N$ and the iteration-$k$ compressor
		randomness are independent of $\bm{\cF}^k$.
		
		\begin{assumption}[\sc \small Stochastic Gradient Oracle]
			\label{assump:stochastic_gradient} \rm
				Each agent \(i\) has access to an unbiased stochastic gradient
				\(\nabla F_i(x_i^k;\xi_i^k)\) with bounded conditional variance: there exists
				\(\sigma\ge0\) such that, for every \(i\) and \(k\),
				\begin{subequations}\label{eq:noise_cond}
					\begin{align}
						\mathbb E\!\left[
						\nabla F_i(x_i^k;\xi_i^k)
						\mid \bm{\cF}^k
						\right]
						&=
						\nabla f_i(x_i^k),
						\\
						\mathbb E\!\left[
						\|\nabla F_i(x_i^k;\xi_i^k)-\nabla f_i(x_i^k)\|^2
						\mid \bm{\cF}^k
						\right]
						&\le \sigma^2.
					\end{align}
				\end{subequations}
				For each \(k\), conditioned on \(\bm{\cF}^k\), the current samples
				\(\{\xi_i^k\}_{i=1}^N\) are mutually independent across agents. We further
				assume that the samples used at iteration \(k\) are independent of the
				randomness generated at previous iterations.
			
		\end{assumption}
		
		\begin{assumption}[\sc \small Compression operator]\label{assump:comp} \rm
			The (possibly randomized) compression operator $\cC_i:\real^n\to\real^n$ is $\delta$-contractive:
			for every input $x\in\real^n$,
			\begin{align}\label{compresion_contractive}
				\Ex\big[\|\cC_i(x)-x\|^2\mid x\big]\leq(1-\delta)\|x\|^2,
				\quad 0<\delta\le1,
			\end{align}
			where the expectation is taken over the internal randomness of $\cC_i$, which is independent of
			$\bm{\cF}^k$ and of the gradient samples.
			
		\end{assumption}

	\subsection{General nonconvex results}
We next give our convergence result in the general nonconvex setting.
\begin{theorem}[\sc \small Convergence in nonconvex setting]\label{thm_convergence}
	\rm
	Suppose Assumptions~\ref{assump:network}--\ref{assump:comp} hold and
	the algorithm parameters $\alpha,\beta,\gamma,\eta,\theta$ satisfy the
	admissibility conditions in \eqref{final_conditions_fix}. Then, the iterates generated by CED-EF, initialized with $\phi_i^0 = x^0 $ for all $i$, satisfy
	\begin{equation}
		\begin{aligned}
			\frac{1}{K}\sum_{k=0}^{K-1}
			\Ex\|\nabla f(\bar\phi_e^k)\|^2
			&\lesssim
			\frac{\tilde f(x^0)}{\alpha\beta K}
			+
			\alpha\beta\Psi_1
			+
			\alpha^2\beta^2
			\left(\Psi_2+\frac{\widetilde\Psi_2}{K}\right)
			+
			\alpha^4\beta^4\Psi_3,
		\end{aligned}
	\end{equation}
	where $\bar\phi_e^k\define \frac{1}{N}\sum_{i=1}^N(\phi_i^k+\beta e_i^k)$, $\tilde f(x)\define f(x)-f^\star$, and
	\begin{subequations}\label{Psi_order}
		\begin{align}
			\Psi_1
			&=
			\cO\left(\frac{L\sigma^2}{N}\right),
			\label{Psi1_order}
			\\
			\Psi_2
			&=
			\cO\left(
			L^2\sigma^2
			\left[
			\frac{1}{\nu\theta\gamma\Delta_{\lambda}}
			+
			\frac{1-\delta}{\eta^2\delta^2}
			+
			\frac{1}{\beta\gamma\underline\Delta_\lambda}
			\right]
			\right),
			\label{Psi2_order}
			\\
			\widetilde\Psi_2
			&=
			\cO\left(
			\frac{L^2\varsigma_0^2}
			{\nu\theta\beta\gamma^2\Delta_{\lambda}^2}
			\right),
			\label{tildePsi2_order}
			\\
			\Psi_3
			&=
			\cO\left(
			\frac{L^4\sigma^2}
			{N\nu\theta^2\beta\gamma^3\Delta_{\lambda}^3}
			\right).
			\label{Psi3_order}
		\end{align}
	\end{subequations}
	Here, $
	\varsigma_0^2
	\define
	\frac1N\sum_{i=1}^N
	\|
	\nabla f_i(x^0)-\nabla f(x^0)
	\|^2$,
	$\Delta_\lambda\define 1-\lambda$, $\underline\Delta_\lambda\define 1-\underline\lambda$, and $
	\nu
	\define
	1-\frac{(\beta+\theta)^2\gamma\underline\Delta_\lambda}{4\beta}>0$.

\end{theorem}
Theorem~\ref{thm_convergence} is proved in
Appendix~\ref{app_noncvx_proof}. Corollary~\ref{corr:rate}, proved in
Appendix~\ref{app_corr_step_rate}, specializes this result to obtain explicit
convergence and transient-time rates.
\begin{corollary}[\sc \small Convergence rate]\label{corr:rate}
	\rm
Assume that the conditions of Theorem~\ref{thm_convergence} hold. In the heavy-compression regime where \(1-\delta = \Theta(1)\), we may choose the parameters as
\begin{align}
	\theta=\Theta(1),
	\quad
	\eta=\Theta(1),
	\quad
	\beta
	\asymp
	\frac{\delta\Delta_\lambda}
	{\underline\Delta_\lambda},
	\quad
	\gamma
	\asymp
	\frac{\delta\Delta_\lambda}
	{\underline\Delta_\lambda^2},
\end{align}
with the implied constants chosen sufficiently small that the admissibility conditions \eqref{final_conditions_fix} hold and $\nu=\Theta(1)$. With this choice, there exists a stepsize product $\tau=\alpha\beta$ (depending on the number of iterations $K$; see \eqref{tau_optimized_choice}) achieving the following convergence rate
\begin{equation}\label{eq:opt_step_rate}
	\begin{aligned}
		\frac1K\sum_{k=0}^{K-1}
		\Ex\|\nabla f(\bar\phi_e^k)\|^2
		&\le
		\cO\left(\sqrt{\tfrac{\sigma^2}{NK}}\right)
		+
		\cO\left(
		\Bigl(\tfrac{\sigma\,\underline\Delta_\lambda}{\delta\Delta_\lambda}\Bigr)^{2/3}
		\frac{1}{K^{2/3}}
		\right)
		+
		\cO\left(
		\Bigl(\tfrac{\sigma^2}{N}\Bigr)^{1/5}
		\Bigl(\tfrac{\underline\Delta_\lambda}{\Delta_\lambda}\Bigr)^{7/5}
		\tfrac{1}{\delta^{4/5}K^{4/5}}
		\right)
		\\
		&\quad
		+
		\cO\left(
		\varsigma_0^{2/3}
		\Bigl(\tfrac{\underline\Delta_\lambda}{\Delta_\lambda}\Bigr)^{5/3}
		\frac{1}{\delta K}
		\right)
		+
		\cO\left(
		\frac{\underline\Delta_\lambda^2}{\delta\Delta_\lambda^2}
		\frac{1}{K}
		\right).
	\end{aligned}
\end{equation}
Moreover, in terms of the network and compression parameters,   linear-speedup is achieved after transient time $K\ge K_{\mathrm{tr}}$, where
\begin{equation}\label{Ktr_dom_explicit_opt}
	K_{\textnormal{tr}}
	=
	\cO\left(
	\max\left\{
	\frac{N^3\underline\Delta_\lambda^4}
	{\delta^4\Delta_\lambda^4},
	~
	\frac{N\varsigma_0^{4/3}\underline\Delta_\lambda^{10/3}}
	{\delta^2\Delta_\lambda^{10/3}},
	~
	\frac{N\underline\Delta_\lambda^{14/3}}
	{\delta^{8/3}\Delta_\lambda^{14/3}}
	\right\}
	\right).
\end{equation}

\end{corollary}
The leading term in \eqref{eq:opt_step_rate} is
$\cO(\sigma/\sqrt{NK})$, which matches the centralized stochastic scaling.
The remaining terms depend on the network connectivity, compression level,
and initial heterogeneity and determine the transient period before this
leading term dominates. In particular, $\varsigma_0^2$ appears only in a
lower-order term and therefore does not affect the asymptotic
linear-speedup rate.  Compared with GT-based compressed methods, CED-EF communicates a single compressed vector per agent per iteration, halving the per-iteration communication cost.
\begin{remark}[\sc \small Transient time] \label{remark_tr_ncvx}
	\rm
	The transient time is dominated by its first term,
	\begin{equation}\label{Ktr_largeN_opt}
		K_{\textnormal{tr}}
		=
		\cO\!\left(\frac{N^3\underline\Delta_\lambda^4}{\delta^4\Delta_\lambda^4}\right),
	\end{equation}
	whenever the other two terms in \eqref{Ktr_dom_explicit_opt} do not exceed it, \ie, 
	\begin{equation}\label{Ktr_cond_network}
		\Delta_\lambda\ge\frac{\delta^{2}\underline\Delta_\lambda}{N^3},\quad 
		\varsigma_0^2\le\frac{N^3\underline\Delta_\lambda}{\delta^{3}\Delta_\lambda}.
	\end{equation}
	The dominance conditions themselves become less restrictive as \(\delta\)
	decreases because the first term deteriorates more rapidly with compression
	than the competing transient terms. This observation concerns only which term
	dominates the bound; it should not be interpreted as compression improving the
	transient time. In fact, the transient time increases as \(\delta\) decreases.

	 With $\underline\Delta_\lambda=\Theta(1)$, these conditions reduce to
	 $\Delta_\lambda\gtrsim\delta^2/N^3$ and
	 $\varsigma_0^2\lesssim N^3/(\delta^3\Delta_\lambda)$. The first condition
	 holds for common network families~\cite{nedic2018network}: for example,
	 $\Delta_\lambda=\Theta(1)$ for complete and expander graphs,
	 $\Delta_\lambda\asymp1/N$ for the two-dimensional grid/torus and star, and
	 $\Delta_\lambda\asymp1/N^2$ for the ring and line. Table~\ref{tab:comparison_transient} compares this bound with the transient-time
	 bounds available for related methods.
	
\end{remark}

\begin{remark}[\sc \small Unbiased compressors]
	\rm
	Suppose that \(\mathcal Q\) is unbiased and satisfies $
	\mathbb E\|\mathcal Q(x)-x\|^2
	\le
	C\|x\|^2$.
	Then the scaled compressor $
	\mathcal C(x)
	\define
	\frac{1}{1+C}\mathcal Q(x)$
	is contractive with parameter
	\(\delta=1/(1+C)\)~\cite{beznosikov2023biased}. Consequently,
	\eqref{Ktr_largeN_opt} becomes 
	\begin{align}
		K_{\mathrm{tr}}
		=
		\mathcal O\!\left(
		\frac{
			N^3(1+C)^4\underline\Delta_\lambda^4
		}{
			\Delta_\lambda^4
		}
		\right)
	\end{align}
	In particular, for \(C\ge1\), this yields $
	K_{\mathrm{tr}}
	=
	\mathcal O\!\left(
	\frac{
		N^3C^4\underline\Delta_\lambda^4
	}{
		\Delta_\lambda^4
	}
	\right)$.
	This improves the \(C^6\) dependence established for CEDAS under unbiased
	compression~\cite{huang2023cedas}. Moreover, CED-EF directly accommodates
	biased contractive compressors, including deterministic sparsifiers such as
	Top-\(c\).
	
\end{remark}

\begin{remark}[\small \sc Stepsize choice] \label{remark_standard_step}\rm
	The stepsize in \eqref{tau_optimized_choice} is selected by balancing the
	terms in the convergence bound, following the approach used in
	\cite{koloskova2020unified,islamov2024towards,stich2019unified}. This choice yields a nearly optimal bound covering both the stochastic and
	deterministic settings; for example, setting $\sigma=0$ gives the deterministic rate $\cO(
	\varsigma_0^{2/3}
	\Bigl(\tfrac{\underline\Delta_\lambda}{\Delta_\lambda}\Bigr)^{5/3}
	\frac{1}{\delta K})
	+
	\cO(
	\frac{\underline\Delta_\lambda^2}{\delta\Delta_\lambda^2}
	\frac{1}{K})$. If we instead use the typical choice  $\alpha\beta=\cO(\sqrt{N/K})$, then for large $K$, we have
	\begin{equation}\label{eq:cor_final_rate}
		\begin{aligned}
			\frac{1}{K}\sum_{k=0}^{K-1}
			\Ex\|\nabla f(\bar\phi_e^k)\|^2
			&\le
			\cO\left(\frac{1+\sigma^2}{\sqrt{NK}}\right)
			+
			\cO\left(
			\frac{N\sigma^2}{K}
			\cdot
			\frac{\underline\Delta_\lambda^2}
			{\delta^2\Delta_\lambda^2}
			\right)
			\\
			&\quad
			+
			\cO\left(
			\frac{N\varsigma_0^2}{K^2}
			\cdot
			\frac{\underline\Delta_\lambda^5}
			{\delta^3\Delta_\lambda^5}
			\right)
			+
			\cO\left(
			\frac{N\sigma^2}{K^2}
			\cdot
			\frac{\underline\Delta_\lambda^7}
			{\delta^4\Delta_\lambda^7}
			\right).
		\end{aligned}
	\end{equation}
This choice does not recover the $\cO(1/K)$ deterministic rate when $\sigma=0$, since the stepsize is not optimized. 
\end{remark}

\subsection{P\L~condition result}
Here, we state the convergence result for CED-EF under the P\L~condition given below.
\begin{assumption}[\sc \small P\L~condition] \label{assump_pl} \rm The aggregate function $f(x) = \frac{1}{N} \sum_{i=1}^N f_i(x)$ satisfies the P\L~inequality:
	\begin{equation} \label{PL_ineq_assump}
		2\mu \left(f(x) - f^\star\right) \leq \|\nabla f(x)\|^2, \quad \forall~x \in \mathbb{R}^n, 
	\end{equation}
for some $0<\mu\le L$, where $f^\star$ is defined in Assumption~\ref{assump:smoothness}. 
\end{assumption}
\begin{theorem}[\sc \small Convergence of CED-EF in the P\L~setting]
	\label{thm:pl_rate}
	\rm
	Suppose Assumptions~\ref{assump:network}--\ref{assump_pl} hold. Suppose also that the algorithm
	parameters $\alpha,\beta,\gamma,\eta,\theta$ satisfy the P\L~admissibility
	conditions in \eqref{PL_explicit}.	Then, the iterates generated by CED-EF, initialized with $\phi_i^0 = x^0 $ for all $i$, satisfy
	\begin{align}\label{PL_rate_general_tau}
		\Ex\tilde f(\bar\phi_e^K)
		&\le
		\exp\left(
		-\frac{\mu\tau K}{2}
		\right)
		\left(
		\tilde f(\bar\phi_e^0)+\widetilde a_2\tau^2
		\right)
		+
		\cO\left(
		\frac{\tau L\sigma^2}{\mu N}
		\right)
		\nonumber\\
		&\quad
		+
		\cO\left(
		\frac{\tau^2L^2\sigma^2}{\mu}
		\left[
		\frac{1}{\nu\theta\gamma\Delta_\lambda}
		+
		\frac{1-\delta}{\eta^2\delta^2}
		+
		\frac{1}{\nu\beta\gamma\underline\Delta_\lambda}
		\right]
		\right)
		\nonumber\\
		&\quad
		+
		\cO\left(
		\frac{\tau^4L^4\sigma^2}
		{\mu N\nu\theta^2\beta\gamma^3\Delta_\lambda^3}
		\right),
	\end{align}
	where $
	\tau\define \alpha\beta$, $\bar\phi_e^k
	\define
	\frac1N\sum_{i=1}^N(\phi_i^k+\beta e_i^k)$,
	$\tilde f(x)\define f(x)-f^\star$, $\widetilde a_2
	=
	\cO\left(
	\frac{L^2 \bar{\tau}\varsigma_0^2}
	{\nu\theta\beta\gamma^2\Delta_\lambda^2}
	\right)$ ($\bar\tau$ denotes the largest admissible value of $\tau$),
	with $
	\varsigma_0^2
	\define
	\frac1N\sum_{i=1}^N
	\left\|
	\nabla f_i(x^0)-\nabla f(x^0)
	\right\|^2$ and $\Delta_\lambda\define 1-\lambda$, $\underline\Delta_\lambda\define 1-\underline\lambda$, $\nu
	\define
	1-\frac{(\beta+\theta)^2\gamma\underline\Delta_\lambda}{4\beta}>0$.
	
\end{theorem}
Theorem \ref{thm:pl_rate} is proven in Appendix \ref{app_pl_proof}. The following convergence rate result follows from the previous theorem.
\begin{corollary}[\sc \small Convergence rate in P\L~setting]
	\label{corr:pl_rate}
	\rm
	Suppose the conditions in Theorem \ref{thm:pl_rate} hold.  In the heavy-compression regime where \(1-\delta = \Theta(1)\), we may choose the parameters as
\begin{align}
		\theta=\Theta(1),
	\quad
	\eta=\Theta(1),
	\quad
	\nu=\Theta(1),\quad
	\beta
	\asymp
	\frac{\delta\Delta_\lambda}
	{\underline\Delta_\lambda},
	\quad
	\gamma
	\asymp
	\frac{\delta\Delta_\lambda}
	{\underline\Delta_\lambda^2}.
\end{align}
With this choice, there exists  a stepsize $\tau\le\bar\tau$ (depending on $K$), where $\bar\tau$ denotes the largest admissible value of $\tau$, that achieves the following convergence rate:
\begin{equation}\label{PL_final_rate}
		\begin{aligned}
		\Ex\tilde f(\bar\phi_e^K)
		&\leq
		\widetilde\cO\left(
		\frac{L\sigma^2}{\mu^2NK}
		\right)
		+
		\widetilde\cO\left(
		\frac{L^2\sigma^2\underline\Delta_\lambda^2}
		{\mu^3\delta^2\Delta_\lambda^2K^2}
		\right)
		\\
		&\quad
		+
		\widetilde\cO\left(
		\frac{L^4\sigma^2\underline\Delta_\lambda^7}
		{\mu^5N\delta^4\Delta_\lambda^7K^4}
		\right)
		+
		\left(
		\tilde{f}(\bar\phi_e^0)+\tilde{\varsigma}_0^2
		\right)
		\exp\left(
		-\frac{\mu\bar\tau K}{2}
		\right),
	\end{aligned}
\end{equation}
	with  $
	\tilde{\varsigma}_0^2
	\define
	\widetilde a_2\bar\tau^2
	=
	\mathcal O\!\left(
	\frac{
		\varsigma_0^2\Delta_\lambda
	}{
		L\underline\Delta_\lambda
	}
	\right)$. 
Moreover, suppose \(\sigma>0\). Treating the condition number
\(\kappa=L/\mu\) as fixed and suppressing \(\kappa\)-dependent and logarithmic
factors in \(\widetilde{\mathcal O}(\cdot)\), linear speedup is achieved after
\(K\ge K_{\mathrm{tr}}\), where
	\begin{align}\label{PL_transient}
		K_{\mathrm{tr}}
		&=
		\widetilde\cO\left(
		\max\left\{
		\frac{N\underline\Delta_\lambda^2}
		{\delta^2\Delta_\lambda^2},
		\frac{\underline\Delta_\lambda^{7/3}}
		{\delta^{4/3}\Delta_\lambda^{7/3}}
		\right\}
		\right).
	\end{align}
	
\end{corollary}
The first term in \eqref{PL_final_rate} decays exponentially with $K$, while
the leading stochastic term is
$\widetilde{\cO}(\sigma^2/(NK))$. The remaining stochastic terms decay faster
in $K$ and therefore affect only the transient period before the
$1/(NK)$ term dominates. 
\begin{remark}[\sc \small P\L~transient time]
	\rm
	The P\L~transient time in \eqref{PL_transient} is dominated by its first term,
	\begin{equation}\label{PL_transient_dominant}
		K_{\mathrm{tr}}
		=
		\widetilde\cO\left(
		\frac{N\underline\Delta_\lambda^2}
		{\delta^2\Delta_\lambda^2}
		\right),
	\end{equation}
	whenever the second term in \eqref{PL_transient} does not exceed it. This
	holds iff $
		\Delta_\lambda
		\ge
		\frac{\delta^2\underline\Delta_\lambda}{N^3}$.
	Since $\underline\Delta_\lambda<2$, this is implied by the mild condition
	$\Delta_\lambda\gtrsim \delta^2/N^3$. Hence, for standard connected networks (see Remark \ref{remark_tr_ncvx}),
	the dominant P\L~transient time is given by \eqref{PL_transient_dominant}.
	
		For an unbiased compressor $\cQ$ with $\Ex\|\cQ(x)-x\|^2\le C\|x\|^2$, 
	the scaled operator $\cC(x)=\frac{1}{1+C}\cQ(x)$ is contractive with
	$\delta=\frac{1}{1+C}$. Substituting this relation into
	\eqref{PL_transient_dominant} gives the dominant transient time $
	K_{\mathrm{tr}}
	=
	\widetilde\cO(NC^2/\Delta_\lambda^2)$. This improves the
	$\cO(NC^3/\Delta_\lambda^2)$ transient-time dependence established for
	CEDAS~\cite{huang2023cedas}. 
	
\end{remark}
\begin{remark}[\sc\small Deterministic P\L~case]
	\rm
	The logarithmic stepsize selection in Corollary~\ref{corr:pl_rate} is stated
	for \(\sigma>0\). When \(\sigma=0\), the stochastic terms in
	Theorem~\ref{thm:pl_rate} vanish, and one may instead take any fixed admissible
	\(\tau\in(0,\bar\tau]\), for example \(\tau=\bar\tau\), to obtain
	\begin{align}
		\tilde f(\bar\phi_e^K)
	\le
	\left(
	\tilde f(\bar\phi_e^0)+\widetilde a_2\bar\tau^2
	\right)
	\exp\!\left(-\frac{\mu\bar\tau K}{2}\right).
	\end{align}
	
\end{remark}

		\begin{table}[t]
	\centering
	\caption{Comparison of the dependence of the transient time on the number of
		agents, network connectivity, and compression quality. Here
		\(\Delta_\lambda=1-\lambda\) is the spectral gap defined in
		\eqref{lambda_Deltalam}. Fixed problem-dependent quantities such as smoothness,
		condition-number, variance, and initialization factors, as well as logarithmic
		factors where applicable, are suppressed.}
	\label{tab:comparison_transient}
	\renewcommand{\arraystretch}{1.25}
	\resizebox{0.95\textwidth}{!}{%
		\begin{tabular}{c c c c}
			\toprule
			\textbf{Method}
			& \multicolumn{2}{c}{\textbf{Transient time}}
			& \textbf{Comments} \\
			\cmidrule(lr){2-3}
			& \textbf{nonconvex} & \textbf{PL/SC} & \\
			\midrule
			\makecell{Choco-SGD\\\cite{koloskova2019decentralized,koloskova2020Decentralized}}
			& $\dfrac{N^3G^4}{\Delta_{\lambda}^8\delta^4}$
			& SC: $\dfrac{NG^2}{\delta^2\Delta_\lambda^4}$
			& \makecell{Bounded gradients\\ $\mathbb{E}\left[\|\nabla F_i(x,\xi)\|^2\right]\le G^2$} \\
			\midrule
			\makecell{DeepSqueeze\\\cite{tang2019deepsqueeze}}
			& $\dfrac{N^3\zeta^4}{\delta^6\Delta_\lambda^{12}}$
			& \xmark
			& \makecell{Bounded heterogeneity\\ $\frac{1}{N}\sum_i\|\nabla f_i(x)-\nabla f(x)\|^2\le\zeta^2$} \\
			\midrule
			\makecell{MoTEF\\\cite{islamov2024towards}}
			& $\dfrac{N^3}{\delta^4\Delta_\lambda^{10}}$
			& $\dfrac{N}{\delta^2\Delta_\lambda^5}$
			& \makecell{Communicating two compressed vectors} \\
			\midrule
			\makecell{CEDAS\\\cite{huang2023cedas}}
			& $\dfrac{N^3C^6}{\Delta_\lambda^4}$
			& SC: $\dfrac{NC^3}{\Delta_\lambda^2}$
			& \makecell{Stochastic unbiased compressor$^\dagger$ \\$\Ex\|\cQ(x)-x\|^2\le C\|x\|^2$} \\
			\midrule
			\rowcolor{gray!15}
			\makecell{CED-EF\\\textbf{[This work]}}
			& $\dfrac{N^3}{\delta^4\Delta_\lambda^4}$
			& $\dfrac{N}{\delta^2\Delta_\lambda^2}$
			& -- \\
			\bottomrule
		\end{tabular}%
	}
	\vspace{1mm}
	\parbox{0.95\textwidth}{\scriptsize
		$^\dagger$ If an unbiased compressor satisfies
		\(\mathbb E\|\mathcal Q(x)-x\|^2\le C\|x\|^2\), then the scaling
		\(\mathcal C=\mathcal Q/(1+C)\) gives
		\(\delta=1/(1+C)\), and hence \(\delta^{-1}=1+C\).
		Therefore, for \(C\ge1\), the corresponding CED-EF
		network/compression dependencies are
		\(N^3C^4/\Delta_\lambda^4\) and
		\(NC^2/\Delta_\lambda^2\) in the nonconvex and P\L~settings, respectively.
	}
\end{table}

\section{Numerical experiments}
\label{sec:experiments}

We evaluate {CED-EF} against three communication-compressed decentralized baselines: {CEDAS} \cite{huang2023cedas}, {BEER} \cite{zhao2022beer}, and {MoTEF} \cite{islamov2024towards}. We consider regularized least squares on synthetic data and logistic
regression with a nonconvex regularizer on real data.


\paragraph{Common setup}
For every experiment, \(N=20\) nodes are connected through either an
exponential graph or a line graph. We construct a symmetric doubly stochastic
mixing matrix using the Metropolis--Hastings rule; the corresponding spectral
gaps are approximately \(0.44\) and \(8.2\times10^{-3}\), respectively. All
methods use the same biased \(\mathrm{Top}\text{-}c\) compressor, which retains
the \(c\) largest-magnitude coordinates and satisfies $
\|\mathrm{Top}\text{-}c(v)-v\|^2
\le
\left(1-\frac{c}{n}\right)\|v\|^2$,
so that \(\delta=c/n\). Although the convergence analysis of CEDAS assumes
unbiased compression, we use the same biased compressor for all methods to
provide a controlled empirical comparison; CEDAS converges in all reported
experiments. In each stochastic trial, all methods use the same mini-batch sequence.
The primal and compression-memory variables are initialized at zero, while
the gradient-tracking and momentum variables in BEER and MoTEF are initialized
according to their respective algorithms at the zero primal iterate.

\subsection{Least squares problem}
\label{sec:exp-ls}
In the least squares setup, each node objective is
\begin{equation}
	f_i(x)
	=
	\tfrac{1}{2M}\|A_i x-b_i\|^2
	+
	\tfrac{\lambda^{\text{reg}}}{2}\|x\|^2,
	\label{eq:ls-objective}
\end{equation}
with $M=200$, $n=10$, and $\lambda^{\text{reg}}=10^{-3}$.  The matrices are generated randomly with controlled conditioning and data
heterogeneity. Let $\widetilde A_i$ denote the design matrix obtained before
the final scaling. A normalized ground-truth vector $x^\circ$ is drawn once,
and the responses are generated according to
$b_i=\widetilde A_i x^\circ+\varepsilon_i$, where the noise standard deviation
is $0.1(1+0.5|\tau_i|)$ with
$\tau_i=2(i-1)/(N-1)-1$. The design matrices used by the optimization
algorithms are then set to $A_i=1.5\,\widetilde A_i$.

\paragraph{Least squares results}
We use $\mathrm{Top}\text{-}2$ compression, corresponding to $\delta=0.2$, and mini-batches of size $B=25$. The methods are tuned to comparable steady-state error levels so that the
experiment primarily compares the rate at which this accuracy regime is
attained.  Communication cost is measured in compressed-vector transmissions, with one vector per iteration  charged to {CED-EF} and {CEDAS}, and two to {BEER} and {MoTEF}. The results  are averaged over 30 independent runs.

Figure~\ref{fig:ls-stoch} reports the stochastic results on the exponential
(left) and line (right) graphs. CED-EF and CEDAS converge faster than BEER and
MoTEF on both topologies, with CED-EF attaining the prescribed accuracy
slightly earlier than CEDAS. Figure~\ref{fig:ls-det} reports the corresponding
full-gradient experiment, where CED-EF reaches high accuracy substantially
sooner than the competing methods.

We also observe that the performance of the gradient-tracking-based methods,
BEER and MoTEF, deteriorates more markedly than that of the
exact diffusion-based methods, CED-EF and CEDAS, as the network becomes more
poorly connected. This empirical trend is consistent with the stronger
network dependence summarized in Table~\ref{tab:comparison_transient} and
with analogous comparisons between uncompressed exact diffusion and
gradient-tracking methods~\cite{alghunaim2021unified}.

 \begin{figure}[t]
	     \centering
	     \includegraphics[width=0.48\linewidth]
	     {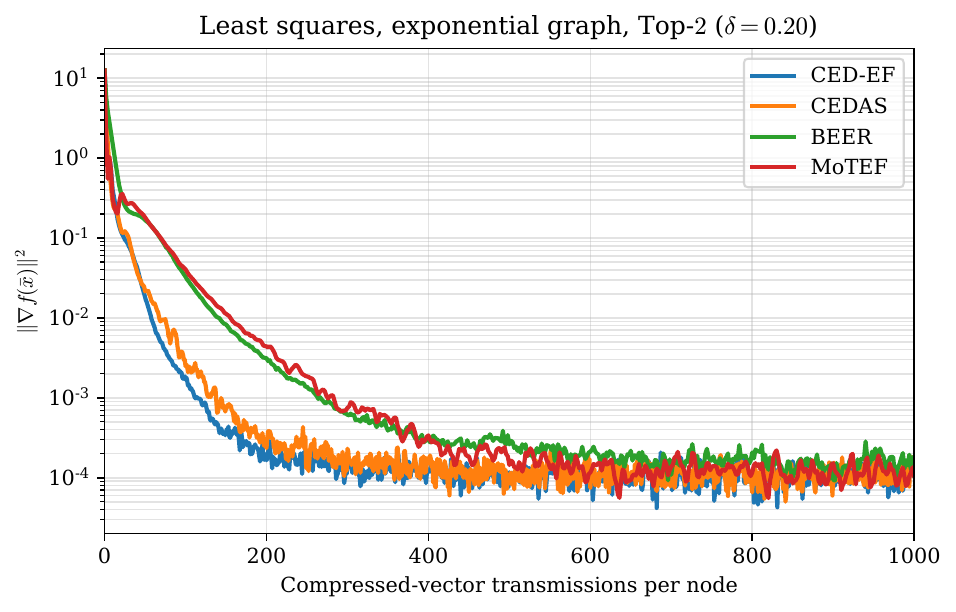}
	     \includegraphics[width=0.48\linewidth]
	     {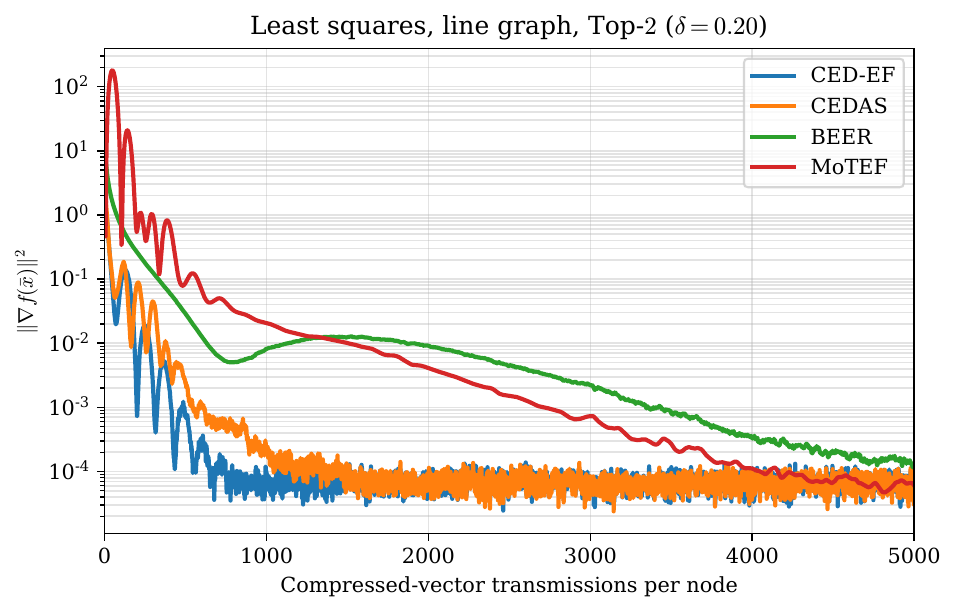}
	     \caption{Stochastic least squares with
		     $\mathrm{Top}\text{-}2$ compression on the exponential (left) and line
		     (right) graphs.}
	     \label{fig:ls-stoch}
	 \end{figure}

 \begin{figure}[t]
	     \centering
	     \includegraphics[width=0.48\linewidth]
	     {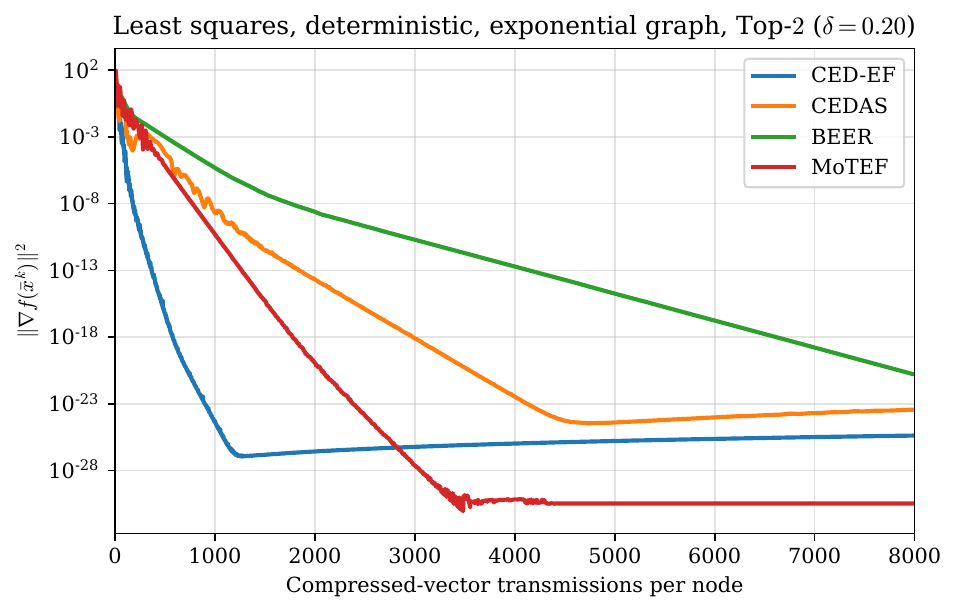}
	    \includegraphics[width=0.48\linewidth]
	     {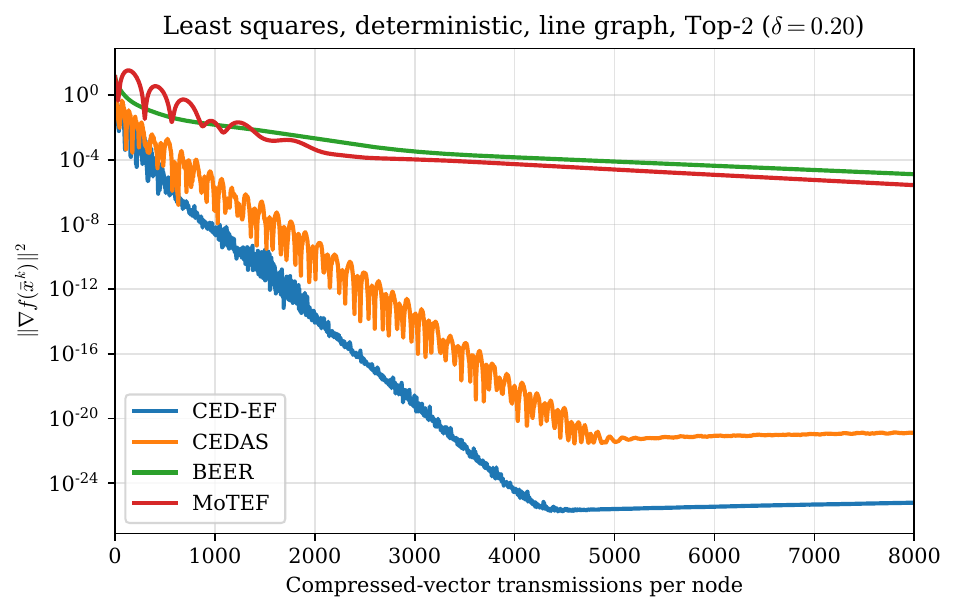}
	     \caption{Deterministic least squares with full gradients and
		     $\mathrm{Top}\text{-}2$ compression on the exponential (left) and line
		     (right) graphs.}
	     \label{fig:ls-det}
	 \end{figure}


\paragraph{Effect of the compression level}
To study the communication savings provided by compression, we simulate
CED-EF with $
c\in\{1,3,5,n\}$,
where \(c=n\) denotes uncompressed communication. Each compression level is
tuned independently, and the horizontal axis reports transmitted bits per
node. A \(\mathrm{Top}\text{-}c\) message is charged
\(c(16+\lceil\log_2n\rceil)\) bits, whereas an uncompressed message requires
\(16n\) bits. Figure~\ref{fig:levels-ls} reports the results. The compressed
variants substantially reduce the communication required per iteration and
can therefore require fewer communicated bits to attain a prescribed
accuracy, particularly on the well-connected exponential graph.

 \begin{figure}[t]
	     \centering
	     \includegraphics[width=0.48\linewidth]
	     {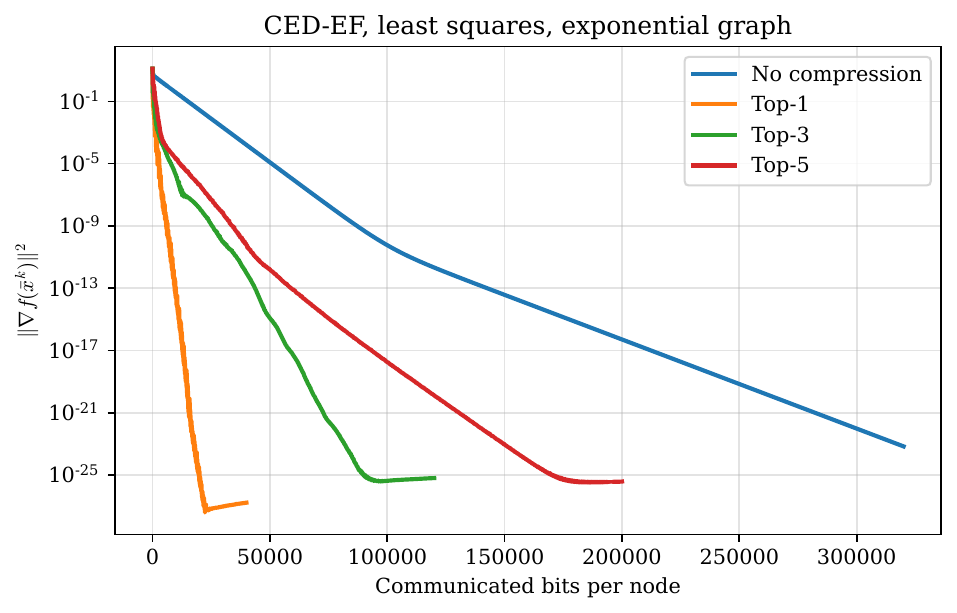}
	     \includegraphics[width=0.48\linewidth]
	     {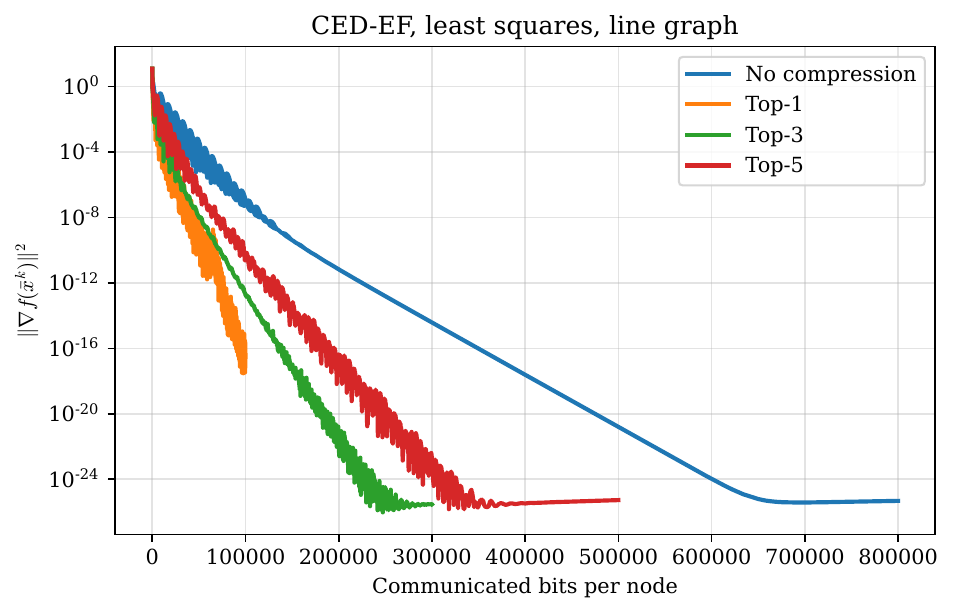}
	     \caption{{CED-EF} with no compression and
		     $\mathrm{Top}\text{-}k$, $k\in\{1,3,5\}$, on the least squares problem.}
	     \label{fig:levels-ls}
	 \end{figure}

\subsection{Logistic-regression problem}
\label{sec:exp-logreg}

We consider logistic regression with the saturating nonconvex regularizer
\begin{equation}
	f_i(x)
	=
	\frac{1}{M}
	\sum_{j=1}^{M}
	\log\!\left(1+\exp(-b_{ij}a_{ij}^{\top}x)\right)
	+
	\lambda^{\mathrm{reg}}
	\sum_{\ell=2}^{n}
	\frac{x_\ell^2}{1+x_\ell^2},
	\label{eq:logistic-objective}
\end{equation}
where \(a_{ij}\in\mathbb R^n\) is the augmented feature vector, whose first
coordinate is the intercept, and \(b_{ij}\in\{-1,+1\}\) is the corresponding
label. The intercept is not regularized.

For the real-data experiment, we use the LIBSVM \texttt{a9a} binary
classification dataset. The features are standardized using their global
empirical means and standard deviations, after which an intercept is appended.
The samples are randomly permuted using a fixed seed and partitioned evenly
and disjointly among the \(N=20\) nodes. We retain \(20{,}000\) examples, so
each node receives \(M=1000\) samples. This instance is run in the stochastic
regime with Top-$5$ compression, mini-batch size $B=64$, and regularization
$\lambda^{\text{reg}}=10^{-3}$. The results are averaged over 5 independent runs.


\paragraph{Logistic regression results}
Figure~\ref{fig:logreg-a9a} reports stochastic logistic regression on the
\texttt{a9a} dataset using \(c=5\) and \(B=64\). As in the least-squares experiments, CED-EF and CEDAS attain a given stationarity level using fewer compressed-vector transmissions than BEER and
MoTEF.

\begin{figure}[t]
	\centering
	 \includegraphics[width=0.48\linewidth]{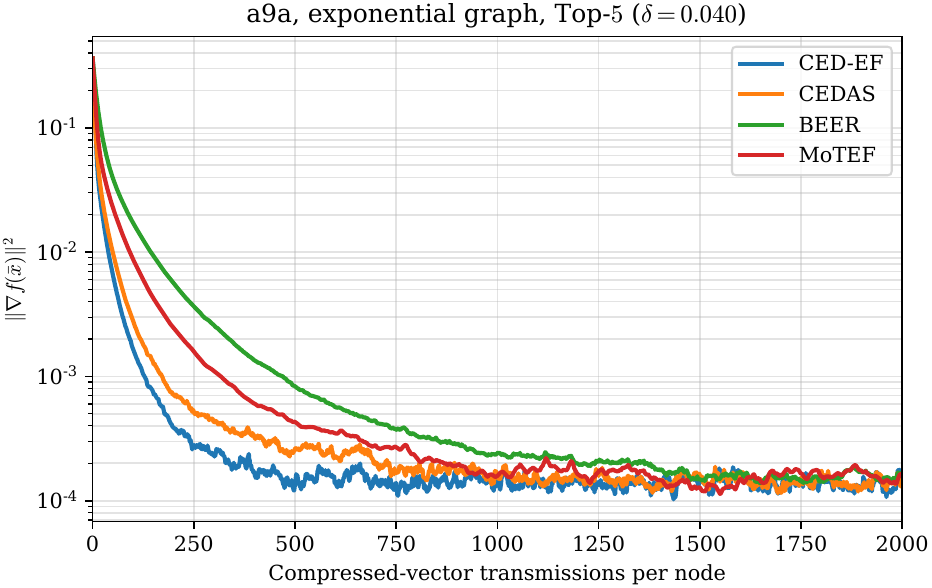}
	 \includegraphics[width=0.48\linewidth]{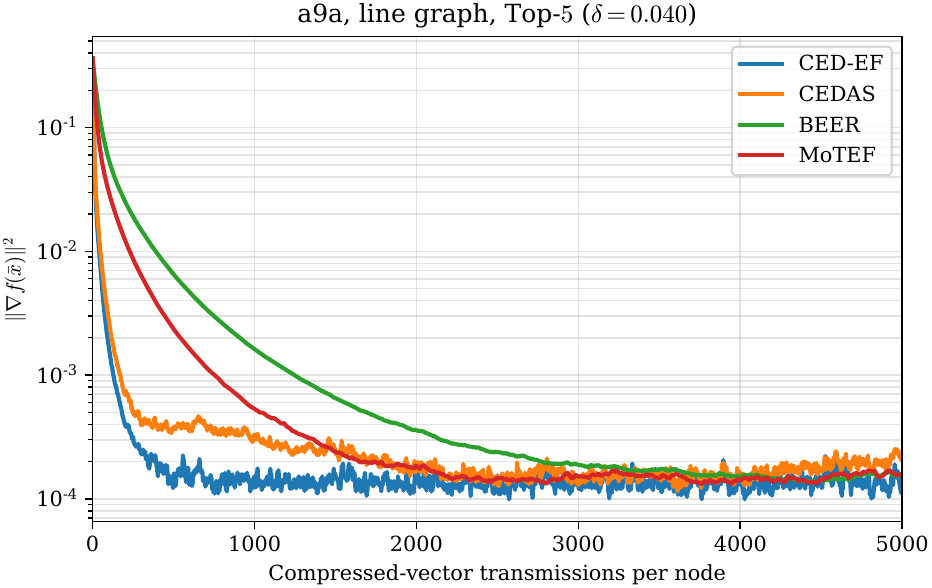}
	\caption{Stochastic logistic regression on the LIBSVM \texttt{a9a} data set on the exponential (left) and line (right)
		graphs.}
	\label{fig:logreg-a9a}
\end{figure}


		\newpage
		\appendices
		\section{Convergence analysis preliminaries}
		We first rewrite CED-EF in compact network form and introduce the variables
		used in the convergence analysis. Lemma~\ref{lemma_cent_evo} gives the
		centroid recursion, while Lemma~\ref{lemma_avg_deviation} gives the
		corresponding disagreement dynamics.
		
		\subsection{CED-EF in compact form}
		For analysis purposes, we express CED-EF given in Algorithm \ref{alg:ced-ef} in compact network form. To do that, we define the following quantities:
		\begin{subequations} \label{defs}
			\begin{align}
				\W&\define W \otimes I_n, \quad \W_\gamma
				\define
				\bigl((1-\gamma)I_N+\gamma W\bigr)\otimes I_n
				=
				(1-\gamma)\I+\gamma\W. \label{W_network}\\
				\x^k&\define\col\{x_i^k\}_{i=1}^N, \quad \y^k\define \col\{y_i^k\}_{i=1}^N \\
				\bm{\upphi}^k &\define\col\{\phi_i^k\}_{i=1}^N, \quad \e^k\define \col\{e_i^k\}_{i=1}^N \\
				\f(\x)&\define \sum_{i=1}^N f_i(x_i), ~~ \F(\x;\bxi)\define \sum_{i=1}^N F_i(x_i;\xi_i) \\
				\grad \f(\x)&\define \col\{\grad f_i(x_i)\}_{i=1}^N, ~~ \grad \F(\x;\bxi)\define \col\{\grad F_i(x_i;\xi_i)\}_{i=1}^N \\
				\bm{\cC}(\x)&\define\col\{\cC_i(x_i)\}_{i=1}^N.
			\end{align}
		\end{subequations}
		Using this notation, Algorithm \ref{alg:ced-ef} updates \eqref{ced-ef} can be expressed in network compact form as follows:
			\begin{subequations}  \label{ced_ef_network}
			\begin{align}
				\hat{\x}^{k} &= \x^{k} - \alpha \nabla \F(\x^{k};\bxi^{k})-\y^{k}  \label{hat_x_ced_ef_net} \\
				\hat{\bm{\upphi}}^{k}&=\bm{\upphi}^{k}+ \theta\bm{\cC}(\hat{\x}^{k}-\bm{\upphi}^{k}+\eta \e^{k})\label{phi_hat_x_ced_ef_net} \\
				\bm{\upphi}^{k+1} &= \bm{\upphi}^{k} +  \beta \bm{\cC}(\hat{\x}^{k} - \bm{\upphi}^{k}+ \eta \e^{k}) \label{phi_x_ced_ef_net}\\
				\y^{k+1} &= \y^{k}+   \gamma (\I-\W) \hat{\bm{\upphi}}^{k} \label{y_ced_net} \\
				\x^{k+1} &= \W_{\gamma}\bm{\upphi}^{k+1} \label{x_ced_ef_net} \\
			\e^{k+1}&=\hat{\x}^k-\bm{\upphi}^{k}+\e^k-  \bm{\cC}(\hat{\x}^{k} - \bm{\upphi}^{k}+ \eta \e^{k}) \label{e_ced_ef_net},
			\end{align}
		\end{subequations}
		with $\x^{0} = \W_{\gamma}\bm{\upphi}^{0}$, $\y^{0}=0$, and $\e^{0}=0$. Here, $\bm{\uppsi}^k\define\col\{\psi_i^k\}_{i=1}^N$ is eliminated using the invariant $\bm{\uppsi}^k=\W\bm{\upphi}^k$, which holds for all $k$ by initialization and the updates. Indeed, with $\bm q^k\define\bm{\cC}(\hat{\x}^{k}-\bm{\upphi}^{k}+\eta\e^{k})$, we have $\bm{\uppsi}^{k+1}=\bm{\uppsi}^{k}+\beta\W\bm q^k=\W(\bm{\upphi}^{k}+\beta\bm q^k)=\W\bm{\upphi}^{k+1}$ and, similarly, $\hat{\bm{\uppsi}}^k=\W\hat{\bm{\upphi}}^k$. Hence, \eqref{x_ced} and \eqref{y_ced} reduce to \eqref{x_ced_ef_net} and \eqref{y_ced_net}, respectively.
		
		\subsection{Centroid and average deviation}
		Here, we transform the recursion \eqref{ced_ef_network} into a form suitable for our analysis. To do that, we define the following useful quantities:
		\begin{subequations} \label{def_phi_e}
			\begin{align}
				\bm{\upphi}_{e}^{k}&\define \bm{\upphi}^{k}+\beta \e^{k} \label{phi_e_def}\\
				\bar{\bm{\upphi}}_{e}^{k}&\define \one \otimes \bar{\phi}_e^{k}, \quad \bar{\phi}_e^{k}\define(1/N)\sum_{i=1}^N \phi_{e,i}^{k} \label{phi_e_bar_def}.
			\end{align}
		\end{subequations}
		We next show that $\bar{\phi}_e^{k}$ evolves as a perturbed stochastic gradient descent on the average cost $f$.
		\begin{lemma}[\sc \small Centroid evolution] \label{lemma_cent_evo}
			\rm  The centroid $\bar{\phi}_{e}^{k}$ defined in \eqref{def_phi_e} evolves as follows:
			\begin{align}
						\bar{\phi}_{e}^{k+1}
						&= \bar{\phi}_{e}^{k}  - \frac{\alpha \beta}{N}\sum_{i=1}^N\nabla F_{i}(x_{i}^{k};\xi_{i}^{k}). \label{centroid_phi_e}
					\end{align}
		\end{lemma}
	\begin{proof}
		Multiplying \eqref{e_ced_ef_net} by $\beta$  and adding to \eqref{phi_x_ced_ef_net}, leads to
		\begin{align} \label{phi_e_recursion}
			\bm{\upphi}_{e}^{k+1}&=\bm{\upphi}^{k}+\beta\hat{\x}^{k} - \beta\bm{\upphi}^{k} +  \beta \e^{k} \nonumber \\
			&\overset{\eqref{hat_x_ced_ef_net}}{=}(1-\beta)\bm{\upphi}^{k} +\beta\x^{k} - \alpha \beta \nabla \F(\x^{k};\bxi^{k})-\beta\y^{k} +  \beta \e^{k} \nonumber \\
			&\overset{\eqref{x_ced_ef_net}}{=}\W_{\beta,\gamma}\bm{\upphi}^{k} - \alpha \beta \nabla \F(\x^{k};\bxi^{k})-\beta\y^{k} +  \beta \e^{k}  \nonumber \\
			&\overset{\eqref{phi_e_def}}{=}\W_{\beta,\gamma}\bm{\upphi}_{e}^{k} - \alpha \beta \nabla \F(\x^{k};\bxi^{k})-\beta\y^{k} +  \beta (\I-\W_{\beta,\gamma}) \e^{k},
		\end{align}
		where $\W_{\beta,\gamma}=(1-\beta\gamma)\I+\beta\gamma \W$. The result follows by taking the average and using $(\one^{\top}\otimes I) (\I-\W_{\beta,\gamma})=0$ and the fact that $(\one^{\top}\otimes I) \y^{k}=0$ under $\y^{0}=0$ since $\y^{k}$ will always be in the range of $\I-\W$.
	\end{proof}
		Under  Assumption \ref{assump:network}, the weight matrix $W$ admits the symmetric eigendecomposition \cite{horn2012matrix,sayed2014nowbook}:
	\begin{align*}
		W= \begin{bmatrix}
			\frac{1}{\sqrt{N}}  \one & Q
		\end{bmatrix} \begin{bmatrix}
			1 & 0 \\
			0 & \Lambda
		\end{bmatrix} \begin{bmatrix}
			\frac{1}{\sqrt{N}} \one^{\top} \vspace{0.6mm} \\  Q^{\top}
		\end{bmatrix},
	\end{align*}
	where $\Lambda=\diag\{\lambda_{i}\}_{i=2}^N$ is a diagonal matrix containing the eigenvalues with magnitude less than one. The matrix $Q\in \real^{N \times (N-1)}$  satisfies $Q^{\top}Q=I_{N-1}$, $QQ^{\top}=I_N-\tfrac{1}{N} \one \one^{\top}$,  and  $\one^{\top} Q=0$. It follows that $\W$ defined in \eqref{W_network} can be decomposed as
	\begin{subequations}  \label{W_QQtran}
		\begin{align}
			\W&=  \begin{bmatrix}
				\frac{1}{\sqrt{N}}   \one \otimes I_n & \Q
			\end{bmatrix} \begin{bmatrix}
				I_n & 0 \\
				0 & \mathbf{\Lambda}
			\end{bmatrix} \begin{bmatrix}
				\frac{1}{\sqrt{N}}  \one^{\top} \otimes I_n \\ \Q^{\top}
			\end{bmatrix}=\tfrac{1}{N} \one \one^{\top} \otimes I_n+\Q \mathbf{\Lambda} \Q^{\top},
		\end{align}
		where $\mathbf{\Lambda} \define \Lambda \otimes I_n \in \real^{n(N-1)\times n(N-1)}$ and $\Q\define Q \otimes I_n \in \real^{n N \times n(N-1)}$ satisfies:
		\begin{align}
			\Q^{\top} \Q=\I, \quad \Q\Q^{\top}=\I-\tfrac{1}{N} \one \one^{\top} \otimes I_n, \quad (\one^{\top} \otimes I_n) \Q=0.
		\end{align}
	\end{subequations}
	Note  that $\|\Q\|=1$ and for any vector $\u=\col\{u_1,\dots,u_N\}$ with $u_i\in \real^n$,
	\begin{align} \label{Q_matrix_bound}
		\|\Q^{\top} \u \|^2&= \|\Q \Q^{\top} \u \|^2= \|\u - \bar{\u}\|^2 
	\end{align}
	where $\bar{\u}=\one \otimes \bar{u}$ and $\bar{u}=(1/N)\sum_{i=1}^N u_{i}$. That is, $\|\Q^{\top} \u \|^2$ measures the deviation of $\u$ from the average. We next give an important relation regarding the deviation from the average for $\bm{\upphi}_e^k$ and another relevant quantity.
	\begin{lemma}[\sc \small Deviation from average] \label{lemma_avg_deviation}
		\rm Using the network form of CED-EF given by \eqref{ced_ef_network}, it holds that
		\begin{align} \label{phi_e_and_z_trans_update}
					\begin{bmatrix}
					\Q^{\top}\bm{\upphi}_e^{k+1}\\[2mm]
					\Q^{\top}\widehat{\z}^{k+1}
			\end{bmatrix}
			=\D
			\begin{bmatrix}
				\Q^{\top}	\bm{\upphi}_e^k\\[2mm]
					\Q^{\top}\widehat{\z}^k
			\end{bmatrix}
			-\alpha
			\begin{bmatrix}
				\beta 	\Q^{\top} \\[1mm]
				\theta \mathbf{L}_{\gamma}	\Q^{\top}
			\end{bmatrix}
			\g_1^k
			+
			\begin{bmatrix}
				\mathbf 0 \\[1mm]
				\alpha 	\Q^{\top}
			\end{bmatrix}
			\g_2^{k+1}
			+
			\begin{bmatrix}
				\beta(\beta+\theta)\mathbf{L}_{\gamma}	\Q^{\top} \\[1mm]
				(\theta^2\mathbf{L}_{\gamma}^2-\beta \mathbf{L}_{\gamma} \mathbf{\Lambda}_{\theta,\gamma})	\Q^{\top}
			\end{bmatrix}
			\e^k,
		\end{align}
		where
		\begin{subequations} \label{def_cons_parameters}
			\begin{align}
				\D &\define 	 	\begin{bmatrix}
					\mathbf{\Lambda}_{\beta,\gamma} & -\beta \I
								\\[1mm]
					\mathbf{L}_{\gamma}\mathbf{\Lambda}_{\theta,\gamma} & \mathbf{\Lambda}_{\theta,\gamma}
				\end{bmatrix}, 	\quad 
				\mathbf{L}_{\gamma}\define(\I-\mathbf{\Lambda}_{\gamma})=\gamma(\I-\mathbf{\Lambda}) 
					\label{matrix_cons} \\
				\g_1^k &\define
				\nabla \F(\x^k;\bxi^k)-\grad \f(\bar{\bm{\upphi}}_e^k), \quad 
				\g_2^{k+1}
				\define
				\grad \f(\bar{\bm{\upphi}}_e^{k+1})
				-
				\grad \f(\bar{\bm{\upphi}}_e^k) 	\label{g12}\\
				\z^{k}&\define{\y}^{k}+\alpha\grad \f (\bar{\bm{\upphi}}_{e}^{k}), \quad
				\widehat{\z}^k\define\z^{k}+(\I-\W_{\theta,\gamma})\e^k, 	\label{z_zhat_def}
			\end{align}
		\end{subequations}
	with $
	\mathbf{\Lambda}_{s,\gamma}
	\define
	\I-s(\I-\mathbf{\Lambda}_\gamma)
	=
	(1-s\gamma)\I+s\gamma\mathbf{\Lambda}$  for \(s\in\{\beta,\theta\}\) and	$\W_{\theta,\gamma}=\I-\theta\gamma (\I-\W)=(1-\theta \gamma)\I+\theta \gamma \W$.
	\end{lemma}
		
	\begin{proof}
		First note that for any $s$, we have
		\begin{align}
			\W_{s}\define(1-s)\I+s \W=\I-s (\I-\W) \quad \iff \quad \I-\W_{s}=s(\I-\W).
		\end{align}
			Hence, from \eqref{phi_e_recursion}, we have
			\begin{align} \label{phiphi}
			\bm{\upphi}_{e}^{k+1}&=\W_{\beta,\gamma}\bm{\upphi}_{e}^{k} - \alpha \beta \nabla \F(\x^{k};\bxi^{k})-\beta\y^{k} +  \beta^2 (\I-\W_{\gamma}) \e^{k},
		\end{align}
			where $\W_{\beta,\gamma}=(1-\beta\gamma)\I+\beta\gamma \W$.
			Following a similar argument on \eqref{phi_hat_x_ced_ef_net}, we also have
		\begin{align}
			\hat{\bm{\upphi}}^{k}&=\bm{\upphi}^{k}+ \theta\bm{\cC}(\hat{\x}^{k}-\bm{\upphi}^{k}+\eta \e^{k}) \nonumber \\
			&\overset{\eqref{e_ced_ef_net}}{=}(1-\theta)\bm{\upphi}^{k}+\theta \hat{\x}^{k}  + \theta (\e^{k}-\e^{k+1})
			\nonumber \\
			&\overset{\eqref{hat_x_ced_ef_net}}{=}\W_{\theta,\gamma}\bm{\upphi}^{k}- \alpha \theta \nabla \F(\x^{k};\bxi^{k})-\theta\y^{k}  + \theta (\e^{k}-\e^{k+1}).
		\end{align}
		Substituting into \eqref{y_ced_net} gives
			\begin{align} \label{yy}
				\y^{k+1} &= \y^{k}+   (\I-\W_\gamma) \left(\W_{\theta,\gamma}\bm{\upphi}^{k}- \alpha \theta \nabla \F(\x^{k};\bxi^{k})-\theta\y^{k}  + \theta (\e^{k}-\e^{k+1})\right)  \nonumber \\
				&=\W_{\theta,\gamma}\y^{k} +(\I-\W_\gamma) \W_{\theta,\gamma}\bm{\upphi}^{k} - \alpha  (\I-\W_{\theta,\gamma}) \nabla \F(\x^{k};\bxi^{k})+  (\I-\W_{\theta,\gamma})(\e^{k}-\e^{k+1}),
			\end{align}
				where $\W_{\theta,\gamma}=(1-\theta\gamma)\I+\theta\gamma \W$. 	Using the change of variable $	\z^{k}={\y}^{k}+\alpha\grad \f (\bar{\bm{\upphi}}_{e}^{k})$ and $\bm{\upphi}_{e}^{k}= \bm{\upphi}^{k}+\beta \e^{k}$, we can rewrite \eqref{phiphi} and \eqref{yy} as follows
			\begin{subequations} \label{phi_e_z_update}
		\begin{align}
				\bm{\upphi}_e^{k+1}&=\W_{\beta,\gamma}\bm{\upphi}_e^{k}- \alpha \beta (\nabla \F(\x^{k};\bxi^{k})-\grad \f (\bar{\bm{\upphi}}_{e}^{k}))-\beta\z^{k}  +\beta^2 (\I-\W_{\gamma}) \e^{k} \label{phi_e_update} \\
				\z^{k+1}			&=\W_{\theta,\gamma}\z^{k} +(\I-\W_\gamma) \W_{\theta,\gamma}\bm{\upphi}_e^{k} - \alpha  (\I-\W_{\theta,\gamma}) (\nabla \F(\x^{k};\bxi^{k})-\grad \f (\bar{\bm{\upphi}}_{e}^{k})) \nonumber \\
				& \quad-\beta(\I-\W_\gamma) \W_{\theta,\gamma}\e^k +  (\I-\W_{\theta,\gamma})(\e^{k}-\e^{k+1})+\alpha(\grad \f (\bar{\bm{\upphi}}_{e}^{k+1})-\grad \f (\bar{\bm{\upphi}}_{e}^{k})). \label{y1}
		\end{align}
	\end{subequations}
	Using another change of variable $\widehat{\z}^k=\z^{k}+(\I-\W_{\theta,\gamma})\e^k$ and the definitions in \eqref{g12}, gives
	\begin{subequations} \label{phi_e_hat_z_update_compact}
		\begin{align}
			\bm{\upphi}_e^{k+1}
			&=
			\W_{\beta,\gamma}\bm{\upphi}_e^k
			-\alpha\beta \g_1^k
			-\beta \widehat{\z}^k
			+\beta(\beta+\theta)(\I-\W_\gamma)\e^k ,
			\\
			\widehat{\z}^{k+1}
			&=
			\W_{\theta,\gamma}\widehat{\z}^k
			+(\I-\W_\gamma)\W_{\theta,\gamma}\bm{\upphi}_e^k
			-\alpha\theta(\I-\W_\gamma)\g_1^k
			+\alpha \g_2^{k+1}
			\nonumber\\
			&\quad
			+\Bigl[
			(\I-\W_{\theta,\gamma})^2
			-\beta(\I-\W_\gamma)\W_{\theta,\gamma}
			\Bigr]\e^k .
		\end{align}
	\end{subequations}
		Using the decomposition of $\W$ given in \eqref{W_QQtran}, it holds that $\Q^{\top} \W=\mathbf{\Lambda} \Q^{\top}$.
		Therefore, using
		\(\I-\W_{\theta,\gamma}=\theta(\I-\W_\gamma)\) and multiplying the upper and
		lower blocks of \eqref{phi_e_hat_z_update_compact} by \(\Q^\top\) yields the
		claimed recursion.
	\end{proof}

		\section{Intermediate lemmas}
			We exploit the block structure of $\D$ in \eqref{matrix_cons} to diagonalize
			the disagreement recursion \eqref{phi_e_and_z_trans_update}. The resulting
			coordinates are then used to bound the centroid, disagreement, and
			compression-error dynamics.
		\subsection{Fundamental transformation}
			\begin{lemma}[\sc \small Fundamental transformation]\label{lemma:transf_avg_companion}
				\rm Suppose Assumption \ref{assump:network} holds and the parameters satisfy the  condition
					\begin{align}\label{rot_cond_comp}
					\gamma(1-\underline{\lambda})<\frac{4\beta}{(\beta+\theta)^2}
					\quad \iff \quad \nu\define 1-\frac{(\beta+\theta)^2\gamma(1-\underline{\lambda})}{4\beta}>0,
				\end{align}
				where $\underline{\lambda}\define\min_{i\ge2}\lambda_i$ is the smallest eigenvalue of $W$.
				Then the matrix \eqref{matrix_cons} admits the decomposition $\D=\V\mathbf{\Delta}\V^{-1}$ for some $\V$ and a matrix $\mathbf{\Delta}$ with
				\begin{align}
					\lambda_{\theta}\define\|\mathbf{\Delta}\|
					=\sqrt{1-\theta\gamma(1-\lambda)}<1,\quad \lambda\define\max_{i\ge2}\lambda_i.
				\end{align}
				Therefore, multiplying \eqref{phi_e_and_z_trans_update} by $\V^{-1}$ on the left gives:
					\begin{align} \label{phi_e_and_z_trans_diag_update}
					\d^{k+1}
					=\mathbf{\Delta}
					\d^{k}
					-\alpha
					\V^{-1} \begin{bmatrix}
						\beta 	\Q^{\top} \\[1mm]
						\theta \mathbf{L}_{\gamma}	\Q^{\top}
					\end{bmatrix}
					\g_1^k
					+
					\V^{-1}\begin{bmatrix}
						\mathbf 0 \\[1mm]
						\alpha 	\Q^{\top}
					\end{bmatrix}
					\g_2^{k+1}
					+
					\V^{-1}\begin{bmatrix}
						\beta(\beta+\theta)\mathbf{L}_{\gamma}	\Q^{\top} \\[1mm]
						(\theta^2\mathbf{L}_{\gamma}^2-\beta \mathbf{L}_{\gamma} \mathbf{\Lambda}_{\theta,\gamma})	\Q^{\top}
					\end{bmatrix}
					\e^k,
				\end{align}
				where 
				\begin{align}
				\d^{k}\define\V^{-1}	\begin{bmatrix}
						\Q^{\top}	\bm{\upphi}_e^k\\[2mm]
						\Q^{\top}\widehat{\z}^k
					\end{bmatrix}
				\end{align}
				Moreover, partitioning $\V^{-1}=[\V_l^{-1}~~\V_r^{-1}]$,
				\begin{subequations} \label{V_left_right_bound}
					\begin{align}
						\|\V_l^{-1}\|^2
						\le\frac{\lambda_{\theta}^2}{2\beta^2\nu}\define v_l^2,
						\quad
						\|\V_r^{-1}\|^2
						\le\frac{1}{2\beta\gamma(1-\lambda)\nu}\define v_r^2,
					\end{align}
					\begin{align}
						\|\V_l^{-1}(\I-\mathbf{\Lambda})\|^2
						\le\frac{\lambda_{\theta}^2(1-\underline{\lambda})^2}{2\beta^2\nu}\define\bar v_l^2,
						\quad
						\|\V_r^{-1}(\I-\mathbf{\Lambda})\|^2
						\le\frac{1-\underline{\lambda}}{2\beta\gamma\nu}\define\bar v_r^2.
					\end{align}
				\end{subequations}
			\end{lemma}
		
			\begin{proof}
				The structure of $\D$ given by \eqref{matrix_cons} implies that there is a permutation matrix $\P$ such that
				\begin{align}
					\P\D\P^{\top}=\bdiag\{D_i\}_{i=2}^N\otimes I_n,\quad
					D_i=\begin{bmatrix} \lambda_{\beta,i} & -\beta\\[1mm]
						\gamma(1-\lambda_i) \lambda_{\theta,i} & \lambda_{\theta,i}\end{bmatrix},
				\end{align}
				where $\lambda_{\beta,i}\define(1-\beta \gamma)+\beta \gamma \lambda_i$ and $\lambda_{\theta,i}\define(1-\theta \gamma)+\theta \gamma \lambda_i$. The trace and determinant of $D_i$ are $\Tr(D_i)=\lambda_{\beta,i}+\lambda_{\theta,i}$ and
				\begin{align}
					\det(D_i)=\lambda_{\beta,i}\lambda_{\theta,i}+\beta\gamma(1-\lambda_i)\lambda_{\theta,i}
					=\lambda_{\theta,i}.
				\end{align}
				With $\bar{\lambda}_{\beta,\theta,i}\define\tfrac12(\lambda_{\beta,i}+\lambda_{\theta,i})$, the two eigenvalues of $D_i$ are
				$\rho_{(1,2),i}=\bar{\lambda}_{\beta,\theta,i}\pm\tfrac12\sqrt{\Tr(D_i)^2-4\det(D_i)}$. Moreover, with $\ell_{\gamma,i}\define\gamma(1-\lambda_i)>0$, it holds under condition \eqref{rot_cond_comp} that
				\begin{align}
					\Tr(D_i)^2-4\det(D_i)=\ell_{\gamma,i}\big[(\beta+\theta)^2\ell_{\gamma,i}-4\beta\big]<0
				\end{align}
				for all $i\ge2$. Hence the eigenvalues are complex conjugates
				$\rho_{(1,2),i}=\bar{\lambda}_{\beta,\theta,i}\pm\mathrm{j}\omega_i$ with 
					\begin{align}\label{omega_b_def}
				\omega_i^2\define \det(D_i)-\Big(\tfrac{\Tr(D_i)}{2}\Big)^2
					=\beta\gamma(1-\lambda_i) \nu_i\;>\;0, \quad 	\nu_i\define 1-\frac{(\beta+\theta)^2\gamma(1-\lambda_i)}{4\beta} .
				\end{align}
				Moreover, the magnitudes of the eigenvalues are
				\begin{align}
					|\rho_{(1,2),i}|=\sqrt{\bar{\lambda}_{\beta,\theta,i}^2+\omega_i^2}=\sqrt{\det(D_i)}=\sqrt{\lambda_{\theta,i}}
					=\sqrt{1-\theta\gamma(1-\lambda_i)}.
				\end{align}
				Since $(\beta+\theta)^2\ge4\beta\theta$, condition \eqref{rot_cond_comp} also gives $\theta\gamma(1-\underline{\lambda})<1$,
				so $\lambda_{\theta,i}\in(0,1)$ and $\lambda_{\theta}=\max_{i\ge2}\sqrt{\lambda_{\theta,i}}=\sqrt{1-\theta\gamma(1-\lambda)}<1$.
				
				Since the eigenvalues are distinct, each $D_i$ is diagonalizable with $D_i=V_i\Delta_i V_i^{-1}$ where $\Delta_i=\diag\{\rho_{1,i},\rho_{2,i}\}$ and defining
				$p_i\define\bar{\lambda}_{\beta,\theta,i}-\lambda_{\beta,i}=\tfrac12(\beta-\theta)\gamma(1-\lambda_i)$, we have
				\begin{align}
					V_i=\begin{bmatrix} -\beta & -\beta\\[1mm] p_i+\mathrm{j}\omega_i & p_i-\mathrm{j}\omega_i\end{bmatrix},
					\quad
					V_i^{-1}
					=\begin{bmatrix}
						-\dfrac{1}{2\beta}-\dfrac{\mathrm{j}p_i}{2\beta\omega_i} & -\dfrac{\mathrm{j}}{2\omega_i}\\[4mm]
						-\dfrac{1}{2\beta}+\dfrac{\mathrm{j}p_i}{2\beta\omega_i} & \dfrac{\mathrm{j}}{2\omega_i}
					\end{bmatrix}.
				\end{align}
					This follows since for any $2 \times 2$ matrix $D=[d_{ij}]$ with distinct eigenvalues $\rho_1,\rho_2$, the matrix with eigenvectors as columns is
				\begin{align} \label{eigenvectors_diag_formula}
					V=\begin{bmatrix}
						d_{12}  & d_{12}  \\
						\rho_1-d_{11} & \rho_2-d_{11}
					\end{bmatrix}.
				\end{align}
				To verify, one can check $ V_i^{-1} V_i = I_2 $ and $ D_i = V_i \Delta_i V_i^{-1} $ by direct (symbolic) multiplication. Thus, $\D=\V\mathbf{\Delta}\V^{-1}=\P^{\top}\hat{\V}\mathbf{\Delta}\hat{\V}^{-1}\P$ with
				$\hat{\V}=\bdiag\{V_i\}\otimes I_n$, $\mathbf{\Delta}=\bdiag\{\Delta_i\}\otimes I_n$ ($\V=\P^{\top}\hat{\V}$ and
				$\V^{-1}=\hat{\V}^{-1}\P$). 
				
				Since $\V^{-1}=\hat{\V}^{-1}\P$ and $\P$ routes the
				first column of each block to the $\bm{\upphi}_e$-half and the second to the $\widehat{\z}$-half,
				the partition $\V^{-1}=[\,\V_l^{-1}\ \ \V_r^{-1}\,]$ is
				\begin{align}
					\V_l^{-1}
					&=\bdiag\left\{
					\begin{bmatrix}
						-\dfrac{1}{2\beta}-\dfrac{\mathrm{j}p_i}{2\beta\omega_i}\\[4mm]
						-\dfrac{1}{2\beta}+\dfrac{\mathrm{j}p_i}{2\beta\omega_i}
					\end{bmatrix}
					\right\}_{i=2}^{N}\otimes I_n,
					\quad
					\V_r^{-1}
					=\bdiag\left\{
					\begin{bmatrix}
						-\dfrac{\mathrm{j}}{2\omega_i}\\[4mm]
						\phantom{-}\dfrac{\mathrm{j}}{2\omega_i}
					\end{bmatrix}
					\right\}_{i=2}^{N}\otimes I_n .
				\end{align}
			 Using the identity (which follows from $\bar{\lambda}_{\beta,\theta,i}^2+\omega_i^2=\lambda_{\theta,i}$)
				\begin{align}\label{key_id_comp}
					\omega_i^2+p_i^2=\lambda_{\theta,i}(1-\lambda_{\beta,i})=\beta\gamma(1-\lambda_i) \lambda_{\theta,i},
				\end{align}
				the left block column gives
				\begin{align}
					\|\V_l^{-1}\|^2
					=\max_{i\ge2}2\Big|{-}\tfrac{1}{2\beta}-\tfrac{\mathrm{j}p_i}{2\beta\omega_i}\Big|^2
					=\max_{i\ge2}\frac{1}{2\beta^2}\cdot\frac{\omega_i^2+p_i^2}{\omega_i^2}
					\overset{\eqref{key_id_comp},\eqref{omega_b_def}}{=}\max_{i\ge2}\frac{\lambda_{\theta,i}}{2\beta^2 \nu_i}
					\le\frac{\lambda_{\theta}^2}{2\beta^2\nu},
				\end{align}
				where we used $\nu_i\ge\nu$, $\lambda_{\theta,i}\le\lambda_\theta^2$, with 	$\lambda_{\theta}=\sqrt{1-\theta\gamma(1-\lambda)}$. The right block column gives
				\begin{align}
					\|\V_r^{-1}\|^2=\max_{i\ge2}\frac{1}{2\omega_i^2}
					=\max_{i\ge2}\frac{1}{2\beta\gamma(1-\lambda_i)\nu_i}\le\frac{1}{2\beta\gamma(1-\lambda)\nu},
				\end{align}
				using $1-\lambda_i\ge1-\lambda$. Weighting by $(\I-\mathbf{\Lambda})$,
				\begin{align}
					\|\V_l^{-1}(\I-\mathbf{\Lambda})\|^2
					\le\frac{\lambda_{\theta}^2(1-\underline{\lambda})^2}{2\beta^2\nu},
					\quad
					\|\V_r^{-1}(\I-\mathbf{\Lambda})\|^2
					\le\frac{1-\underline{\lambda}}{2\beta\gamma\nu}.
				\end{align}
			\end{proof}

		\begin{corollary} \label{corr:phi_z_deviation}
			\rm Under the conditions of Lemma~\ref{lemma:transf_avg_companion}, it holds that
			\begin{subequations}
				\begin{align}
					\|\bm{\upphi}_{e}^{k}-\bar{\bm{\upphi}}_{e}^{k}\|^2
					&=\|\Q^{\top}\bm{\upphi}_{e}^{k}\|^2 \le 4\beta^2 \|\d^k\|^2, \label{phi_deviation}\\[1mm]
					\|\z^{k}-\bar{\z}^{k}\|^2
					&=
					\|\Q^{\top}\z^{k}\|^2
					\le
					8 \lambda_{\theta}^2\beta\gamma(1-\underline{\lambda})\|\d^k\|^2
					+
					2\theta^2\gamma^2(1-\underline{\lambda})^2
					\|\e^k\|^2 .
					\label{z_deviation}
				\end{align}
			\end{subequations}
			where $\underline{\lambda}\define\min_{i\ge2}\lambda_i$.
		\end{corollary}
		
		\begin{proof}
			From the decomposition in Lemma~\ref{lemma:transf_avg_companion}, $\V\d^{k}=\col\{\Q^{\top}\bm{\upphi}_e^{k},\Q^{\top}\widehat{\z}^{k}\}$. Using $\V=\P^{\top}\hat{\V}$ with $\hat{\V}=\bdiag\{V_i\}_{i=2}^N\otimes I_n$ and the explicit form of $V_i$,
			\begin{align}
				\begin{bmatrix}
					\Q^{\top}\bm{\upphi}_{e}^{k}\\[1mm]
					\Q^{\top}\widehat{\z}^{k}
				\end{bmatrix}
				=\V\d^{k}
				=\left(
				\begin{bmatrix}
					-\beta I_{N-1} & -\beta I_{N-1}\\[1mm]
					\diag\{p_i+\mathrm{j}\omega_i\}_{i=2}^N & \diag\{p_i-\mathrm{j}\omega_i\}_{i=2}^N
				\end{bmatrix}\otimes I_n
				\right)\P^{\top}\d^{k}.
			\end{align}
			Let $\P_u$ and $\P_l$ be the first and last $(N-1)n$ columns of $\P$, so that
			$\P^{\top}\d^{k}=\col\{\P_u^{\top}\d^{k},\P_l^{\top}\d^{k}\}$. Since $\|\P\|=1$ we have $\|\P_u^{\top}\d^{k}\|\le\|\d^{k}\|$ and $\|\P_l^{\top}\d^{k}\|\le\|\d^{k}\|$.	For the upper block,
			\begin{align}
				\|\Q^{\top}\bm{\upphi}_{e}^{k}\|^2
				=\big\|-\beta \P_u^{\top}\d^{k}-\beta \P_l^{\top}\d^{k}\big\|^2
				\le 2\beta^2\|\P_u^{\top}\d^{k}\|^2+2\beta^2\|\P_l^{\top}\d^{k}\|^2
				\le 4\beta^2 \|\d^{k}\|^2,
			\end{align}
			which proves \eqref{phi_deviation}. For the lower block, using $\|\diag\{p_i\pm\mathrm{j}\omega_i\}\otimes I_n\|^2=\max_{i\ge2}|p_i\pm\mathrm{j}\omega_i|^2=\max_{i\ge2}(p_i^2+\omega_i^2)$,
			\begin{align}
				\|\Q^{\top}\widehat{\z}^{k}\|^2
				&\le 2\big\|(\diag\{p_i+\mathrm{j}\omega_i\}\otimes I_n)\P_u^{\top}\d^{k}\big\|^2
				+2\big\|(\diag\{p_i-\mathrm{j}\omega_i\}\otimes I_n)\P_l^{\top}\d^{k}\big\|^2
				\nonumber\\
				&\le 4\max_{i\ge2}(p_i^2+\omega_i^2) \|\d^{k}\|^2
				\overset{\eqref{key_id_comp}}{=}4\max_{i\ge2}\big[\beta\gamma(1-\lambda_i) \lambda_{\theta,i}\big] \|\d^{k}\|^2
				\le 4\beta\gamma \lambda_{\theta}^2(1-\underline{\lambda}) \|\d^{k}\|^2,\label{z_hat_deviation}
			\end{align}
			where the last inequality uses $\lambda_{\theta,i}\le \lambda_{\theta}^2$ and $1-\lambda_i\le1-\underline{\lambda}$ for all $i\ge2$. By the definition of $\widehat{\z}^k$,
			\[
			\widehat{\z}^k
			=
			\z^k+(\I-\W_{\theta,\gamma})\e^k \quad \Longrightarrow \quad 	\Q^\top\z^k
			=
			\Q^\top\widehat{\z}^k
			-
			\theta \mathbf L_\gamma \Q^\top\e^k .
			\]
			Therefore,
			\begin{align}
				\|\Q^\top\z^k\|^2
				&\le
				2\|\Q^\top\widehat{\z}^k\|^2
				+
				2\theta^2
				\|\mathbf L_\gamma\Q^\top\e^k\|^2
				\nonumber\\
				&\le
				8\beta\gamma \lambda_{\theta}^2(1-\underline{\lambda})\|\d^k\|^2
				+
				2\theta^2
				\|\mathbf L_\gamma\|^2
				\|\Q^\top\e^k\|^2 \nonumber \\
					&\le
				8\beta\gamma \lambda_{\theta}^2(1-\underline{\lambda})\|\d^k\|^2
				+
				2\theta^2\gamma^2(1-\underline{\lambda})^2
				\|\e^k\|^2 ,
			\end{align}
			which proves \eqref{z_deviation}.
		\end{proof}
\begin{remark}[\sc \small Complex change of coordinates]\rm
	Under \eqref{rot_cond_comp}, the eigenvalues of each block $D_i$ form a
	complex-conjugate pair; hence $\V$, $\V^{-1}$, $\mathbf{\Delta}$, and the
	transformed iterate $\d^k$ may be complex-valued. Throughout this part of the
	analysis, $\|\cdot\|$ denotes the Euclidean norm on the corresponding complex
	space and its induced operator norm, and
	$\langle a,b\rangle=a^*b$ denotes the Hermitian inner product. Thus, norm
	expansions involve $2\Re\langle a,b\rangle$, while
	$|\langle a,b\rangle|$ denotes the complex modulus.
	
\end{remark}
		
		\subsection{Intermediate bounds}

				\begin{lemma}[\sc \small Centroid descent]\label{lemma:centroid_descent}
					\rm Suppose Assumptions~\ref{assump:network}--\ref{assump:stochastic_gradient} 
				and the conditions of Lemma~\ref{lemma:transf_avg_companion}
					hold with $\gamma \in (0,1]$. Then
					\begin{equation} 			\label{centroid_phi_bnd}
						\begin{aligned}		
								\Ex f(\bar{\phi}_{e}^{k+1})
								&\le
								\Ex f(\bar{\phi}_{e}^{k})
								-
								\frac{\alpha\beta}{2}
								\Ex\|\grad f(\bar{\phi}_{e}^{k})\|^2
								-
								\frac{\alpha\beta}{2}
								(1-\alpha\beta L)
								\Ex\|\overline{\grad\f}(\x^k)\|^2
								\\
								&\quad
								+
								\frac{4\alpha\beta^3L^2}{N}
								\Ex\|\d^k\|^2
								+
								\frac{\alpha\beta^3L^2}{N}
								\Ex\|\e^k\|^2
								+
								\frac{\alpha^2\beta^2L\sigma^2}{2N},
						\end{aligned}
					\end{equation}
					where $\bar{\phi}_{e}^{k}
					\define
					\frac{1}{N}(\one^\top\otimes I_n)\bm{\upphi}_{e}^{k}$ and $\overline{\grad\f}(\x^k)
					\define
					\frac1N\sum_{i=1}^N \grad f_i(x_i^k)$.
				\end{lemma}
				
				\begin{proof}
					From \eqref{centroid_phi_e}, we have $\bar{\phi}_{e}^{k+1}
						=
						\bar{\phi}_{e}^{k}
						-
						\alpha\beta
						\overline{\grad\F}(\x^k;\bxi^k)$,
					where $
					\overline{\grad\F}(\x^k;\bxi^k)
					\define
					\frac1N\sum_{i=1}^N
					\nabla F_i(x_i^k;\xi_i^k)$.
					By the $L$-smoothness of $f$, for any $y,z\in\mathbb R^n$,
					\begin{align} \label{L_smoothness}
						f(y)\le	f(z)	+
						\langle \grad f(z),y-z\rangle	+
						\frac{L}{2}\|y-z\|^2 .
					\end{align}
					Setting $y=\bar{\phi}_e^{k+1}$ and $z=\bar{\phi}_e^k$ gives
					\begin{align}
						f(\bar{\phi}_{e}^{k+1})
						&\le
						f(\bar{\phi}_{e}^{k})
						-
						\alpha\beta
						\big\langle
						\grad f(\bar{\phi}_{e}^{k}),
						\overline{\grad\F}(\x^k;\bxi^k)
						\big\rangle
						+
						\frac{\alpha^2\beta^2L}{2}
						\|
						\overline{\grad\F}(\x^k;\bxi^k)
						\|^2 .
					\end{align}
				Using Assumption \ref{assump:stochastic_gradient}, the	stochastic-gradient	 conditional expectation with respect to $\bm{\cF}^k$ satisfies
				\begin{subequations}
					\begin{align}
					\Ex[
					\overline{\grad\F}(\x^k;\bxi^k)
					\mid \bm{\cF}^k]
					&=
					\overline{\grad\f}(\x^k), \\
					\Ex\left[
					\|
					\overline{\grad\F}(\x^k;\bxi^k)
					\|^2
					\mid \bm{\cF}^k
					\right]
					&\le
					\|
					\overline{\grad\f}(\x^k)
					\|^2
					+
					\frac{\sigma^2}{N}.
				\end{align}
				\end{subequations}
		Therefore,
					\begin{align}
						\Ex[
						f(\bar{\phi}_{e}^{k+1})
						\mid \bm{\cF}^k]
						&\le
						f(\bar{\phi}_{e}^{k})
						-
						\alpha\beta
						\left\langle
						\grad f(\bar{\phi}_{e}^{k}),
						\overline{\grad\f}(\x^k)
						\right\rangle
						+
						\frac{\alpha^2\beta^2L}{2}
						\|
						\overline{\grad\f}(\x^k)
						\|^2
						+
						\frac{\alpha^2\beta^2L\sigma^2}{2N}.
						\label{centroid_descent_pre_identity}
					\end{align}
					Using the identity $
					2\langle a,b\rangle
					=
					\|a\|^2+\|b\|^2-\|a-b\|^2$,
					with $a=\grad f(\bar{\phi}_{e}^{k})$ and
					$b=\overline{\grad\f}(\x^k)$,
					we get
					\begin{align}
						-
						\left\langle
						\grad f(\bar{\phi}_{e}^{k}),
						\overline{\grad\f}(\x^k)
						\right\rangle
						&=
						-\frac12
						\|\grad f(\bar{\phi}_{e}^{k})\|^2
						-\frac12
						\|\overline{\grad\f}(\x^k)\|^2
						+
						\frac12
						\|
						\overline{\grad\f}(\x^k)
						-
						\grad f(\bar{\phi}_{e}^{k})
						\|^2 .
					\end{align}
					Substituting this into \eqref{centroid_descent_pre_identity} yields
					\begin{align}
						\Ex[
						f(\bar{\phi}_{e}^{k+1})
						\mid \bm{\cF}^k]
						&\le
						f(\bar{\phi}_{e}^{k})
						-
						\frac{\alpha\beta}{2}
						\|\grad f(\bar{\phi}_{e}^{k})\|^2
						-
						\frac{\alpha\beta}{2}
						(1-\alpha\beta L)
						\|\overline{\grad\f}(\x^k)\|^2
						\nonumber\\
						&\quad
						+
						\frac{\alpha\beta}{2}
						\|
						\overline{\grad\f}(\x^k)
						-
						\grad f(\bar{\phi}_{e}^{k})
						\|^2
						+
						\frac{\alpha^2\beta^2L\sigma^2}{2N}.
						\label{centroid_descent_with_mismatch}
					\end{align}
					It remains to bound the gradient mismatch. By definition, $\grad f(\bar{\phi}_{e}^{k})
					=
					\frac1N\sum_{i=1}^N
					\grad f_i(\bar{\phi}_{e}^{k})$, and thus,
					\begin{align}
						\|
						\overline{\grad\f}(\x^k)
						-
						\grad f(\bar{\phi}_{e}^{k})
						\|^2
						&=
						\left\|
						\frac1N
						\sum_{i=1}^N
						(
						\grad f_i(x_i^k)
						-
						\grad f_i(\bar{\phi}_{e}^{k})
						)
						\right\|^2
						\nonumber\\
						&\le
						\frac1N
						\sum_{i=1}^N
						\|
						\grad f_i(x_i^k)
						-
						\grad f_i(\bar{\phi}_{e}^{k})
						\|^2
						\nonumber\\
						&\le
						\frac{L^2}{N}
						\sum_{i=1}^N
						\|x_i^k-\bar{\phi}_{e}^{k}\|^2
						\nonumber\\
						&=
						\frac{L^2}{N}
						\|\x^k-\bar{\bm{\upphi}}_{e}^{k}\|^2 .
						\label{grad_mismatch_pre_x_bound}
					\end{align}
					Next, since $\x^k=\W_\gamma\bm{\upphi}^k$ and
					$\bm{\upphi}^k=\bm{\upphi}_e^k-\beta\e^k$, we have
					\begin{align}
						\x^k-\bar{\bm{\upphi}}_e^k
						&=
						\W_\gamma(\bm{\upphi}_e^k-\beta\e^k)
						-
						\bar{\bm{\upphi}}_e^k
						\nonumber\\
						&=
						\W_\gamma
						(\bm{\upphi}_e^k-\bar{\bm{\upphi}}_e^k)
						-
						\beta\W_\gamma\e^k .
					\end{align}
					Using $\|\W_\gamma\|\le 1$, we obtain
					\begin{align}
						\|\x^k-\bar{\bm{\upphi}}_e^k\|^2
						&\le
						2
						\|\bm{\upphi}_e^k-\bar{\bm{\upphi}}_e^k\|^2
						+
						2\beta^2\|\e^k\|^2
						\nonumber\\
						&\overset{\eqref{phi_deviation}}{\le}
						8\beta^2\|\d^k\|^2
						+
						2\beta^2\|\e^k\|^2 .
						\label{x_bar_phi_e_bound}
					\end{align}
					Combining \eqref{grad_mismatch_pre_x_bound} and
					\eqref{x_bar_phi_e_bound} gives
					\begin{align}
						\|
						\overline{\grad\f}(\x^k)
						-
						\grad f(\bar{\phi}_{e}^{k})
						\|^2
						&\le
						\frac{8\beta^2L^2}{N}\|\d^k\|^2
						+
						\frac{2\beta^2L^2}{N}\|\e^k\|^2 .
						\label{grad_mismatch_final_bound}
					\end{align}
					Substituting \eqref{grad_mismatch_final_bound} into
					\eqref{centroid_descent_with_mismatch} yields
					\begin{align}
						\Ex[
						f(\bar{\phi}_{e}^{k+1})
						\mid \bm{\cF}^k]
						&\le
						f(\bar{\phi}_{e}^{k})
						-
						\frac{\alpha\beta}{2}
						\|\grad f(\bar{\phi}_{e}^{k})\|^2
						-
						\frac{\alpha\beta}{2}
						(1-\alpha\beta L)
						\|\overline{\grad\f}(\x^k)\|^2
						\nonumber\\
						&\quad
						+
						\frac{4\alpha\beta^3L^2}{N}
						\|\d^k\|^2
						+
						\frac{\alpha\beta^3L^2}{N}
						\|\e^k\|^2
						+
						\frac{\alpha^2\beta^2L\sigma^2}{2N}.
					\end{align}
					Taking total expectation completes the proof.
				\end{proof}
\begin{lemma}[\sc \small Average deviation bound]\label{lemma:avg_deviation_bound}
	\rm Suppose Assumptions~\ref{assump:network}--\ref{assump:stochastic_gradient}
	and the conditions of Lemma~\ref{lemma:transf_avg_companion} hold, and define
	$\bar\lambda_{\theta}\define\frac{1+\lambda_{\theta}}{2}$. If
	\begin{align}\label{step_size_d_bound}
		\alpha\beta\le\frac{\sqrt{\nu}\,(1-\lambda_\theta)}{20L},
	\end{align}
	then
	\begin{equation}\label{hat_d_bnd}
		\begin{aligned}
			\Ex\|\d^{k+1}\|^2
			&\le
			\bar\lambda_\theta\,\Ex\|\d^k\|^2
			+
			\left[
			\frac{52\beta(1-\underline\lambda)}{\nu\,\theta(1-\lambda)}
			+
			\frac{96\alpha^2\beta^2L^2}{\nu\,\theta\gamma(1-\lambda)}
			\right]
			\Ex\|\e^k\|^2
			\\
			&\quad
			+
			\frac{10\alpha^4\beta NL^2}{\nu\,\theta\gamma^2(1-\lambda)^2}
			\Ex\|\overline{\grad\f}(\x^k)\|^2
			+
			\frac{10\alpha^4\beta L^2\sigma^2}{\nu\,\theta\gamma^2(1-\lambda)^2}
			+
			\frac{12\alpha^2N\sigma^2}{\nu}.
		\end{aligned}
	\end{equation}
\end{lemma}
						\begin{proof}
						Recall from \eqref{phi_e_and_z_trans_diag_update} that
								\begin{align}
							\d^{k+1}
							=\mathbf{\Delta}
							\d^{k}
							-\alpha
							\V^{-1} \begin{bmatrix}
								\beta 	\I \\[1mm]
								\theta \mathbf{L}_{\gamma}
							\end{bmatrix}
							\Q^{\top}\g_1^k
							+
							\V^{-1}\begin{bmatrix}
								\mathbf 0 \\[1mm]
								\alpha 	\I
							\end{bmatrix}
							\Q^{\top}\g_2^{k+1}
							+
							\V^{-1}\E_e
							\e^k, \label{phi_e_and_z_trans_diag_update2}
						\end{align}
						where $	\g_1^k =
						\nabla \F(\x^k;\bxi^k)-\grad \f(\bar{\bm{\upphi}}_e^k)$,
						$\g_2^{k+1}
						=
						\grad \f(\bar{\bm{\upphi}}_e^{k+1})
						-
						\grad \f(\bar{\bm{\upphi}}_e^k)$, and
						\begin{align} \label{E_e}
							\E_e\define\begin{bmatrix}
								\beta(\beta+\theta)\mathbf{L}_{\gamma}	\\[1mm]
								(\theta^2\mathbf{L}_{\gamma}^2-\beta \mathbf{L}_{\gamma} \mathbf{\Lambda}_{\theta,\gamma})	
							\end{bmatrix}\Q^\top
						\end{align}
						For notational convenience, we also define
							\begin{subequations}
							\begin{align}
									\u^k&\define\Q^\top\big(\nabla\F(\x^k;\bxi^k)-\grad\f(\x^k)\big) \\
								\mathbf m^k&\define\Q^\top\big(\grad\f(\x^k)-\grad\f(\bar{\bm\upphi}_e^k)\big) \\
								\mathbf r^{k+1}&\define\V_r^{-1}\Q^\top\big(\grad\f(\bar{\bm\upphi}_e^{k+1})-\grad\f(\bar{\bm\upphi}_e^k)\big) \\
									\mathbf B&\define\beta\V_l^{-1}+\theta\V_r^{-1}\mathbf L_\gamma,
							\end{align}
							\end{subequations}
							where $\V_l^{-1}$ and $\V_r^{-1}$ are the left block and right block of $\V^{-1}=[\V_l^{-1}~\V_r^{-1}]$, respectively.								Splitting $\Q^\top\g_1^k=\u^k+\m^k$ in
							\eqref{phi_e_and_z_trans_diag_update2} and using $\V^{-1}=[\V_l^{-1}~\V_r^{-1}]$,
							\begin{align}
								\d^{k+1}
								=\underbrace{\mathbf\Delta\d^k+\V^{-1}\E_e\e^k-\alpha \mathbf B\mathbf m^k}_{\define ~\mathbf a^k\ (\bm{\cF}^k\text{-measurable})}
								+\alpha \mathbf r^{k+1}-\alpha \mathbf B\u^k ,
								\label{d_split}
							\end{align}
							Squaring \eqref{d_split} and separating the zero-mean term $\alpha\mathbf B\u^k$, the cross term with the
							$\bm{\cF}^k$-measurable $\mathbf a^k$ vanishes in expectation while
							$-2\Re\langle\mathbf r^{k+1},\mathbf B\u^k\rangle
							\le
							2|\langle\mathbf r^{k+1},\mathbf B\u^k\rangle|
							\le
							\|\mathbf r^{k+1}\|^2+\|\mathbf B\u^k\|^2$; hence
							\begin{align}
								\Ex[\|\d^{k+1}\|^2\mid\bm{\cF}^k]
								\le\Ex[\|\mathbf a^k+\alpha\mathbf r^{k+1}\|^2\mid\bm{\cF}^k]
								+\alpha^2\Ex[\|\mathbf r^{k+1}\|^2\mid\bm{\cF}^k]
								+2\alpha^2\Ex[\|\mathbf B\u^k\|^2\mid\bm{\cF}^k].
							\end{align}
					We now use Jensen's inequality $\|a+b\|^2 \leq \frac{1}{t}\|a\|^2+ \frac{1}{1-t}\|b\|^2$ for any $0<t<1$.		Since $\mathbf a^k+\alpha\mathbf r^{k+1}=\mathbf\Delta\d^k+\V^{-1}\E_e\e^k-\alpha(\B \mathbf m^k-\mathbf r^{k+1})$, Jensen's
							inequality with weight $\lambda_{\theta}$, $\|\mathbf\Delta\d^k\|\le\lambda_{\theta}\|\d^k\|$, and
							\begin{align}
								\|\V^{-1}\E_e\e^k-\alpha(\B \mathbf m^k-\mathbf r^{k+1})\|^2 &\le2\|\V^{-1}\E_e\e^k\|^2+2\|\alpha(\B \mathbf m^k-\mathbf r^{k+1})\|^2 \nonumber \\
								&\leq2\|\V^{-1}\E_e\e^k\|^2+ 4\alpha^2\|\B \mathbf m^k\|^2+4\alpha^2\|\mathbf r^{k+1}\|^2
							\end{align}
							 give,
							\begin{align}
								\Ex[\|\d^{k+1}\|^2\mid\bm{\cF}^k]
								&\le\lambda_{\theta}\|\d^k\|^2
								+\frac{2}{1-\lambda_{\theta}}\|\V^{-1}\E_e\e^k\|^2
								+\frac{4\alpha^2}{1-\lambda_{\theta}}\|\mathbf B\mathbf m^k\|^2
								\nonumber\\
								&\quad
								+\frac{5\alpha^2}{1-\lambda_{\theta}}\Ex[\|\mathbf r^{k+1}\|^2\mid\bm{\cF}^k]
								+2\alpha^2\Ex[\|\mathbf B\u^k\|^2\mid\bm{\cF}^k],
								\label{d_master}
							\end{align}
							where we used  $1\le\frac1{1-\lambda_{\theta}}$. Both the gradient noise and the gradient mismatch enter through $\mathbf B$, and for any $\v$,
					\begin{align}
						\|\mathbf B\v\|^2
						\le 2\beta^2 v_l^2\|\v\|^2+2\theta^2\gamma^2\bar v_r^2\|\v\|^2
						=2w^2\|\v\|^2,
						\label{B_bound}
					\end{align}
					where $
					w^2\define\beta^2 v_l^2+\theta^2\gamma^2\bar v_r^2$.	By \eqref{B_bound}, smoothness, $\|\Q^\top\cdot\|\le\|\cdot\|$, and Assumption \ref{assump:stochastic_gradient}, we have
						\begin{subequations}
								\begin{align}
								\|\mathbf B\mathbf m^k\|^2&\le 2w^2 L^2\|\x^k-\bar{\bm\upphi}_e^k\|^2\\
								\Ex[\|\mathbf B\u^k\|^2\mid\bm{\cF}^k]&\le 2w^2 N\sigma^2 \\
								\|\mathbf r^{k+1}\|^2 &\le v_r^2 L^2\|\bar{\bm\upphi}_e^{k+1}-\bar{\bm\upphi}_e^k\|^2.
							\end{align}
						\end{subequations}
					Therefore, putting these bounds into \eqref{d_master}, we get
								\begin{align}
						\Ex[\|\d^{k+1}\|^2\mid\bm{\cF}^k]
						&\le\lambda_{\theta}\|\d^k\|^2
						+\frac{2}{1-\lambda_{\theta}}\|\V^{-1}\E_e\e^k\|^2
						+\frac{8\alpha^2w^2L^2}{1-\lambda_{\theta}}\|\x^k-\bar{\bm\upphi}_e^k\|^2
						\nonumber\\
						&\quad
						+\frac{5\alpha^2v_r^2 L^2}{1-\lambda_{\theta}}\Ex[\|\bar{\bm\upphi}_e^{k+1}-\bar{\bm\upphi}_e^k\|^2\mid\bm{\cF}^k]
						+4\alpha^2 w^2 N\sigma^2,
						\label{d_master2}
					\end{align}	
					We now bound the term $\|\V^{-1}\E_e\e^k\|^2$, where $\E_e$ is defined in
					\eqref{E_e}. Using $
					\mathbf\Lambda_{\theta,\gamma}
					=
					\I-\theta\mathbf L_\gamma$,
					the lower block can be factored as
					\begin{align}
						\theta^2\mathbf L_\gamma^2
						-
						\beta\mathbf L_\gamma\mathbf\Lambda_{\theta,\gamma}
						&=
						\mathbf L_\gamma
						\big[
						\theta(\theta+\beta)\mathbf L_\gamma-\beta\I
						\big].
					\end{align}
					The condition \eqref{rot_cond_comp}, namely $
					\gamma(1-\underline{\lambda})
					<
					\frac{4\beta}{(\beta+\theta)^2}$,
					implies that, for every $i\ge2$,
					\[
					\ell_{\gamma,i}
					\define
					\gamma(1-\lambda_i)
					\le
					\gamma(1-\underline{\lambda})
					<
					\frac{4\beta}{(\beta+\theta)^2}.
					\]
					Hence
					\[
					0
					\le
					\theta(\theta+\beta)\ell_{\gamma,i}
					<
					\theta(\theta+\beta)
					\frac{4\beta}{(\beta+\theta)^2}
					=
					\frac{4\beta\theta}{\beta+\theta}
					\le
					4\beta.
					\]
					Therefore, $\max_{i\ge2}
					\left|
					\theta(\theta+\beta)\ell_{\gamma,i}-\beta
					\right|
					\le
					3\beta$.
					Moreover, from \eqref{V_left_right_bound}, we have
					\[
					\|\V_l^{-1}\mathbf L_\gamma\|^2
					\le
					\gamma^2\bar v_l^2
					=
					\frac{\lambda_\theta^2\gamma^2(1-\underline\lambda)^2}
					{2\beta^2\nu}, \quad
					\|\V_r^{-1}\mathbf L_\gamma\|^2
					\le
					\gamma^2\bar v_r^2
					=
					\frac{\gamma(1-\underline\lambda)}
					{2\beta\nu}.
					\]
					Thus,
					\begin{align}
						\|\V^{-1}\E_e\e^k\|^2
						&\le
						2\beta^2(\beta+\theta)^2
						\|\V_l^{-1}\mathbf L_\gamma\|^2
						\|\e^k\|^2
						+
						18\beta^2
						\|\V_r^{-1}\mathbf L_\gamma\|^2
						\|\e^k\|^2
						\nonumber\\
						&\le
						\frac{
							(\beta+\theta)^2\lambda_\theta^2\gamma^2(1-\underline\lambda)^2
						}{\nu}
						\|\e^k\|^2
						+
						\frac{
							9\beta\gamma(1-\underline\lambda)
						}{\nu}
						\|\e^k\|^2
						\nonumber\\
						&\le
						\frac{
							13\beta\gamma(1-\underline\lambda)
						}{\nu}
						\|\e^k\|^2 .
						\label{buffer_bound}
					\end{align}
					In the last step, we used $\lambda_\theta^2\le1$ and $
					(\beta+\theta)^2\gamma(1-\underline\lambda)
					<
					4\beta$.
					Substituting \eqref{buffer_bound} into \eqref{d_master2}, and using
					\[
					\|\x^k-\bar{\bm\upphi}_e^k\|^2
					\overset{\eqref{x_bar_phi_e_bound}}{\leq}
					8\beta^2\|\d^k\|^2
					+
					2\beta^2\|\e^k\|^2
					\]
					together with
					\[
					\Ex[
					\|\bar{\bm\upphi}_e^{k+1}-\bar{\bm\upphi}_e^k\|^2
					\mid
					\bm{\cF}^{k}]
					\le
					N\alpha^2\beta^2
					\|\overline{\grad\f}(\x^k)\|^2
					+
					\alpha^2\beta^2\sigma^2,
					\]
					gives
					\begin{align}
						\Ex\|\d^{k+1}\|^2
						&\le
						\left(
						\lambda_{\theta}
						+
						\frac{64\alpha^2\beta^2 w^2L^2}
						{1-\lambda_{\theta}}
						\right)
						\Ex\|\d^k\|^2
						\nonumber\\
						&\quad
						+
						\left(
						\frac{26\beta\gamma(1-\underline\lambda)}
						{\nu(1-\lambda_{\theta})}
						+
						\frac{16\alpha^2\beta^2w^2L^2}
						{1-\lambda_{\theta}}
						\right)
						\Ex\|\e^k\|^2
						\nonumber\\
						&\quad
						+
						\frac{5\alpha^4\beta^2Nv_r^2L^2}
						{1-\lambda_{\theta}}
						\Ex\|\overline{\grad\f}(\x^k)\|^2
						+
						\frac{5\alpha^4\beta^2v_r^2L^2\sigma^2}
						{1-\lambda_{\theta}}
						+
						4\alpha^2w^2N\sigma^2 .
						\label{d_pre_subst}
					\end{align}
					Now, by \eqref{V_left_right_bound},
					\[
					w^2
					=
					\beta^2v_l^2+\theta^2\gamma^2\bar v_r^2
					=
					\frac{\lambda_\theta^2}{2\nu}
					+
					\frac{\theta^2\gamma(1-\underline\lambda)}
					{2\beta\nu}.
					\]
					Using $\gamma(1-\underline\lambda)
					<
					\frac{4\beta}{(\beta+\theta)^2}$,
					we obtain
					\begin{align}
						\frac{\theta^2\gamma(1-\underline\lambda)}
						{2\beta\nu}
						&<
						\frac{\theta^2}{2\beta\nu}
						\cdot
						\frac{4\beta}{(\beta+\theta)^2}
						=
						\frac{2\theta^2}
						{\nu(\beta+\theta)^2}
						\le
						\frac{2}{\nu}.
					\end{align}
					Since $\lambda_\theta^2\le1$, it follows that $w^2\le				\frac{3}{\nu}$. 		Moreover,
					\[
					v_r^2
					=
					\frac{1}{2\beta\gamma(1-\lambda)\nu}
					\le
					\frac{1}{\beta\gamma(1-\lambda)\nu},
					\]
					and $\frac{1}{1-\lambda_\theta}
					\le
					\frac{2}{\theta\gamma(1-\lambda)}$.
					Substituting these estimates into \eqref{d_pre_subst} yields
					\begin{align}
						\Ex\|\d^{k+1}\|^2
						&\le
						\left(
						\lambda_\theta
						+
						\frac{192\alpha^2\beta^2L^2}
						{\nu(1-\lambda_\theta)}
						\right)
						\Ex\|\d^k\|^2
						\nonumber\\
						&\quad
						+
						\left[
						\frac{52\beta(1-\underline\lambda)}
						{\nu\theta(1-\lambda)}
						+
						\frac{96\alpha^2\beta^2L^2}
						{\nu\theta\gamma(1-\lambda)}
						\right]
						\Ex\|\e^k\|^2
						\nonumber\\
						&\quad
						+
						\frac{10\alpha^4\beta NL^2}
						{\nu\theta\gamma^2(1-\lambda)^2}
						\Ex\|\overline{\grad\f}(\x^k)\|^2
						+
						\frac{10\alpha^4\beta L^2\sigma^2}
						{\nu\theta\gamma^2(1-\lambda)^2}
						+
						\frac{12\alpha^2N\sigma^2}{\nu}.
						\label{d_pre_final_nu}
					\end{align}
					Finally, under
					\[
					\alpha\beta
					\le
					\frac{\sqrt{\nu}(1-\lambda_\theta)}{20L}\quad \Longrightarrow \quad
					\frac{192\alpha^2\beta^2L^2}
					{\nu(1-\lambda_\theta)}
					\le
					\frac{1-\lambda_\theta}{2}.
					\]
					Hence the coefficient of $\Ex\|\d^k\|^2$ is bounded by $
					\lambda_\theta+\frac{1-\lambda_\theta}{2}
					=
					\frac{1+\lambda_\theta}{2}
					=
					\bar\lambda_\theta$. Putting this into \eqref{d_pre_final_nu} yields our result.
						\end{proof}

\begin{lemma}[\sc \small Error-Feedback bound]\label{lemma:bound_of_e}
	\rm
	Suppose Assumptions~\ref{assump:network}--\ref{assump:comp} hold and the conditions of
	Lemma~\ref{lemma:transf_avg_companion} hold. For $0<\delta \leq1$, assume in addition
\begin{subequations}\label{step_conditions_e_bnd}
	\begin{align}
		\gamma\beta
		&\le
		\frac{\eta\delta}
		{8\sqrt3(1-\underline\lambda)\sqrt{1-\delta}},
		\label{eq:gamma-beta-separated-condition}
		\\
		\gamma\theta
		&\le
		\frac{\eta\delta}
		{8\sqrt6(1-\underline\lambda)\sqrt{1-\delta}},
		\label{eq:gamma-theta-separated-condition}
		\\
		\alpha\beta
		&\le
		\frac{\eta\delta}
		{4\sqrt6L\sqrt{1-\delta}}.
		\label{eq:alpha-separated-condition}
	\end{align}
\end{subequations}
	Then
	\begin{align}\label{e_bnd}
		\Ex\|\e^{k+1}\|^2
		&\le
		\Big(1-\tfrac{\eta\delta}{4}\Big)
		\Ex\|\e^k\|^2
		+
		\tfrac{2(1-\delta)}{\eta\delta}
		\Big(
		24\beta^2\gamma^2(1-\underline\lambda)^2
		+
		24\alpha^2\beta^2L^2
		+
		48\beta\gamma(1-\underline\lambda)
		\Big)
		\Ex\|\d^k\|^2
		\nonumber\\
		&\quad
		+
		\tfrac{12\alpha^2N(1-\delta)}{\eta\delta}
		\Ex\|\grad f(\bar\phi_e^k)\|^2
		+
		\tfrac{6(1-\delta)}{\eta\delta}
		\alpha^2N\sigma^2 .
	\end{align}
\end{lemma}

\begin{proof}
	Write $
	\s^k\define \hat\x^k-\bm\upphi^k$. Then,
	\begin{align}
		\e^{k+1}
		&=
		\s^k+\e^k-\bm{\cC}(\s^k+\eta\e^k) \nonumber\\
		&=
		(1-\eta)\e^k
		+
		\big[
		(\s^k+\eta\e^k)
		-
		\bm{\cC}(\s^k+\eta\e^k)
		\big].
	\end{align}
	Viewing the right-hand side as a convex combination of $\e^k$ and
	\[
	\eta^{-1}
	\big[
	(\s^k+\eta\e^k)
	-
	\bm{\cC}(\s^k+\eta\e^k)
	\big]
	\]
	with weights $1-\eta$ and $\eta$, convexity of $\|\cdot\|^2$ gives
	\begin{align}
		\|\e^{k+1}\|^2
		\le
		(1-\eta)\|\e^k\|^2
		+
		\frac1\eta
		\big\|
		(\s^k+\eta\e^k)
		-
		\bm{\cC}(\s^k+\eta\e^k)
		\big\|^2 . \label{hhh222}
	\end{align}
	To apply the compressor contraction condition rigorously, define $
	\bm{\cG}^k
	\define
	\bm{\cF}^k
	\vee
	\sigma(\xi_1^k,\ldots,\xi_N^k)$.
	Conditioned on $\bm{\cG}^k$, the compressor input
	$\s^k+\eta\e^k$ is fixed, while the current compressor randomness is
	independent of $\bm{\cG}^k$. Therefore, by
	\eqref{compresion_contractive},
	\begin{align}
		&\Ex\!\left[
		\left\|
		(\s^k+\eta\e^k)
		-
		\bm{\cC}(\s^k+\eta\e^k)
		\right\|^2
		\,\middle|\,
		\bm{\cG}^k
		\right]
		\le
		(1-\delta)
		\|\s^k+\eta\e^k\|^2.
		\label{compression_conditional_ef}
	\end{align}
	Combining \eqref{compression_conditional_ef} with \eqref{hhh222}
	and conditioning on $\bm{\cG}^k$ gives
	\begin{align}
		\Ex\big[
		\|\e^{k+1}\|^2
		\mid
		\bm{\cG}^k
		\big]
		\le
		(1-\eta)\|\e^k\|^2
		+
		\frac{1-\delta}{\eta}
		\|\s^k+\eta\e^k\|^2.
	\end{align}
	Taking conditional expectation with respect to $\bm{\cF}^k$ and using the
	tower property, since $\bm{\cF}^k\subseteq \bm{\cG}^k$,
	$
	\Ex\!\left[
	\Ex\!\left[\|\e^{k+1}\|^2\mid \bm{\cG}^k\right]
	\middle|\bm{\cF}^k
	\right]
	=
	\Ex\!\left[\|\e^{k+1}\|^2\mid\bm{\cF}^k\right]$ yields
	\begin{align}
		\Ex\big[\|\e^{k+1}\|^2\mid\bm{\cF}^k\big]
		\le
		(1-\eta)\|\e^k\|^2
		+
		\frac{1-\delta}{\eta}
		\Ex\big[
		\|\s^k+\eta\e^k\|^2
		\mid
		\bm{\cF}^k
		\big].
	\end{align}
	Using Jensen's inequality, $\|a+b\|^2 \le \frac{1}{t}\|a\|^2 + \frac{1}{1-t}\|b\|^2$, with $t=\frac{\delta}{2-\delta}$,
	we have
	\[
	\|\s^k+\eta\e^k\|^2
	\le
	\frac{2-\delta}{\delta}\|\s^k\|^2
	+
	\frac{2-\delta}{2(1-\delta)}
	\eta^2\|\e^k\|^2 .
	\]
	Since $\e^k$ is $\\bm{cF}^{k}$-measurable, the $\|\e^k\|^2$ terms combine to
	$1-\eta\delta/2$. Bounding $2-\delta\le2$ in the remaining term gives
	\begin{align}\label{e_bnd1}
		\Ex\big[\|\e^{k+1}\|^2\mid\bm{\cF}^{k}\big]
		\le
		\Big(1-\tfrac{\eta\delta}{2}\Big)\|\e^k\|^2
		+
		\frac{2(1-\delta)}{\eta\delta}
		\Ex\big[
		\|\s^k\|^2
		\mid
		\bm{\cF}^{k}
		\big].
	\end{align}
	We now bound $\|\s^k\|^2$. Using $\x^k=\W_\gamma\bm\upphi^k$,
	$\z^k=\y^k+\alpha\grad\f(\bar{\bm\upphi}_e^k)$, and
	$(\W_\gamma-\I)\bar{\bm\upphi}_e^k=0$, we obtain
	\begin{align}
		\s^k
		&=
		(\W_\gamma-\I)\bm\upphi^k
		-
		\alpha\nabla\F(\x^k;\bxi^k)
		-
		\y^k
		\nonumber\\
		&=
		(\W_\gamma-\I)(\bm\upphi^k-\bar{\bm\upphi}_e^k)
		-
		\alpha
		\big(
		\nabla\F(\x^k;\bxi^k)
		-
		\grad\f(\bar{\bm\upphi}_e^k)
		\big)
		-
		\z^k .
	\end{align}
	Therefore,
	\begin{align}
		\|\s^k\|^2
		&\le
		3\gamma^2(1-\underline\lambda)^2
		\|\bm\upphi^k-\bar{\bm\upphi}_e^k\|^2
		+
		3\alpha^2
		\|
		\nabla\F(\x^k;\bxi^k)
		-
		\grad\f(\bar{\bm\upphi}_e^k)
		\|^2
		+
		3\|\z^k\|^2 .
		\label{s_split_bound}
	\end{align}
	Here we used $\|\W_\gamma-\I\|
	=
	\gamma(1-\underline\lambda)$. Using $\bm{\upphi}^k=\bm{\upphi}_e^k-\beta\e^k$ and
	\eqref{phi_deviation}, we get
	\begin{align}
		\|\bm\upphi^k-\bar{\bm\upphi}_e^k\|^2
		&\le
		2\|\bm\upphi_e^k-\bar{\bm\upphi}_e^k\|^2
		+
		2\beta^2\|\e^k\|^2
		\nonumber\\
		&\le
		8\beta^2\|\d^k\|^2
		+
		2\beta^2\|\e^k\|^2 .
		\label{phi_minus_bar_bound_e}
	\end{align}
	The gradient term obeys
	\begin{align}
		\Ex\big[
		\|
		\nabla\F(\x^k;\bxi^k)
		-
		\grad\f(\bar{\bm\upphi}_e^k)
		\|^2
		\mid
		\bm{\cF}^{k}
		\big]
		&\le
		L^2\|\x^k-\bar{\bm\upphi}_e^k\|^2
		+
		N\sigma^2
		\nonumber\\
		&\overset{\eqref{x_bar_phi_e_bound}}{\le}
		8\beta^2L^2\|\d^k\|^2
		+
		2\beta^2L^2\|\e^k\|^2
		+
		N\sigma^2 .
		\label{grad_noise_bound_e}
	\end{align}
	Next, splitting $\z^k=(\z^k-\bar\z^k)+\bar\z^k$ and using the 
	deviation bound \eqref{z_deviation}, we have
	\[
	\|\z^k-\bar\z^k\|^2
	\le
	8\beta\gamma(1-\underline\lambda)\|\d^k\|^2
	+
	2\theta^2\gamma^2(1-\underline\lambda)^2\|\e^k\|^2,
	\]
where we used $\lambda_\theta\leq 1$.	Together with $
	\|\bar\z^k\|^2
	\le
	\alpha^2N\|\grad f(\bar\phi_e^k)\|^2$,
	we obtain
	\begin{align}
		\|\z^k\|^2
		&\le
		2\|\z^k-\bar\z^k\|^2
		+
		2\|\bar\z^k\|^2
		\nonumber\\
		&\le
		16\beta\gamma(1-\underline\lambda)\|\d^k\|^2
		+
		4\theta^2\gamma^2(1-\underline\lambda)^2\|\e^k\|^2
		+
		2\alpha^2N\|\grad f(\bar\phi_e^k)\|^2 .
		\label{z_bound_e}
	\end{align}
	Combining \eqref{s_split_bound}, \eqref{phi_minus_bar_bound_e},
	\eqref{grad_noise_bound_e}, and \eqref{z_bound_e}, we get
	\begin{align}\label{s_bound}
		\Ex[
		\|\s^k\|^2
		\mid
		\bm{\cF}^{k}
		]
		&\le
		\Big(
		24\beta^2\gamma^2(1-\underline\lambda)^2
		+
		24\alpha^2\beta^2L^2
		+
		48\beta\gamma(1-\underline\lambda)
		\Big)
		\|\d^k\|^2
		\nonumber\\
		&\quad
		+
		\Big(
		6\gamma^2(1-\underline\lambda)^2(\beta^2+2\theta^2)
		+
		6\alpha^2\beta^2L^2
		\Big)
		\|\e^k\|^2
		\nonumber\\
		&\quad
		+
		6N\alpha^2
		\|\grad f(\bar\phi_e^k)\|^2
		+
		3\alpha^2N\sigma^2 .
	\end{align}
	Substituting \eqref{s_bound} into \eqref{e_bnd1} and taking total expectation
	gives
	\begin{align}
		\Ex\|\e^{k+1}\|^2
		&\le
		\Big(
		1-\tfrac{\eta\delta}{2}
		+
		\tfrac{2(1-\delta)}{\eta\delta}
		\Big[
		6\gamma^2(1-\underline\lambda)^2(\beta^2+2\theta^2)
		+
		6\alpha^2\beta^2L^2
		\Big]
		\Big)
		\Ex\|\e^k\|^2
		\nonumber\\
		&\quad
		+
		\tfrac{2(1-\delta)}{\eta\delta}
		\Big(
		24\beta^2\gamma^2(1-\underline\lambda)^2
		+
		24\alpha^2\beta^2L^2
		+
		48\beta\gamma(1-\underline\lambda)
		\Big)
		\Ex\|\d^k\|^2
		\nonumber\\
		&\quad
		+
		\tfrac{12\alpha^2N(1-\delta)}{\eta\delta}
		\Ex\|\grad f(\bar\phi_e^k)\|^2
		+
		\tfrac{6(1-\delta)}{\eta\delta}
		\alpha^2N\sigma^2 .
	\end{align}
	The coefficient of $\Ex\|\e^k\|^2$ is bounded by
	$1-\eta\delta/4$ provided
	\[
	\tfrac{2(1-\delta)}{\eta\delta}
	\Big(
	6\gamma^2(1-\underline\lambda)^2(\beta^2+2\theta^2)
	+
	6\alpha^2\beta^2L^2
	\Big)
	\le
	\tfrac{\eta\delta}{4}.
	\]
	A sufficient separated condition is
	\begin{subequations}
		\begin{align}
		12\gamma^2(1-\underline\lambda)^2(\beta^2+2\theta^2)
		\frac{1-\delta}{\eta\delta}
		&\le
		\frac{\eta\delta}{8},
		\\
		12\alpha^2\beta^2L^2
		\frac{1-\delta}{\eta\delta}
		&\le
		\frac{\eta\delta}{8}.
	\end{align}
	\end{subequations}
	These are satisfied under \eqref{step_conditions_e_bnd}, which implies our result \eqref{e_bnd}.
\end{proof}

		\section{Nonconvex proof} \label{app_noncvx_proof}
Here, we use Lemmas \ref{lemma:centroid_descent}, \ref{lemma:avg_deviation_bound},  \ref{lemma:bound_of_e} to prove Theorem \ref{thm_convergence}.	For convenience, define the shorthand
	\begin{align} \label{def_G_D_E_caligraphic}
		\cG_k \define \Ex\|\nabla f(\bar{\phi}_e^k)\|^2,\quad
		\bar \cG_k \define \Ex\|\overline{\nabla f}(\x^k)\|^2,\quad
		\cD_k \define  \Ex\|\d^k\|^2,\quad
		\cE_k \define  \Ex\|\e^k\|^2 .
	\end{align}
	Subtracting $f^\star$ from \eqref{centroid_phi_bnd} and using
	$\alpha\beta\le\frac{1}{2L}$, so that $1-\alpha\beta L\ge\frac12$, with
	$\tilde f(\bar\phi_e^k)=f(\bar\phi_e^k)-f^\star$, gives
	\begin{align} \label{centroid_phi_bnd2}
		\Ex\tilde f(\bar{\phi}_e^{k+1})
		&\le
		\Ex\tilde f(\bar{\phi}_e^{k})
		-\frac{\alpha\beta}{2}\cG_k
		-\frac{\alpha\beta}{4}\bar\cG_k
		+\frac{4\alpha\beta^3 L^2}{N}\cD_k
		+\frac{\alpha\beta^3 L^2}{N}\cE_k
		+\frac{\alpha^2\beta^2 L\sigma^2}{2N}.
	\end{align}
	Recursions \eqref{hat_d_bnd} and \eqref{e_bnd} are used unchanged. For notational convenience, let
	\begin{subequations}\label{ABC_constants}
		\begin{align}
			A_1
			&\define
			\frac{4\alpha\beta^3L^2}{N},
			\quad
			A_2
			\define
			\frac{\alpha\beta^3L^2}{N},
			\quad
			A_3
			\define
			\frac{\alpha^2\beta^2L\sigma^2}{2N},
			\\
			B_1
			&\define
			\frac{52\beta\underline{\Delta}_{\lambda}}
			{\nu\theta\Delta_{\lambda}}
			+
			\frac{96\alpha^2\beta^2L^2}
			{\nu\theta\gamma\Delta_{\lambda}},
			\\
			B_2
			&\define
			\frac{10\alpha^4\beta NL^2}
			{\nu\theta\gamma^2\Delta_{\lambda}^2},
			\quad
			B_3
			\define
			\frac{10\alpha^4\beta L^2\sigma^2}
			{\nu\theta\gamma^2\Delta_{\lambda}^2}
			+
			\frac{12\alpha^2N\sigma^2}{\nu},
			\\
			C_1
			&\define
			\frac{2(1-\delta)}{\eta\delta}
			\left(
			24\beta^2\gamma^2\underline{\Delta}_{\lambda}^2
			+
			24\alpha^2\beta^2L^2
			+
			48\beta\gamma\underline{\Delta}_{\lambda}
			\right),
			\\
			C_2
			&\define
			\frac{12\alpha^2N(1-\delta)}{\eta\delta},
			\quad
			C_3
			\define
			\frac{6(1-\delta)}{\eta\delta}\alpha^2N\sigma^2,
		\end{align}
	\end{subequations}
	where $
			\Delta_{\lambda}
			\define 1-\lambda$, $\underline{\Delta}_{\lambda}
			\define 1-\underline\lambda$, $\Delta_\theta\define
			1-\bar\lambda_\theta
			=
			\frac{1-\lambda_{\theta}}{2}$, and $\lambda_{\theta}
			=
			\sqrt{1-\theta\gamma \Delta_{\lambda}}$.
	With this notation, recursions \eqref{centroid_phi_bnd2}, \eqref{hat_d_bnd}, and \eqref{e_bnd} become
	\begin{subequations} \label{compact_ineq}
		\begin{align}
			\Ex\tilde f(\bar{\phi}_e^{k+1})
			&\le
			\Ex\tilde f(\bar{\phi}_e^k)
			-\frac{\alpha\beta}{2}\cG_k
			-\frac{\alpha\beta}{4}\bar\cG_k
			+A_1\cD_k
			+A_2\cE_k
			+A_3,
			\label{AA}
			\\
			\cD_{k+1}
			&\le
			(1-\Delta_\theta)\cD_k
			+B_1\cE_k
			+B_2\bar\cG_k
			+B_3,
			\label{BB}
			\\
			\cE_{k+1}
			&\le
			(1-\mu_\eta)\cE_k
			+C_1\cD_k
			+C_2\cG_k
			+C_3,
			\label{CC}
		\end{align}
	\end{subequations}
	where $		\mu_\eta\define
	\frac{\eta\delta}{4}$. Note that since $1-\sqrt{1-a}\ge a/2$ for any $a\in[0,1]$, we have
	\begin{equation}\label{Delta_lower}
		\Delta_\theta
		=
		\frac{1-\lambda_\theta}{2}
		=
		\frac{1-\sqrt{1-\theta\gamma\Delta_\lambda}}{2}
		\ge
		\frac{\theta\gamma\Delta_\lambda}{4},
		\quad
		\text{and hence}
		\quad
		\frac1{\Delta_\theta}
		\le
		\frac{4}{\theta\gamma\Delta_\lambda}.
	\end{equation}
	
\subsection{Proof of Theorem \ref{thm_convergence}}
	We consider the Lyapunov function
	\begin{align}
		\cL_k\define\Ex\tilde f(\bar\phi_e^k)+\chi_1\cD_k+\chi_2\cE_k
	\end{align} 
	where $\chi_1,\chi_2>0$ will be specified shortly. It follows from \eqref{AA}--\eqref{CC} that
	\begin{align}\label{lyap_diff_fix}
		\cL_{k+1}-\cL_k
		&\le-\Big(\tfrac{\alpha\beta}{2}-\chi_2C_2\Big)\cG_k
		-\Big(\tfrac{\alpha\beta}{4}-\chi_1B_2\Big)\bar\cG_k
		+\big(A_1-\chi_1\Delta_\theta+\chi_2C_1\big)\cD_k\nonumber\\
		&\quad+\big(A_2+\chi_1B_1-\chi_2\mu_\eta\big)\cE_k
		+A_3+\chi_1B_3+\chi_2C_3 .
	\end{align}
	
	\medskip\noindent\textbf{\small Annihilate $\cD_k,\cE_k$.} To make the coefficients of $\cD_k$, $\cE_k$ nonpositive we choose
	\begin{equation}\label{theta_choice_fix}
		\chi_1\define\frac{2A_1}{\Delta_\theta}=\frac{8\alpha\beta^3L^2}{N\Delta_\theta},
		\quad
		\chi_2\define\frac{2(A_2+\chi_1B_1)}{\mu_\eta},
	\end{equation}
	so that
	\begin{align}
		A_2+\chi_1B_1-\chi_2\mu_\eta=-(A_2+\chi_1B_1)\le0.
	\end{align}
	Moreover,
	\begin{align}
		A_1-\chi_1\Delta_\theta+\chi_2C_1=-A_1+\chi_2C_1 \leq 0
	\end{align}
	when $\chi_2C_1\le A_1$. Using $A_2=A_1/4$ and \eqref{theta_choice_fix},
	\begin{align}\label{chi2C1_id}
		\chi_2C_1=\frac{2(A_2+\chi_1B_1)}{\mu_\eta} C_1
		=\frac{A_1}{\mu_\eta}\Big(\frac{C_1}{2}+\frac{4B_1C_1}{\Delta_\theta}\Big),
	\end{align}
	so $\chi_2C_1\le A_1$ if $\tfrac{C_1}{2}+\tfrac{4B_1C_1}{\Delta_\theta}\le\mu_\eta$. To satisfy this, we impose
\begin{align}\label{cond_D_fix}
			C_1&\le\frac{\mu_\eta}{2}, \quad
			 B_1C_1\le\frac{\Delta_\theta\mu_\eta}{8},
\end{align}
which can be satisfied for sufficiently small parameters, as specified below. Under \eqref{cond_D_fix},
	\begin{align}
		\frac{C_1}{2}+\frac{4B_1C_1}{\Delta_\theta}\le\frac{\mu_\eta}{4}+\frac{\mu_\eta}{2}=\frac{3\mu_\eta}{4}
		\quad\Longrightarrow\quad
		\chi_2C_1\le\frac{3A_1}{4}\le A_1 .
	\end{align}
	Finally, the second bound in \eqref{cond_D_fix} gives $\chi_1B_1=\tfrac{2A_1}{\Delta_\theta}B_1\le\tfrac{A_1\mu_\eta}{4C_1}=\tfrac{A_2\mu_\eta}{C_1}$, so that
	\begin{equation}\label{chi2_bound}
		\chi_2=\frac{2(A_2+\chi_1B_1)}{\mu_\eta}\le\frac{2A_2}{\mu_\eta}+\frac{2A_2}{C_1}.
	\end{equation}

\medskip\noindent\textbf{\small Negative gradient coefficients.}
To make the coefficients of the gradient terms in \eqref{lyap_diff_fix} negative, we impose
\[
\chi_2C_2\le\frac{\alpha\beta}{4},
\quad
\chi_1B_2\le\frac{\alpha\beta}{8}.
\]
With $A_2=\frac{\alpha\beta^3L^2}{N}$, $
C_2=\frac{12\alpha^2N(1-\delta)}{\eta\delta}$,
the bound \eqref{chi2_bound} splits $\chi_2C_2$ into two parts, each controlled separately.
For the first part,
\begin{align}\label{chi2C2a}
	\frac{2A_2}{\mu_\eta}C_2
	&=
	\frac{96\alpha^3\beta^3L^2(1-\delta)}
	{(\eta\delta)^2}
	\le
	\frac{\alpha\beta}{8}
\quad
	\Longleftrightarrow\quad
	\alpha\beta
	\le
	\frac{\eta\delta}
	{16\sqrt3\,L\sqrt{1-\delta}} .
\end{align}	
For the second part, note from $C_1
=
\frac{2(1-\delta)}{\eta\delta}
\left(
24\beta^2\gamma^2\underline{\Delta}_{\lambda}^2
+
24\alpha^2\beta^2L^2
+
48\beta\gamma\underline{\Delta}_{\lambda}
\right)$ that, keeping only the third term in $C_1$, we have
\[
C_1
\ge
\frac{96(1-\delta)\beta\gamma\underline\Delta_\lambda}
{\eta\delta}.
\]
Using this and $
A_2
=
\frac{\alpha\beta^3L^2}{N}$,
$C_2
=
\frac{12\alpha^2N(1-\delta)}{\eta\delta}$,
we obtain
\begin{align}\label{chi2C2b}
	\frac{2A_2}{C_1}C_2
	&\le
	\frac{
		2\left(\frac{\alpha\beta^3L^2}{N}\right)
		\left(\frac{12\alpha^2N(1-\delta)}{\eta\delta}\right)
	}
	{
		\frac{96(1-\delta)\beta\gamma\underline\Delta_\lambda}
		{\eta\delta}
	}
	=
	\frac{\alpha^3\beta^2 L^2}
	{4\gamma\underline\Delta_\lambda}
	\le
	\frac{\alpha\beta}{8} \nonumber \\
	&\quad
	\Longleftrightarrow
	\quad
	\alpha^2\beta
	\le
	\frac{\gamma\underline\Delta_\lambda}{2L^2}.
\end{align}
Together, \eqref{chi2C2a}--\eqref{chi2C2b} give $
\chi_2C_2\le\frac{\alpha\beta}{4}$. Next, by \eqref{theta_choice_fix} and $
B_2
=
\frac{10\alpha^4\beta NL^2}
{\nu\theta\gamma^2\Delta_\lambda^2}$, 
we have
\begin{align}\label{chi1B2}
	\chi_1B_2
	&=
	\frac{2A_1B_2}{\Delta_\theta}
	=
	\frac{80\alpha^5\beta^4L^4}
	{\nu\theta\gamma^2\Delta_\lambda^2\Delta_\theta}
	\le
	\frac{\alpha\beta}{8}
	\nonumber\\
	&\Longleftrightarrow\quad
	\alpha^4\beta^3
	\le
	\frac{\nu\theta\gamma^2\Delta_\lambda^2\Delta_\theta}
	{640L^4}.
\end{align}
By \eqref{Delta_lower}, we have $\Delta_\theta \ge \frac{\theta\gamma\Delta_\lambda}{4}$.  Therefore, \eqref{chi1B2} is guaranteed by the sufficient condition
 \begin{equation}\label{B2_sep_fix} 
 	\alpha^4\beta^3 \le \frac{\nu\theta^2\gamma^3\Delta_\lambda^3}{2560L^4}. 
 	\end{equation}

\medskip\noindent\textbf{\small Conditions to satisfy \eqref{cond_D_fix}.}
We now give sufficient conditions under which \eqref{cond_D_fix} holds.

\begin{enumerate}[(1)]
	\item \emph{$B_1C_1\le \Delta_\theta\mu_\eta/8$.}
	Using \eqref{Delta_lower} and $\mu_\eta=\eta\delta/4$, it suffices to impose $
	B_1C_1
	\le
	\frac{\theta\gamma\Delta_\lambda\eta\delta}{128}$.
	Recall that
	\begin{align*}
		B_1
		&=
		\frac{52\beta\underline\Delta_\lambda}
		{\nu\theta\Delta_\lambda}
		+
		\frac{96\alpha^2\beta^2L^2}
		{\nu\theta\gamma\Delta_\lambda}, \\
		C_1
		&=
		\frac{2(1-\delta)}{\eta\delta}
		\left(
		24\beta^2\gamma^2\underline\Delta_\lambda^2
		+
		24\alpha^2\beta^2L^2
		+
		48\beta\gamma\underline\Delta_\lambda
		\right).
	\end{align*}
	Write $
	B_1=B_1^0+B_1^\alpha$,
	where
	\[
	B_1^0
	\define
	\frac{52\beta\underline\Delta_\lambda}
	{\nu\theta\Delta_\lambda},
	\quad
	B_1^\alpha
	\define
	\frac{96\alpha^2\beta^2L^2}
	{\nu\theta\gamma\Delta_\lambda}.
	\]
	Imposing
	\begin{equation}\label{loopgain_step}
		\alpha^2\beta
		\le
		\frac{\gamma\underline\Delta_\lambda}{4L^2}
	\end{equation}
	gives $
	B_1^\alpha
	\le
	\frac{24\beta\underline\Delta_\lambda}
	{\nu\theta\Delta_\lambda}
	\le
	B_1^0$,
	and hence $
	B_1\le 2B_1^0$.
	Moreover, under \eqref{loopgain_step},
	\[
	24\alpha^2\beta^2L^2
	\le
	6\beta\gamma\underline\Delta_\lambda.
	\]
	If also $\beta\gamma\underline\Delta_\lambda\le1$, then
	\[
	24\beta^2\gamma^2\underline\Delta_\lambda^2
	\le
	24\beta\gamma\underline\Delta_\lambda.
	\]
	Therefore,
	\[
	C_1
	\le
	\frac{2(1-\delta)}{\eta\delta}
	\left(
	24+6+48
	\right)
	\beta\gamma\underline\Delta_\lambda
	=
	\frac{156(1-\delta)\beta\gamma\underline\Delta_\lambda}
	{\eta\delta}.
	\]
	Consequently,
	\[
	B_1C_1
	\le
	2B_1^0
	\cdot
	\frac{156(1-\delta)\beta\gamma\underline\Delta_\lambda}
	{\eta\delta}
	=
	\frac{
		16224(1-\delta)\beta^2\gamma\underline\Delta_\lambda^2
	}
	{\nu\theta\Delta_\lambda\eta\delta}.
	\]
	Thus it is sufficient that
	\[
	\frac{
		16224(1-\delta)\beta^2\gamma\underline\Delta_\lambda^2
	}
	{\nu\theta\Delta_\lambda\eta\delta}
	\le
	\frac{\theta\gamma\Delta_\lambda\eta\delta}{128}.
	\]
	This is guaranteed by
	\begin{equation}\label{B1C1_param}
		\beta
		\le
		\frac{
			\sqrt{\nu}\,\theta\Delta_\lambda\eta\delta
		}
		{
			832\sqrt3\,\underline\Delta_\lambda\sqrt{1-\delta}
		}.
	\end{equation}
	
	\item \emph{$C_1\le \mu_\eta/2$.}
	This condition is equivalent to
	\begin{align}
		24\beta^2\gamma^2\underline\Delta_\lambda^2
		+
		24\alpha^2\beta^2L^2
		+
		48\beta\gamma\underline\Delta_\lambda
		\le
		\frac{(\eta\delta)^2}{16(1-\delta)}.
	\end{align}
	When $\beta\gamma\underline\Delta_\lambda\le1$, we have
	\[
	24\beta^2\gamma^2\underline\Delta_\lambda^2
	\le
	24\beta\gamma\underline\Delta_\lambda.
	\]
	Hence it is sufficient that
	\[
	24\alpha^2\beta^2L^2
	+
	72\beta\gamma\underline\Delta_\lambda
	\le
	\frac{(\eta\delta)^2}{16(1-\delta)}.
	\]
	Splitting the two terms equally, it suffices to impose
	\[
	24\alpha^2\beta^2L^2
	\le
	\frac{(\eta\delta)^2}{32(1-\delta)},
	\quad
	72\beta\gamma\underline\Delta_\lambda
	\le
	\frac{(\eta\delta)^2}{32(1-\delta)}.
	\]
	Equivalently,
	\begin{equation}\label{C1_sep_fix}
		\alpha\beta
		\le
		\frac{\eta\delta}
		{L\sqrt{768(1-\delta)}},
		\quad
		\beta\gamma
		\le
		\frac{(\eta\delta)^2}
		{2304(1-\delta)\underline\Delta_\lambda}.
	\end{equation}
	The first condition in \eqref{C1_sep_fix} matches \eqref{chi2C2a}.
\end{enumerate}

\medskip\noindent\textbf{\small Collected sufficient conditions.}
We gather every condition imposed so far and reduce them to a convenient
sufficient set. We assume $0<\eta\le1$, $0<\gamma\le1$, and $0<\beta\le\theta$.
The bound $\gamma\le1$ guarantees that $\W_\gamma=(1-\gamma)\I+\gamma\W$ is a
convex combination of $\I$ and $\W$, so that $\|\W_\gamma\|\le1$, as used in
\eqref{x_bar_phi_e_bound}. Besides the descent bound $\alpha\beta L\le\tfrac12$,
all the conditions imposed are
\begin{equation*}
	\eqref{rot_cond_comp}:\quad
	\gamma\underline\Delta_\lambda
	<
	\frac{4\beta}{(\beta+\theta)^2},
	\quad
	\nu
	=
	1-\frac{(\beta+\theta)^2\gamma\underline\Delta_\lambda}{4\beta}
	>0,
\end{equation*}
\begin{equation*}
	\eqref{B1C1_param}:\quad
	\beta
	\le
	\frac{
		\sqrt{\nu}\,\theta\Delta_\lambda\eta\delta
	}
	{
		832\sqrt3\,\underline\Delta_\lambda\sqrt{1-\delta}
	},
\end{equation*}
\begin{equation*}
	\eqref{eq:gamma-beta-separated-condition}:\quad
	\beta\gamma
	\le
	\frac{\eta\delta}
	{8\sqrt3\,\underline\Delta_\lambda\sqrt{1-\delta}},
	\quad
	\eqref{eq:gamma-theta-separated-condition}:\quad
	\gamma\theta
	\le
	\frac{\eta\delta}
	{8\sqrt6\,\underline\Delta_\lambda\sqrt{1-\delta}},
\end{equation*}
\begin{equation*}
	\text{second bound of }\eqref{C1_sep_fix}:\quad
	\beta\gamma
	\le
	\frac{(\eta\delta)^2}
	{2304(1-\delta)\underline\Delta_\lambda},
\end{equation*}
the $\alpha$ stepsize bounds
\begin{equation*}
	\eqref{step_size_d_bound}:\quad
	\alpha\beta
	\le
	\frac{\sqrt{\nu}(1-\lambda_\theta)}{20L},
	\quad
	\eqref{chi2C2a}/\text{first of }\eqref{C1_sep_fix}:\quad
	\alpha\beta
	\le
	\frac{\eta\delta}
	{16\sqrt3\,L\sqrt{1-\delta}},
\end{equation*}
\begin{equation*}
	\eqref{eq:alpha-separated-condition}:\quad
	\alpha\beta
	\le
	\frac{\eta\delta}
	{4\sqrt6\,L\sqrt{1-\delta}},
	\quad
	\eqref{chi2C2b}:\quad
	\alpha^2\beta
	\le
	\frac{\gamma\underline\Delta_\lambda}{2L^2},
\end{equation*}
\begin{equation*}
	\eqref{B2_sep_fix}:\quad
	\alpha\beta
	\le
	\left(
	\frac{\nu\beta\theta^2\gamma^3\Delta_\lambda^3}
	{2560L^4}
	\right)^{1/4},
\end{equation*}
and the two auxiliary conditions used to establish \eqref{cond_D_fix},
\begin{equation*}
	\eqref{loopgain_step}:\quad
	\alpha^2\beta\le\frac{\gamma\underline\Delta_\lambda}{4L^2},
	\quad
	\beta\gamma\underline\Delta_\lambda\le1.
\end{equation*}

\noindent \emph{The two auxiliary conditions are automatic.} Multiplying
\eqref{rot_cond_comp} by $\beta$ and using $\theta\ge\beta$ gives
\[
\beta\gamma\underline\Delta_\lambda
<
\frac{4\beta^2}{(\beta+\theta)^2}
\le1,
\]
which is the second auxiliary condition. Moreover, if
\[
\alpha\beta
\le
\frac{\sqrt{\nu}\theta\gamma\Delta_\lambda}{40L},
\]
then
\[
\alpha^2\beta
=
\frac{(\alpha\beta)^2}{\beta}
\le
\frac{\nu\theta^2\gamma^2\Delta_\lambda^2}{1600L^2\beta}.
\]
Using \eqref{rot_cond_comp}, we have
\[
\gamma
<
\frac{4\beta}{(\beta+\theta)^2\underline\Delta_\lambda}.
\]
Therefore,
\[
\alpha^2\beta
\le
\frac{\nu\theta^2\gamma\Delta_\lambda^2}
{400L^2(\beta+\theta)^2\underline\Delta_\lambda}
\le
\frac{\gamma\underline\Delta_\lambda}{400L^2}
\le
\frac{\gamma\underline\Delta_\lambda}{4L^2},
\]
where we used $\theta^2\le(\beta+\theta)^2$ and
$\Delta_\lambda\le\underline\Delta_\lambda$. Hence \eqref{loopgain_step} holds.
The same argument also implies the weaker condition \eqref{chi2C2b}. Thus the
auxiliary conditions hold automatically whenever \eqref{rot_cond_comp} and
$\alpha\beta\le\sqrt{\nu}\theta\gamma\Delta_\lambda/(40L)$ hold.

\noindent \emph{The $\beta\gamma$ and $\gamma\theta$ constraints.} Since
$(\beta+\theta)^2\ge\theta^2$, the condition \eqref{rot_cond_comp} gives
\begin{equation}\label{rot_consequences_new}
	\gamma\theta
	\le
	\frac{4\beta}{\theta\underline\Delta_\lambda},
	\quad
	\beta\gamma
	\le
	\frac{4\beta^2}{\theta^2\underline\Delta_\lambda}.
\end{equation}
Let the upper bound in \eqref{B1C1_param} be
\[
B
\define
\frac{
	\sqrt{\nu}\,\theta\Delta_\lambda\eta\delta
}
{
	832\sqrt3\,\underline\Delta_\lambda\sqrt{1-\delta}
}.
\]
If $\beta\le B$, then \eqref{rot_consequences_new} gives
\begin{equation}\label{rot_param_bounds_new}
	\gamma\theta
	\le
	\frac{4B}{\theta\underline\Delta_\lambda}
	=
	\frac{
		\sqrt{\nu}\,\Delta_\lambda\eta\delta
	}
	{
		208\sqrt3\,\underline\Delta_\lambda^2\sqrt{1-\delta}
	},
\end{equation}
and
\begin{equation}\label{rot_param_bounds_new_beta_gamma}
	\beta\gamma
	\le
	\frac{4B^2}{\theta^2\underline\Delta_\lambda}
	=
	\frac{
		4\nu\Delta_\lambda^2(\eta\delta)^2
	}
	{
		832^2\cdot 3\,(1-\delta)\underline\Delta_\lambda^3
	}.
\end{equation}
Since $\nu\le1$, $\Delta_\lambda\le\underline\Delta_\lambda$, and
$208\sqrt3\ge8\sqrt6$, it follows from \eqref{rot_param_bounds_new} that
\[
\gamma\theta
\le
\frac{\sqrt{\nu}\,\Delta_\lambda\eta\delta}
{208\sqrt3\,\underline\Delta_\lambda^2\sqrt{1-\delta}}
\le
\frac{\eta\delta}
{208\sqrt3\,\underline\Delta_\lambda\sqrt{1-\delta}}
\le
\frac{\eta\delta}
{8\sqrt6\,\underline\Delta_\lambda\sqrt{1-\delta}},
\]
and hence \eqref{eq:gamma-theta-separated-condition} holds. Since
$\beta\le\theta$ and $\sqrt6\ge\sqrt3$,
\[
\beta\gamma
\le
\gamma\theta
\le
\frac{\eta\delta}
{8\sqrt6\,\underline\Delta_\lambda\sqrt{1-\delta}}
\le
\frac{\eta\delta}
{8\sqrt3\,\underline\Delta_\lambda\sqrt{1-\delta}},
\]
so \eqref{eq:gamma-beta-separated-condition} holds. Finally, since $\nu\le1$,
$\Delta_\lambda\le\underline\Delta_\lambda$, and
$9216\le832^2\cdot3$,
\[
\frac{4\nu\Delta_\lambda^2}
{832^2\cdot 3\,\underline\Delta_\lambda^3}
\le
\frac{4}{832^2\cdot 3\,\underline\Delta_\lambda}
\le
\frac{1}{2304\,\underline\Delta_\lambda},
\]
so \eqref{rot_param_bounds_new_beta_gamma} implies the second bound of
\eqref{C1_sep_fix},
\[
\beta\gamma
\le
\frac{(\eta\delta)^2}
{2304(1-\delta)\underline\Delta_\lambda}.
\]
Thus \eqref{B1C1_param} and \eqref{rot_cond_comp} imply all the remaining
$\beta\gamma$ and $\gamma\theta$ constraints.

\noindent \emph{The $\alpha$-conditions.} The condition \eqref{rot_cond_comp}
gives
\[
\theta\gamma\underline\Delta_\lambda
<
\frac{4\beta\theta}{(\beta+\theta)^2}
\le1,
\]
and since $\Delta_\lambda\le\underline\Delta_\lambda$, also
$\theta\gamma\Delta_\lambda<1$. Consequently, the descent bound
$\alpha\beta\le1/(2L)$ is implied by
\[
\alpha\beta
\le
\frac{\sqrt{\nu}\theta\gamma\Delta_\lambda}{40L}.
\]
The same condition implies \eqref{step_size_d_bound}: since
$1-\lambda_\theta=1-\sqrt{1-\theta\gamma\Delta_\lambda}
\ge \theta\gamma\Delta_\lambda/2$, we have
\[
\frac{\sqrt{\nu}\theta\gamma\Delta_\lambda}{40L}
\le
\frac{\sqrt{\nu}(1-\lambda_\theta)}{20L}.
\]
It also implies \eqref{chi2C2b} and \eqref{loopgain_step}, as shown above.

It remains to check \eqref{chi2C2a} and \eqref{eq:alpha-separated-condition}.
From \eqref{B1C1_param} and \eqref{rot_cond_comp}, we already showed that
\[
\gamma\theta
\le
\frac{
	\sqrt{\nu}\,\Delta_\lambda\eta\delta
}
{
	208\sqrt3\,\underline\Delta_\lambda^2\sqrt{1-\delta}
}.
\]
Thus
\[
\theta\gamma\Delta_\lambda
\le
\frac{
	\sqrt{\nu}\,\Delta_\lambda^2\eta\delta
}
{
	208\sqrt3\,\underline\Delta_\lambda^2\sqrt{1-\delta}
}
\le
\frac{
	\sqrt{\nu}\eta\delta
}
{
	208\sqrt3\,\sqrt{1-\delta}
}.
\]
Therefore,
\[
\frac{\sqrt{\nu}\theta\gamma\Delta_\lambda}{40L}
\le
\frac{\nu\eta\delta}
{8320\sqrt3\,L\sqrt{1-\delta}}
\le
\frac{\eta\delta}
{16\sqrt3\,L\sqrt{1-\delta}},
\]
which implies \eqref{chi2C2a}. Since $16\sqrt3>4\sqrt6$, \eqref{chi2C2a}
also implies \eqref{eq:alpha-separated-condition}.

Therefore, we can impose the following simplified sufficient conditions:
\begin{subequations}\label{final_conditions_fix}
	\begin{align}
		&0<\eta\le1,\quad
		0<\gamma\le1,\quad
		\theta>0,
		\\
		&\gamma\underline\Delta_\lambda
		<
		\frac{4\beta}{(\beta+\theta)^2},
		\quad
		\nu
		\define
		1-\frac{(\beta+\theta)^2\gamma\underline\Delta_\lambda}{4\beta}
		>0, \label{cond_final_rotation}
		\\
		&0<\beta
		\le \min\left\{
		\frac{
			\sqrt{\nu}\,\theta\Delta_\lambda\eta\delta
		}
		{
			832\sqrt3\,\underline\Delta_\lambda\sqrt{1-\delta}
		},
		\theta
		\right\},
		\\
		&\alpha\beta
		\le
		\min\left\{
		\frac{\sqrt{\nu}\theta\gamma\Delta_\lambda}{40L},
		\left(
		\frac{\nu\beta\theta^2\gamma^3\Delta_\lambda^3}
		{2560L^4}
		\right)^{1/4}
		\right\}. \label{final_conditions_alpha}
	\end{align}
\end{subequations}

\medskip\noindent\textbf{\small Telescoping.}
Under \eqref{final_conditions_fix}, \eqref{lyap_diff_fix} is bounded above by
\begin{align}
	\cL_{k+1}-\cL_k
	\le
	-\frac{\alpha\beta}{4}\cG_k
	-\frac{\alpha\beta}{8}\bar\cG_k
	+
	(A_3+\chi_1B_3+\chi_2C_3).
\end{align}
Rearranging, summing over $k=0,\dots,K-1$, and using $\cL_K\ge0$ gives
\begin{equation}\label{avg_fix}
	\frac{1}{2K}\sum_{k=0}^{K-1}
	\Ex\|\overline{\nabla f}(\x^k)\|^2
	+
	\frac1K\sum_{k=0}^{K-1}
	\Ex\|\nabla f(\bar\phi_e^k)\|^2
	\le
	\frac{4\cL_0}{\alpha\beta K}
	+
	\frac{4(A_3+\chi_1B_3+\chi_2C_3)}{\alpha\beta}. 
\end{equation}

\medskip\noindent\textbf{\small Noise floor.}
We now find the order of the noise terms. Recall the definitions
\eqref{ABC_constants} and \eqref{theta_choice_fix}. The first two terms are
\begin{align}
	\frac{4A_3}{\alpha\beta}
	&\overset{\eqref{ABC_constants}}{=}
	\frac{2\alpha\beta L\sigma^2}{N}
	=
	\cO\left(\frac{\alpha\beta L\sigma^2}{N}\right),
	\\
	\frac{4\chi_1B_3}{\alpha\beta}
	&\overset{\eqref{ABC_constants}\eqref{theta_choice_fix}}{=}
	\frac{320\alpha^4\beta^3L^4\sigma^2}
	{N\nu\Delta_\theta\theta\gamma^2\Delta_{\lambda}^2}
	+
	\frac{384\alpha^2\beta^2L^2\sigma^2}
	{\nu\Delta_\theta}
	\nonumber\\
	&\overset{\eqref{Delta_lower}}{\leq}
	\frac{1280\alpha^4\beta^3L^4\sigma^2}
	{N\nu\theta^2\gamma^3\Delta_{\lambda}^3}
	+
	\frac{1536\alpha^2\beta^2L^2\sigma^2}
	{\nu\theta\gamma \Delta_{\lambda}}.
\end{align}
For the last term, we use
\[
\chi_2
\overset{\eqref{chi2_bound}}{\le}
\frac{2A_2}{\mu_\eta}
+
\frac{2A_2}{C_1}.
\]
The first term gives
\[
\frac{4}{\alpha\beta}
\cdot
\frac{2A_2}{\mu_\eta}C_3
=
\frac{192(1-\delta)\alpha^2\beta^2L^2\sigma^2}
{\eta^2\delta^2}.
\]
For the second term, using the updated definition of $C_1$, we have
\[
C_1
\ge
\frac{2(1-\delta)}{\eta\delta}
\cdot
48\beta\gamma\underline\Delta_\lambda
=
\frac{96(1-\delta)\beta\gamma\underline\Delta_\lambda}
{\eta\delta}.
\]
Consequently,
\begin{align}\label{chi2C3_fix}
	\frac{4\chi_2C_3}{\alpha\beta}
	&\le
	\frac{192(1-\delta)\alpha^2\beta^2L^2\sigma^2}
	{\eta^2\delta^2}
	+
	\frac{\alpha^2\beta^2L^2\sigma^2}
	{2\beta\gamma\underline\Delta_\lambda}.
\end{align}
Using the previous estimates in \eqref{avg_fix}, and multiplying by $2$, we obtain
\begin{align}
	&\frac{1}{K}\sum_{k=0}^{K-1}
	\left(
	\Ex\|\overline{\nabla f}(\x^k)\|^2
	+
	\Ex\|\nabla f(\bar\phi_e^k)\|^2
	\right)
	\nonumber\\
	&\le
	\frac{8\cL_0}{\alpha\beta K}
	+
	\frac{4\alpha\beta L\sigma^2}{N}
	+
	\alpha^2\beta^2L^2\sigma^2
	\left[
	\frac{3072}{\nu\theta\gamma\Delta_{\lambda}}
	+
	\frac{384(1-\delta)}{\eta^2\delta^2}
	+
	\frac{1}{\beta\gamma\underline\Delta_\lambda}
	\right]
	\nonumber\\
	&\quad
	+
	\frac{2560\alpha^4\beta^4L^4\sigma^2}
	{N\nu\theta^2\beta\gamma^3\Delta_{\lambda}^3}.
	\label{avg_fix2}
\end{align}
We now bound the initial term $\cL_0$. With
$\bm{\upphi}^{0}=\x^{0}=\one \otimes x^{0}$, $\y^0=0$, and $\e^0=0$, we have
\begin{subequations}\label{init}
	\begin{align}
		\Q^{\top}\bm{\upphi}_e^{0}
		&=
		\Q^{\top}(\x^{0}+\beta\e^{0})
		=
		0,
		\quad
		\bar{\phi}_e^{0}=x^{0},
		\\
		\z^{0}
		&=
		\alpha\grad\f(\x^{0}).
	\end{align}
\end{subequations}
Hence,
\begin{equation}
	\d^{0}
	=
	\V^{-1}
	\begin{bmatrix}
		\Q^{\top}\bm{\upphi}_e^{0}
		\\[2mm]
		\Q^{\top}\widehat{\z}^{0}
	\end{bmatrix}
	=
	\V^{-1}
	\begin{bmatrix}
		0
		\\[2mm]
		\Q^{\top}\z^{0}
		+
		\Q^{\top}(\I-\W_{\theta,\gamma}) \e^{0}
	\end{bmatrix}
	=
	\alpha\V_r^{-1}\Q^{\top}\grad\f(\x^{0}),
	\label{d0}
\end{equation}
and
\begin{equation}
	\cL_0
	=
	\tilde{f}(x^0)
	+
	\chi_1\|\d^{0}\|^2 .
	\label{L0}
\end{equation}
Since
\begin{equation}
	\frac{8\chi_1}{\alpha\beta K}
	=
	\frac{8}{\alpha\beta K}
	\cdot
	\frac{8\alpha\beta^{3}L^{2}}{N\Delta_\theta}
	=
	\frac{64 \beta^{2}L^{2}}{N\Delta_\theta K},
\end{equation}
we have
\begin{align}
	\frac{8\cL_0}{\alpha\beta K}
	&=
	\frac{8}{\alpha\beta K}\tilde{f}(x^0)
	+
	\frac{64 \alpha^2\beta^{2}L^{2}}{N\Delta_\theta K}
	\|\V_r^{-1}\Q^{\top}\grad\f(\x^{0})\|^2
	\nonumber\\
	&\le
	\frac{8}{\alpha\beta K}\tilde{f}(x^0)
	+
	\frac{64 \alpha^2\beta^{2}v_r^2L^{2}}{N\Delta_\theta K}
	\|\Q^{\top}\grad\f(\x^{0})\|^2
	\nonumber\\
	&\overset{\eqref{Delta_lower}\eqref{V_left_right_bound}}{\leq}
	\frac{8}{\alpha\beta K}\tilde{f}(x^0)
	+
	\frac{128 \alpha^2\beta^2 L^{2}}
	{\nu\theta\beta\gamma^2 \Delta_{\lambda}^2 K}
	\varsigma_0^2,
	\label{full}
\end{align}
where 
$
\varsigma_0^2
\define
\frac1N
\|
\Q^\top\nabla\mathbf f(\mathbf x^0)
\|^2
=
\frac1N
\left\|
\nabla\mathbf f(\mathbf x^0)
-
\mathbf 1\otimes\nabla f(x^0)
\right\|^2
=
\frac1N
\sum_{i=1}^N
\left\|
\nabla f_i(x^0)-\nabla f(x^0)
\right\|^2$.
Substituting \eqref{full} into \eqref{avg_fix2}, we obtain
\begin{align}
	\frac{1}{K}\sum_{k=0}^{K-1}
	\left(
	\Ex\|\overline{\nabla f}(\x^k)\|^2
	+
	\Ex\|\nabla f(\bar\phi_e^k)\|^2
	\right)
	&\le
	\frac{\Psi_0}{\alpha\beta K}
	+
	\Psi_1\alpha\beta
	+
	\left(\Psi_2+\frac{\widetilde\Psi_2}{K}\right)
	\alpha^2\beta^2
	+
	\Psi_3\alpha^4\beta^4,
	\label{avg_fix2_psi}
\end{align}
where
\begin{subequations}\label{Psi_defs_exact}
	\begin{align}
		\Psi_0
		&\define
		8\tilde f(x^0)
		=
		\cO\left(\tilde f(x^0)\right),
		\label{Psi0_def_exact}
		\\
		\Psi_1
		&\define
		\frac{4L\sigma^2}{N}
		=
		\cO\left(\frac{L\sigma^2}{N}\right),
		\label{Psi1_def_exact}
		\\
		\Psi_2
		&\define
		L^2\sigma^2
		\left[
		\frac{3072}{\nu\theta\gamma\Delta_{\lambda}}
		+
		\frac{384(1-\delta)}{\eta^2\delta^2}
		+
		\frac{1}{\beta\gamma\underline\Delta_\lambda}
		\right]
		\nonumber\\
		&=
		\cO\left(
		L^2\sigma^2
		\left[
		\frac{1}{\nu\theta\gamma\Delta_{\lambda}}
		+
		\frac{1-\delta}{\eta^2\delta^2}
		+
		\frac{1}{\beta\gamma\underline\Delta_\lambda}
		\right]
		\right),
		\label{Psi2_def_exact}
		\\
		\widetilde\Psi_2
		&\define
		\frac{128L^2\varsigma_0^2}
		{\nu\theta\beta\gamma^2\Delta_{\lambda}^2}
		=
		\cO\left(
		\frac{L^{2}\varsigma_0^2}
		{\nu\theta\beta\gamma^2\Delta_{\lambda}^2}
		\right),
		\label{tildePsi2_def_exact}
		\\
		\Psi_3
		&\define
		\frac{2560L^4\sigma^2}
		{N\nu\theta^2\beta\gamma^3\Delta_{\lambda}^3}
		=
		\cO\left(
		\frac{L^4\sigma^2}
		{N\nu\theta^2\beta\gamma^3\Delta_{\lambda}^3}
		\right).
		\label{Psi3_def_exact}
	\end{align}
\end{subequations}

\subsection{Proof of Corollary \ref{corr:rate}} \label{app_corr_step_rate}
 We apply Lemma~\ref{lem:stepsize_selection} to obtain a nearly optimal
 stepsize choice, following the balancing strategy used in
 \cite{koloskova2020unified,islamov2024towards,stich2019unified}. Let $
\tau\define \alpha\beta$ and
\[
\tau_{\max}
\define
\min\left\{
\frac{\sqrt{\nu}\theta\gamma\Delta_\lambda}{40L},
\left(
\frac{\nu\beta\theta^2\gamma^3\Delta_\lambda^3}
{2560L^4}
\right)^{1/4}
\right\}.
\]
Then \eqref{avg_fix2_psi} can be written as
\begin{align}
	\frac{1}{K}\sum_{k=0}^{K-1}
	\left(
	\Ex\|\overline{\nabla f}(\x^k)\|^2
	+
	\Ex\|\nabla f(\bar\phi_e^k)\|^2
	\right)
	&\le
	\frac{\Psi_0}{\tau K}
	+
	\Psi_1\tau
	+
	\left(\Psi_2+\frac{\widetilde\Psi_2}{K}\right)\tau^2
	+
	\Psi_3\tau^4.
	\label{avg_fix2_tau}
\end{align}
Applying Lemma~\ref{lem:stepsize_selection} with
\[
a_0=\Psi_0,
\quad
a_1=\Psi_1,
\quad
a_2=\Psi_2,
\quad
\widetilde a_2=\widetilde\Psi_2,
\quad
a_3=\Psi_3,
\quad
\bar\tau=\tau_{\max},
\]
we choose
\begin{equation}\label{tau_optimized_choice}
	\tau
	=
	\alpha\beta
	=
	\min\left\{
	\Bigl(\tfrac{\Psi_0}{\Psi_1K}\Bigr)^{1/2},
	\Bigl(\tfrac{\Psi_0}{(\Psi_2+\widetilde\Psi_2/K)K}\Bigr)^{1/3},
	\Bigl(\tfrac{\Psi_0}{\Psi_3K}\Bigr)^{1/5},
	\tau_{\max}
	\right\}.
\end{equation}
Then
\begin{equation}\label{eq:opt_rate_psi}
	\begin{aligned}
		\frac1K\sum_{k=0}^{K-1}
		\left(\Ex\|\overline{\nabla f}(\x^k)\|^2+\Ex\|\nabla f(\bar\phi_e^k)\|^2\right)
		&\le
		2\Bigl(\tfrac{\Psi_1\Psi_0}{K}\Bigr)^{1/2}
		+2\Psi_2^{1/3}\Bigl(\tfrac{\Psi_0}{K}\Bigr)^{2/3}
		+2\widetilde\Psi_2^{1/3}\tfrac{\Psi_0^{2/3}}{K}
		\\
		&\quad
		+2\Psi_3^{1/5}\Bigl(\tfrac{\Psi_0}{K}\Bigr)^{4/5}
		+\frac{\Psi_0}{\tau_{\max}K}.
	\end{aligned}
\end{equation}
Using the definitions in \eqref{Psi_defs_exact}, and hiding constants depending
only on $L$ and $\tilde f(x^0)$, we obtain
\begin{align}\label{optimized_rate_final_order_simple}
	&\frac{1}{K}\sum_{k=0}^{K-1}
	\left(
	\Ex\|\overline{\nabla f}(\x^k)\|^2
	+
	\Ex\|\nabla f(\bar\phi_e^k)\|^2
	\right)
	\nonumber\\
	&\le
	\cO\left(
	\sqrt{\frac{\sigma^2}{NK}}
	\right)
	+
	\cO\left(
	\left[
	\frac{\sigma^2}{K^2}
	\left(
	\frac{1}{\nu\theta\gamma\Delta_{\lambda}}
	+
	\frac{1}{\eta^2\delta^2}
	+
	\frac{1}{\beta\gamma\underline\Delta_\lambda}
	\right)
	\right]^{1/3}
	\right)
	\nonumber\\
	&\quad
	+
	\cO\left(
	\left[
	\frac{\varsigma_0^2}
	{\nu\theta\beta\gamma^2\Delta_{\lambda}^2K^3}
	\right]^{1/3}
	\right)
	+
	\cO\left(
	\left[
	\frac{\sigma^2}
	{N\nu\theta^2\beta\gamma^3\Delta_{\lambda}^3K^4}
	\right]^{1/5}
	\right)
	+
	\cO\left(
	\frac{1}{\tau_{\max}K}
	\right).
\end{align}
Using the heavy-compression scaling $\nu,\theta,\eta=\Theta(1)$,
\[
\beta
\asymp
\frac{\delta\Delta_\lambda}{\underline\Delta_\lambda},
\quad
\gamma
\asymp
\frac{\delta\Delta_\lambda}{\underline\Delta_\lambda^2},
\quad
1-\delta=\Theta(1),
\]
we have
\[
\frac{1}{\tau_{\max}}
=
\cO\left(
\frac{\underline\Delta_\lambda^2}
{\delta\Delta_\lambda^2}
\right).
\]
Moreover,
\[
\frac{1}{\nu\theta\gamma\Delta_\lambda}
+
\frac{1}{\eta^2\delta^2}
+
\frac{1}{\beta\gamma\underline\Delta_\lambda}
=
\cO\left(
\frac{\underline\Delta_\lambda^2}
{\delta^2\Delta_\lambda^2}
\right),
\]
and
\[
\frac{1}{\nu\theta\beta\gamma^2\Delta_\lambda^2}
=
\cO\left(
\frac{\underline\Delta_\lambda^5}
{\delta^3\Delta_\lambda^5}
\right),
\quad
\frac{1}{\nu\theta^2\beta\gamma^3\Delta_\lambda^3}
=
\cO\left(
\frac{\underline\Delta_\lambda^7}
{\delta^4\Delta_\lambda^7}
\right).
\]
Therefore, the optimized rate is
\begin{equation}\label{eq:opt_rate_final}
	\begin{aligned}
		&\frac1K\sum_{k=0}^{K-1}
		\left(\Ex\|\overline{\nabla f}(\x^k)\|^2+\Ex\|\nabla f(\bar\phi_e^k)\|^2\right)
		\\
		&\le
		\cO\left(\sqrt{\tfrac{\sigma^2}{NK}}\right)
		+
		\cO\left(
		\Bigl(\tfrac{\sigma\,\underline\Delta_\lambda}{\delta\Delta_\lambda}\Bigr)^{2/3}
		\frac{1}{K^{2/3}}
		\right)
		+
		\cO\left(
		\Bigl(\tfrac{\sigma^2}{N}\Bigr)^{1/5}
		\Bigl(\tfrac{\underline\Delta_\lambda}{\Delta_\lambda}\Bigr)^{7/5}
		\tfrac{1}{\delta^{4/5}K^{4/5}}
		\right)
		\\
		&\quad
		+
		\cO\left(
		\varsigma_0^{2/3}
		\Bigl(\tfrac{\underline\Delta_\lambda}{\Delta_\lambda}\Bigr)^{5/3}
		\frac{1}{\delta K}
		\right)
		+
		\cO\left(
		\frac{\underline\Delta_\lambda^2}{\delta\Delta_\lambda^2}
		\frac{1}{K}
		\right).
	\end{aligned}
\end{equation}

\noindent\textbf{\small Transient time.}
Assume \(\sigma>0\). The leading stochastic term in
\eqref{eq:opt_rate_final} is of order
\(\sigma/\sqrt{NK}\). We therefore compare each lower-order term with this
quantity.

First, the \(K^{-2/3}\) term gives
\[
\left(
\frac{\sigma\underline\Delta_\lambda}
{\delta\Delta_\lambda}
\right)^{2/3}
\frac{1}{K^{2/3}}
\asymp
\frac{\sigma}{\sqrt{NK}} \quad \Longrightarrow \quad K_{\mathrm{tr},1}
\asymp
\frac{
	N^3\underline\Delta_\lambda^4
}{
	\sigma^2\delta^4\Delta_\lambda^4
}.
\]
Second, the \(K^{-4/5}\) term gives
\[
\left(\frac{\sigma^2}{N}\right)^{1/5}
\left(
\frac{\underline\Delta_\lambda}{\Delta_\lambda}
\right)^{7/5}
\frac{1}{\delta^{4/5}K^{4/5}}
\asymp
\frac{\sigma}{\sqrt{NK}} \quad \Longrightarrow \quad K_{\mathrm{tr},2}
\asymp
\frac{
	N\underline\Delta_\lambda^{14/3}
}{
	\sigma^2\delta^{8/3}\Delta_\lambda^{14/3}
}.
\]
Third, the \(\varsigma_0\)-dependent \(K^{-1}\) term gives
\[
\varsigma_0^{2/3}
\left(
\frac{\underline\Delta_\lambda}{\Delta_\lambda}
\right)^{5/3}
\frac{1}{\delta K}
\asymp
\frac{\sigma}{\sqrt{NK}} \quad \Longrightarrow \quad K_{\mathrm{tr},3}
\asymp
\frac{
	N\varsigma_0^{4/3}\underline\Delta_\lambda^{10/3}
}{
	\sigma^2\delta^2\Delta_\lambda^{10/3}
}.
\]
Finally,
\[
\frac{\underline\Delta_\lambda^2}
{\delta\Delta_\lambda^2K}
\asymp
\frac{\sigma}{\sqrt{NK}} \quad \Longrightarrow \quad K_{\mathrm{tr},4}
\asymp
\frac{
	N\underline\Delta_\lambda^4
}{
	\sigma^2\delta^2\Delta_\lambda^4
}.
\]
Since \(N\ge1\) and \(0<\delta\le1\), $
K_{\mathrm{tr},4}
\le
K_{\mathrm{tr},1}$.
Taking the maximum of the remaining thresholds gives
\eqref{Ktr_dom_explicit_opt}.

\noindent \textbf{\small Standard stepsize choice.} Here, we derive the rate given in Remark \ref{remark_standard_step} when the stepsize is chosen in a more natural way.
Recall that the final admissible upper bound on the product $\alpha\beta$ is
\begin{equation}\label{tau_max_def_simple}
	\tau_{\max}
	\define
	\min\left\{
	\frac{\sqrt{\nu}\theta\gamma\Delta_\lambda}{40L},
	\left(
	\frac{\nu\beta\theta^2\gamma^3\Delta_\lambda^3}{2560L^4}
	\right)^{1/4}
	\right\}.
\end{equation}
Let $\tau \define \alpha\beta$ and choose $
\tau
=
\sqrt{\frac{N}{K}}$.
We assume that $K$ is sufficiently large so that $
K
\ge
\frac{N}{\tau_{\max}^2}$.
Then $\tau\le \tau_{\max}$, and hence the stepsize condition
\eqref{final_conditions_alpha} holds. Moreover,
\begin{equation}\label{tau_powers_simple}
	\frac{1}{\tau K}
	=
	\frac{1}{\sqrt{NK}},
	\quad
	\frac{\tau\sigma^2}{N}
	=
	\frac{\sigma^2}{\sqrt{NK}},
	\quad
	\tau^2
	=
	\frac{N}{K},
	\quad
	\tau^4
	=
	\frac{N^2}{K^2}.
\end{equation}
In the following, the constants hidden in $\cO(\cdot)$ may depend on
$\tilde f(x^0)$ and $L$, but are independent of
$K,N,\delta,\lambda,\underline\lambda,\nu$. Under this choice, \eqref{avg_fix2_psi} gives
\begin{align}
	&\frac{1}{K}\sum_{k=0}^{K-1}
	\left(
	\Ex\|\overline{\nabla f}(\x^k)\|^2
	+
	\Ex\|\nabla f(\bar\phi_e^k)\|^2
	\right)
	\nonumber\\
	&\le
	\cO\left(\frac{1+\sigma^2}{\sqrt{NK}}\right)
	\nonumber\\
	&\quad
	+
	\cO\left(
	\frac{N\sigma^2}{K}
	\left[
	\frac{1}{\nu\theta\gamma\Delta_\lambda}
	+
	\frac{1-\delta}{\eta^2\delta^2}
	+
	\frac{1}{\beta\gamma\underline\Delta_\lambda}
	\right]
	\right)
	\nonumber\\
	&\quad
	+
	\cO\left(
	\frac{N\varsigma_0^2}
	{K^2\nu\theta\beta\gamma^2\Delta_\lambda^2}
	\right)
	+
	\cO\left(
	\frac{N\sigma^2}
	{K^2\nu\theta^2\beta\gamma^3\Delta_\lambda^3}
	\right).
	\label{rate_O_simple_before_scaling}
\end{align}
This form contains no division by $\sigma$, and the last two terms remain
$K^{-2}$ terms. Now take $\nu=\Theta(1)$ in \eqref{cond_final_rotation}. Since $\theta>0$
and $0<\eta\le1$ are free, choose
\begin{equation}\label{balanced_scaling_simple}
	\theta=\Theta(1),
	\quad
	\eta=\Theta(1),
	\quad
	\beta
	\asymp
	\frac{\delta\Delta_\lambda}{\underline\Delta_\lambda},
	\quad
	\gamma
	\asymp
	\frac{\beta}{\underline\Delta_\lambda}
	\asymp
	\frac{\delta\Delta_\lambda}{\underline\Delta_\lambda^2}.
\end{equation}
In the heavy-compression regime $1-\delta=\Theta(1)$, all powers of
$1-\delta$ are absorbed into the constants hidden in $\cO(\cdot)$.
The implicit constants in \eqref{balanced_scaling_simple} are chosen
sufficiently small so that the admissibility conditions in
\eqref{final_conditions_fix} hold. With \eqref{balanced_scaling_simple}, we have
\begin{equation}\label{tau_max_asymp_simple}
	\tau_{\max}
	\asymp
	\min\left\{
	\frac{\delta\Delta_\lambda^2}{\underline\Delta_\lambda^2},
	\frac{\delta\Delta_\lambda^{7/4}}{\underline\Delta_\lambda^{7/4}}
	\right\}
	\asymp
	\frac{\delta\Delta_\lambda^2}{\underline\Delta_\lambda^2},
	\quad
	\frac{1}{\tau_{\max}}
	\lesssim
	\frac{\underline\Delta_\lambda^2}
	{\delta\Delta_\lambda^2},
\end{equation}
where we used $\Delta_\lambda\le \underline\Delta_\lambda$. Hence the
large-$K$ condition $
K
\ge
\frac{N}{\tau_{\max}^2}$ is satisfied whenever $
K
\gtrsim
N
\frac{\underline\Delta_\lambda^4}
{\delta^2\Delta_\lambda^4}$.
Moreover,
\begin{equation}\label{scaling_terms_simple}
	\begin{aligned}
		\frac{1}{\nu\theta\gamma\Delta_\lambda}
		&\asymp
		\frac{\underline\Delta_\lambda^2}
		{\delta\Delta_\lambda^2}
		\le
		\cO\left(
		\frac{\underline\Delta_\lambda^2}
		{\delta^2\Delta_\lambda^2}
		\right),
		&
		\frac{1-\delta}{\eta^2\delta^2}
		&\asymp
		\frac{1}{\delta^2}
		\le
		\cO\left(
		\frac{\underline\Delta_\lambda^2}
		{\delta^2\Delta_\lambda^2}
		\right),
		\\
		\frac{1}{\beta\gamma\underline\Delta_\lambda}
		&\asymp
		\frac{\underline\Delta_\lambda^2}
		{\delta^2\Delta_\lambda^2},
		&
		\frac{1}{\nu\theta\beta\gamma^2\Delta_\lambda^2}
		&\asymp
		\frac{\underline\Delta_\lambda^5}
		{\delta^3\Delta_\lambda^5},
		\\
		\frac{1}{\nu\theta^2\beta\gamma^3\Delta_\lambda^3}
		&\asymp
		\frac{\underline\Delta_\lambda^7}
		{\delta^4\Delta_\lambda^7}.
	\end{aligned}
\end{equation}
Substituting \eqref{scaling_terms_simple} into
\eqref{rate_O_simple_before_scaling} gives
\begin{align}
	\frac{1}{K}\sum_{k=0}^{K-1}
	\left(
	\Ex\|\overline{\nabla f}(\x^k)\|^2
	+
	\Ex\|\nabla f(\bar\phi_e^k)\|^2
	\right)
	&\le
	\cO\left(\frac{1+\sigma^2}{\sqrt{NK}}\right)
	+
	\cO\left(
	\frac{N\sigma^2}{K}
	\cdot
	\frac{\underline\Delta_\lambda^2}
	{\delta^2\Delta_\lambda^2}
	\right)
	\nonumber\\
	&\quad
	+
	\cO\left(
	\frac{N\varsigma_0^2}{K^2}
	\cdot
	\frac{\underline\Delta_\lambda^5}
	{\delta^3\Delta_\lambda^5}
	\right)
	+
	\cO\left(
	\frac{N\sigma^2}{K^2}
	\cdot
	\frac{\underline\Delta_\lambda^7}
	{\delta^4\Delta_\lambda^7}
	\right).
	\label{final_rate_simple}
\end{align}

\section{P\L~convergence proof} \label{app_pl_proof}
We begin the proof by recalling the key inequality~\eqref{compact_ineq}, which we restate for the reader's convenience:
	\begin{subequations} \label{compact_ineq2}
	\begin{align}
		\Ex\tilde f(\bar{\phi}_e^{k+1})
		&\le
		\Ex\tilde f(\bar{\phi}_e^k)
		-\frac{\alpha\beta}{2}\cG_k
		-\frac{\alpha\beta}{4}\bar\cG_k
		+A_1\cD_k
		+A_2\cE_k
		+A_3,
		\label{AA2}
		\\
		\cD_{k+1}
		&\le
		(1-\Delta_\theta)\cD_k
		+B_1\cE_k
		+B_2\bar\cG_k
		+B_3,
		\label{BB2}
		\\
		\cE_{k+1}
		&\le
		(1-\mu_\eta)\cE_k
		+C_1\cD_k
		+C_2\cG_k
		+C_3.
		\label{CC2}
	\end{align}
\end{subequations}
where $\mu_{\eta}=\eta\delta/4$ and
	\begin{align} \label{def_G_D_E_caligraphic_repeat}
	\cG_k \define \Ex\|\nabla f(\bar{\phi}_e^k)\|^2,\quad
	\bar \cG_k \define \Ex\|\overline{\nabla f}(\x^k)\|^2,\quad
	\cD_k \define  \Ex\|\d^k\|^2,\quad
	\cE_k \define  \Ex\|\e^k\|^2 .
\end{align}
The coefficients \(A_i\), \(B_i\), and \(C_i\) (for \(i=1,2,3\)) are given by
	\begin{subequations}\label{ABC_constants_again}
	\begin{align}
		A_1
		&\define
		\frac{4\alpha\beta^3L^2}{N},
		\quad
		A_2
		\define
		\frac{\alpha\beta^3L^2}{N},
		\quad
		A_3
		\define
		\frac{\alpha^2\beta^2L\sigma^2}{2N},
		\\
		B_1
		&\define
		\frac{52\beta\underline{\Delta}_{\lambda}}
		{\nu\theta\Delta_{\lambda}}
		+
		\frac{96\alpha^2\beta^2L^2}
		{\nu\theta\gamma\Delta_{\lambda}},
		\\
		B_2
		&\define
		\frac{10\alpha^4\beta NL^2}
		{\nu\theta\gamma^2\Delta_{\lambda}^2},
		\quad
		B_3
		\define
		\frac{10\alpha^4\beta L^2\sigma^2}
		{\nu\theta\gamma^2\Delta_{\lambda}^2}
		+
		\frac{12\alpha^2N\sigma^2}{\nu},
		\\
		C_1
		&\define
		\frac{2(1-\delta)}{\eta\delta}
		\left(
		24\beta^2\gamma^2\underline{\Delta}_{\lambda}^2
		+
		24\alpha^2\beta^2L^2
		+
		48\beta\gamma\underline{\Delta}_{\lambda}
		\right),
		\\
		C_2
		&\define
		\frac{12\alpha^2N(1-\delta)}{\eta\delta},
		\quad
		C_3
		\define
		\frac{6(1-\delta)}{\eta\delta}\alpha^2N\sigma^2,
	\end{align}
\end{subequations}
where $
\Delta_{\lambda}
\define 1-\lambda$, $\underline{\Delta}_{\lambda}
\define 1-\underline\lambda$, $\Delta_\theta\define
1-\bar\lambda_\theta
=
\frac{1-\lambda_{\theta}}{2}$, and $\lambda_{\theta}
=
\sqrt{1-\theta\gamma \Delta_{\lambda}}$.
\subsection{Auxiliary results}
The proof relies on the following auxiliary lemma.
\begin{lemma}[\sc \small P\L~auxiliary bounds]\label{lem:pl_auxiliary}
	\rm
	Under Assumptions \ref{assump:smoothness} and \ref{assump_pl}, we have
		\begin{subequations}
		\begin{align}
		\cG_k&\le 2L\Ex\tilde f(\bar{\phi}_e^{k}) \label{pl_smooth_upper} \\		
		\bar\cG_k
		&\le
		4L\Ex\tilde f(\bar{\phi}_e^{k})+\frac{16\beta^2L^2}{N}\cD_k+\frac{4\beta^2L^2}{N}\cE_k \label{pl_smooth_upper2}.
		\end{align}
		\end{subequations}
\end{lemma}
\begin{proof}
	Using $L$-smoothness \eqref{L_smoothness} at $y=x-\frac1L\nabla f(x)$ and $z=x$,
	\[
	f^\star\le f\left(x-\tfrac1L\nabla f(x)\right)
	\le f(x)-\tfrac1{2L}\|\nabla f(x)\|^2 \quad \forall~x\in\mathbb{R}^n.
	\]
	Rearranging and taking $x=\bar\phi_e^k$ and expectations gives
	\eqref{pl_smooth_upper}.

	 Next, by the splitting
	$\overline{\nabla f}(\x^k)=\nabla f(\bar\phi_e^k)
	+\bigl(\overline{\nabla f}(\x^k)-\nabla f(\bar\phi_e^k)\bigr)$,
	Young's inequality and $L$-smoothness, we get
	\begin{align}
			\|\overline{\nabla f}(\x^k)\|^2
		&\le 2\|\nabla f(\bar\phi_e^k)\|^2
		+\frac{2L^2}{N}\|\x^k-\bar{\bm{\upphi}}_e^k\|^2 \\
		&\overset{\eqref{pl_smooth_upper}}{\leq} 4L\tilde f(\bar\phi_e^k) +\frac{2L^2}{N}\|\x^k-\bar{\bm{\upphi}}_e^k\|^2 .
	\end{align}
	 For the
	second, we use \eqref{x_bar_phi_e_bound}, \ie, $	\|\x^k-\bar{\bm{\upphi}}_e^k\|^2\leq
		8\beta^2\|\d^k\|^2
		+
		2\beta^2\|\e^k\|^2$ and taking expectations to arrive at \eqref{pl_smooth_upper2}.
\end{proof}

	The following proposition gives explicit sufficient conditions on the stepsizes that guarantee the P\L~convergence rate in Theorem \ref{thm:pl_rate}.
\begin{proposition}[\sc\small Sufficient conditions for the P\L~rate]
	\label{prop:pl_stepsize}
	\rm
All the requirements ensuring \eqref{compact_ineq2}---namely the descent bound $\alpha\beta L\le\tfrac{1}{2}$, \eqref{rot_cond_comp}, \eqref{step_size_d_bound}, and \eqref{eq:gamma-beta-separated-condition}--\eqref{eq:alpha-separated-condition}---hold whenever the stepsizes satisfy the following explicit bounds: $0<\gamma \leq 1$, $0<\eta \leq 1$,
	\begin{subequations}\label{PL_explicit}
		\begin{align}
			\frac{(\beta+\theta)^2\gamma\underline\Delta_\lambda}{4\beta}
			&\le\frac12,
			\quad
			\nu=1-\frac{(\beta+\theta)^2\gamma\underline\Delta_\lambda}{4\beta}\ge\frac12,
			\label{PL_exp_a}\\
			\beta\gamma\underline\Delta_\lambda
			&\le\min\left\{1,\ \frac{(\eta\delta)^2}{9216(1-\delta)}\right\},
			\label{PL_exp_b}\\
			\gamma\theta\underline\Delta_\lambda
			&\le\frac{\eta\delta}{96\sqrt{1-\delta}},
			\label{PL_exp_c}\\
			\beta
			&\le\frac{\theta\eta\delta\Delta_\lambda}{8192\,\underline\Delta_\lambda\sqrt{1-\delta}},
			\label{PL_exp_d}\\
			\alpha\beta
			&\le \min\left\{\frac{\theta\gamma\Delta_\lambda}{128L},\frac{\eta \delta}{4\mu}\right\},
			\label{PL_exp_e}\\
			\alpha^4\beta^3
			&\le\frac{\mu\theta^2\gamma^3\Delta_\lambda^3}{327680\,L^5},
			\label{PL_exp_f}\\
			\alpha^2\beta^2
			&\le\frac{\mu(\eta\delta)^2}{8192\,L^3(1-\delta)},
			\label{PL_exp_g}\\
			\alpha^2\beta^3
			&\le\frac{\mu\theta^2\gamma(\eta\delta)^2\Delta_\lambda^2}{2^{28}L^3\underline\Delta_\lambda}.
			\label{PL_exp_h}
		\end{align}
	\end{subequations}
	Moreover, these conditions simultaneously ensure that
	\begin{subequations}\label{PL_conditions}
		\begin{align}
			\frac{16\beta^2L^2B_2}{N}&\le\frac{\Delta_\theta}{8},
			\label{PL_cond_1}\\
			C_1\le\frac{\mu_\eta}{8},
			&\quad
			B_1C_1\le\frac{\mu_\eta\Delta_\theta}{64},
			\label{PL_cond_2}\\
			sC_1&\le\frac{A_1}{2},
			\label{PL_cond_3}\\
			r\widetilde B_2+s\widetilde C_2&\le\frac{\mu\alpha\beta}{2},
			\label{PL_cond_4}
		\end{align}
	\end{subequations}
	where
	\begin{align} \label{BC_tilde_s_r}
		\widetilde B_1\define B_1+\frac{4\beta^2L^2B_2}{N},
		\quad
		\widetilde B_2\define 4LB_2,
		\quad
		\widetilde C_2\define 2LC_2,
		\quad
		r\define\frac{4A_1}{\Delta_\theta},
		\quad
		s\define\frac{4(A_2+r\widetilde B_1)}{\mu_\eta}.
	\end{align}
\end{proposition}
\begin{proof}	
	\noindent	\emph{Base-recursion conditions.}
	Condition \eqref{PL_exp_a} implies \eqref{rot_cond_comp} and gives $\nu\ge\tfrac12$. Since $\nu\ge\tfrac12$ and
	$\Delta_\theta\ge\theta\gamma\Delta_\lambda/4$ (by \eqref{Delta_lower}),
	\begin{align}
		\frac{\sqrt\nu\,\Delta_\theta}{20}\ge\frac{\theta\gamma\Delta_\lambda}{80\sqrt2}\ge\frac{\theta\gamma\Delta_\lambda}{128},
	\end{align}
	so the bound $\alpha\beta\le\theta\gamma\Delta_\lambda/(128L)$ in \eqref{PL_exp_e}
	implies \eqref{step_size_d_bound}.
	By \eqref{PL_exp_b}, $\beta\gamma\underline\Delta_\lambda\le 1$ and
	\begin{align}
		\beta\gamma\underline\Delta_\lambda\le \frac{(\eta\delta)^2}{9216(1-\delta)}
		\le \frac{\eta\delta}{96\sqrt{1-\delta}}\le \frac{\eta\delta}{8\sqrt3\sqrt{1-\delta}}.
	\end{align}
	The middle step uses $\eta\delta/(96\sqrt{1-\delta})\le1$ on the active branch,
	else $\beta\gamma\underline\Delta_\lambda\le1\le\eta\delta/(96\sqrt{1-\delta})$, and thus,
	\eqref{eq:gamma-beta-separated-condition} holds. By \eqref{PL_exp_c},
	$\gamma\theta\underline\Delta_\lambda\le\eta\delta/(96\sqrt{1-\delta})
	\le\eta\delta/(8\sqrt6\sqrt{1-\delta})$, which gives \eqref{eq:gamma-theta-separated-condition}.
	By \eqref{PL_exp_e} and \eqref{PL_exp_c},
	$\alpha\beta L\le\theta\gamma\Delta_\lambda/128\le\theta\gamma\underline\Delta_\lambda/128
	\le\eta\delta/(96\cdot128\sqrt{1-\delta})\le\eta\delta/(4\sqrt6\sqrt{1-\delta})$,
	which is \eqref{eq:alpha-separated-condition}. 
	
	\noindent	\emph{Proof of \eqref{PL_cond_1}.}
	Using $\Delta_\theta\ge\theta\gamma\Delta_\lambda/4$, $\mu\le2L$, and $\nu\ge\tfrac12$, we have
	\begin{align}
		\frac{16\beta^2L^2B_2}{N}\overset{\eqref{ABC_constants_again}}{=}\frac{160\alpha^4\beta^3L^4}{\nu\theta\gamma^2\Delta_\lambda^2}
		\le\frac{320\alpha^4\beta^3L^4}{\theta\gamma^2\Delta_\lambda^2}
		\overset{\eqref{PL_exp_f}}{\le}\frac{\mu\theta\gamma\Delta_\lambda}{1024L}\le\frac{\Delta_\theta}{128}\le\frac{\Delta_\theta}{8}
	\end{align}
and
	\begin{align} \label{jslks}
		\frac{4\beta^2L^2B_2}{N}\le\frac{\Delta_\theta}{512}.
	\end{align}
	
	\noindent	\emph{Proof of $C_1\le\mu_\eta/8$.}
	We bound the three terms in
	$C_1=\frac{2(1-\delta)}{\eta\delta}
	(24\beta^2\gamma^2\underline\Delta_\lambda^2+24\alpha^2\beta^2L^2+48\beta\gamma\underline\Delta_\lambda)$.
	By \eqref{PL_exp_b}, $\beta\gamma\underline\Delta_\lambda\le(\eta\delta)^2/(9216(1-\delta))$
	and $\beta\gamma\underline\Delta_\lambda\le1$, so
	$(\beta\gamma\underline\Delta_\lambda)^2\le\beta\gamma\underline\Delta_\lambda
	\le(\eta\delta)^2/(9216(1-\delta))$. By \eqref{PL_exp_c} and \eqref{PL_exp_e},
	$\alpha\beta L\le\eta\delta/(96\sqrt{1-\delta})$, hence
	$(\alpha\beta L)^2\le(\eta\delta)^2/(9216(1-\delta))$. Therefore
	\begin{align}
		\frac{96(1-\delta)\beta\gamma\underline\Delta_\lambda}{\eta\delta}&\le\frac{\eta\delta}{96},\quad
		\frac{48(1-\delta)\beta^2\gamma^2\underline\Delta_\lambda^2}{\eta\delta}\le\frac{\eta\delta}{192},\quad
		\frac{48(1-\delta)\alpha^2\beta^2L^2}{\eta\delta}\le\frac{\eta\delta}{192},
	\end{align}
	so $C_1\le\tfrac{\eta\delta}{96}+\tfrac{\eta\delta}{192}+\tfrac{\eta\delta}{192}
	=\tfrac{\eta\delta}{48}\le\tfrac{\eta\delta}{32}=\tfrac{\mu_\eta}{8}$.
	
	\noindent\emph{Proof of $B_1C_1\le\mu_\eta\Delta_\theta/64$.}
	Since \eqref{PL_exp_b} gives $\beta\gamma\underline\Delta_\lambda\le1$,
	$24\beta^2\gamma^2\underline\Delta_\lambda^2\le24\beta\gamma\underline\Delta_\lambda$, so
	$C_1\le\frac{2(1-\delta)}{\eta\delta}(72\beta\gamma\underline\Delta_\lambda+24\alpha^2\beta^2L^2)$. Hence, using
	$B_1=\tfrac{52\beta\underline\Delta_\lambda}{\nu\theta\Delta_\lambda}+\tfrac{96\alpha^2\beta^2L^2}{\nu\theta\gamma\Delta_\lambda}$, we have $B_1C_1\le T_1+T_2+T_3$, where
	\begin{align*}
		T_1\define\frac{7488(1-\delta)\beta^2\gamma\underline\Delta_\lambda^2}{\nu\theta\eta\delta\Delta_\lambda},
		\quad
		T_2\define\frac{16320(1-\delta)\alpha^2\beta^3L^2\underline\Delta_\lambda}{\nu\theta\eta\delta\Delta_\lambda},
		\quad
		T_3\define\frac{4608(1-\delta)\alpha^4\beta^4L^4}{\nu\theta\eta\delta\gamma\Delta_\lambda}.
	\end{align*}
	For $T_1$, using $\nu\ge1/2$ and \eqref{PL_exp_d},
	\begin{align*}
		T_1\le\frac{14976(1-\delta)\beta^2\gamma\underline\Delta_\lambda^2}{\theta\eta\delta\Delta_\lambda}
		\le\frac{14976\theta\gamma\eta\delta\Delta_\lambda}{8192^2}
		\overset{\eqref{Delta_lower}}{\le}\frac{59904\eta\delta\Delta_\theta}{8192^2}
		\le\frac{\eta\delta\Delta_\theta}{1024}=\frac{\mu_\eta\Delta_\theta}{256}.
	\end{align*}
	For $T_2$, using $\nu\ge1/2$, \eqref{PL_exp_h}, $\mu\le2L$, $1-\delta\le1$, and
	$\theta\gamma\Delta_\lambda\le4\Delta_\theta$,
	\begin{align*}
		T_2\le\frac{32640(1-\delta)\alpha^2\beta^3L^2\underline\Delta_\lambda}{\theta\eta\delta\Delta_\lambda}
		\le\frac{32640(1-\delta)\mu\theta\gamma\eta\delta\Delta_\lambda}{2^{28}L}
		\le\frac{261120\,\eta\delta\Delta_\theta}{2^{28}}
		\le\frac{\eta\delta\Delta_\theta}{1024}=\frac{\mu_\eta\Delta_\theta}{256}.
	\end{align*}
	For $T_3$ we
	factor $\alpha^4\beta^4=(\alpha^4\beta^3)\beta$, apply \eqref{PL_exp_f}, and use $\nu\ge1/2$,
	\begin{align*}
		T_3
		\le\frac{9216(1-\delta)\alpha^4\beta^4L^4}{\theta\eta\delta\gamma\Delta_\lambda}
		\overset{\eqref{PL_exp_f}}{\le}
		\frac{9216(1-\delta)\mu\theta\gamma^2\Delta_\lambda^2\beta}{327680\,\eta\delta L}
		=\frac{9216(1-\delta)\mu\,(\theta\gamma\Delta_\lambda)(\beta\gamma\Delta_\lambda)}{327680\,\eta\delta L}.
	\end{align*}
	By \eqref{PL_exp_b}, $\beta\gamma\Delta_\lambda\le\beta\gamma\underline\Delta_\lambda
	\le(\eta\delta)^2/(9216(1-\delta))$, and $\theta\gamma\Delta_\lambda\le4\Delta_\theta$, hence
	\begin{align*}
		T_3\le\frac{9216(1-\delta)\mu}{327680\,\eta\delta L}\cdot4\Delta_\theta\cdot
		\frac{(\eta\delta)^2}{9216(1-\delta)}
		=\frac{4\mu\eta\delta\Delta_\theta}{327680\,L}
		\overset{\mu\le2L}{\le}\frac{\eta\delta\Delta_\theta}{40960}
		\le\frac{\mu_\eta\Delta_\theta}{256}.
	\end{align*}
	Combining the three bounds,
	$B_1C_1\le T_1+T_2+T_3\le3\cdot\tfrac{\mu_\eta\Delta_\theta}{256}\le\tfrac{\mu_\eta\Delta_\theta}{64}$,
	which proves \eqref{PL_cond_2}.
	
	\noindent	\emph{Proof of \eqref{PL_cond_3}.}
	Since $A_2=A_1/4$ and \eqref{jslks} gives $4\beta^2L^2B_2/N\le\Delta_\theta/512$,
	\begin{align} \label{ggghhh}
		A_2+r\frac{4\beta^2L^2B_2}{N}
		\le\frac{A_1}{4}+\frac{4A_1}{\Delta_\theta}\cdot\frac{\Delta_\theta}{512}
		=\frac{33A_1}{128}\le\frac{A_1}{3}.
	\end{align}
	By the definitions in \eqref{BC_tilde_s_r},
	$sC_1=\frac{4C_1}{\mu_\eta}\bigl(A_2+r\frac{4\beta^2L^2B_2}{N}\bigr)+\frac{4rB_1C_1}{\mu_\eta}$.
	Using $C_1\le\mu_\eta/8$ and \eqref{ggghhh}, the first part is at most $\tfrac{A_1}{6}$;
	using $r=\tfrac{4A_1}{\Delta_\theta}$ and $B_1C_1\le\tfrac{\mu_\eta\Delta_\theta}{64}$, the
	second is $\tfrac{16A_1B_1C_1}{\mu_\eta\Delta_\theta}\le\tfrac{A_1}{4}$. Hence
	$sC_1\le\tfrac{A_1}{6}+\tfrac{A_1}{4}=\tfrac{5A_1}{12}<\tfrac{A_1}{2}$.

	\noindent	\emph{Proof of \eqref{PL_cond_4}.}
	By \eqref{PL_exp_f} and $\theta\gamma\Delta_\lambda\le4\Delta_\theta$, $\nu\ge\tfrac12$,
	\begin{align}
		\frac{r\widetilde B_2}{\alpha\beta}=\frac{16LB_2A_1}{\alpha\beta\Delta_\theta}
		\overset{\eqref{ABC_constants_again}}{=}\frac{640\alpha^4\beta^3L^5}{\nu\Delta_\theta\theta\gamma^2\Delta_\lambda^2}
		\le\frac{1280\alpha^4\beta^3L^5}{\Delta_\theta\theta\gamma^2\Delta_\lambda^2}
		\le\frac{\mu}{64},
	\end{align}
	so $r\widetilde B_2\le\tfrac{\mu\alpha\beta}{64}$. 	
	For $s\widetilde C_2$ with $\widetilde C_2=\tfrac{24\alpha^2NL(1-\delta)}{\eta\delta}$,
	\begin{align}
		s\widetilde C_2=\frac{4\widetilde C_2}{\mu_\eta}\bigl(A_2+r\tfrac{4\beta^2L^2B_2}{N}\bigr)
		+\frac{4rB_1\widetilde C_2}{\mu_\eta}.
	\end{align}
	Using \eqref{ggghhh}, $\mu_\eta=\eta\delta/4$, and \eqref{PL_exp_g}, the first part is
	\begin{align}
		\frac{4\widetilde C_2}{\mu_\eta}\bigl(A_2+r\frac{4\beta^2L^2B_2}{N}\bigr)\le \frac{4 A_1\widetilde C_2}{3\mu_{\eta}}\le\frac{512\alpha^3\beta^3L^3(1-\delta)}{(\eta\delta)^2}\le\frac{\mu\alpha\beta}{16}.
	\end{align}
	For the second part, with $B_1=\tfrac{52\beta\underline\Delta_\lambda}{\nu\theta\Delta_\lambda}+\tfrac{96\alpha^2\beta^2L^2}{\nu\theta\gamma\Delta_\lambda}$,
	\begin{align}
		\frac{4rB_1\widetilde C_2}{\mu_\eta}
		&=\frac{319488\alpha^3\beta^4L^3(1-\delta)\underline\Delta_\lambda}
		{\nu(\eta\delta)^2\theta\Delta_\lambda\Delta_\theta}
		+\frac{589824\alpha^5\beta^5L^5(1-\delta)}
		{\nu(\eta\delta)^2\theta\gamma\Delta_\lambda\Delta_\theta}.
		\label{eq:rB1Ctilde_expanded_min}
	\end{align}
	For the first term, using $\nu\ge1/2$, $\alpha^3\beta^4=(\alpha^2\beta^3)(\alpha\beta)$,
	\eqref{PL_exp_h}, $1-\delta\le1$, and $\theta\gamma\Delta_\lambda\le4\Delta_\theta$,
	\begin{align}
		\frac{319488\alpha^3\beta^4L^3(1-\delta)\underline\Delta_\lambda}
		{\nu(\eta\delta)^2\theta\Delta_\lambda\Delta_\theta}
		\le\frac{638976(1-\delta)\mu\alpha\beta\theta\gamma\Delta_\lambda}{2^{28}\Delta_\theta}
		\le\frac{2555904\mu\alpha\beta}{2^{28}}\le\frac{\mu\alpha\beta}{100}.
		\label{eq:first_min}
	\end{align}
	For the second term: factor
	$\alpha^5\beta^5=(\alpha^4\beta^3)(\alpha\beta)\beta$, apply \eqref{PL_exp_f}, then
	$\theta\gamma\Delta_\lambda\le4\Delta_\theta$ and \eqref{PL_exp_b}
	$\beta\gamma\Delta_\lambda\le(\eta\delta)^2/(9216(1-\delta))$:
	\begin{align}
		\frac{589824\alpha^5\beta^5L^5(1-\delta)}{\nu(\eta\delta)^2\theta\gamma\Delta_\lambda\Delta_\theta}
		&\le\frac{1179648(1-\delta)\mu\,(\theta\gamma\Delta_\lambda)(\beta\gamma\Delta_\lambda)\,(\alpha\beta)}
		{327680(\eta\delta)^2\Delta_\theta}
		\nonumber\\
		&\le\frac{1179648(1-\delta)\mu\,(4\Delta_\theta)\,(\alpha\beta)}
		{327680(\eta\delta)^2\Delta_\theta}\cdot\frac{(\eta\delta)^2}{9216(1-\delta)}
		=\frac{\mu\alpha\beta}{640}.
		\label{eq:second_min}
	\end{align}
	Plugging \eqref{eq:first_min} and \eqref{eq:second_min} into
	\eqref{eq:rB1Ctilde_expanded_min},
	\begin{align}
		\frac{4rB_1\widetilde C_2}{\mu_\eta}
		\le\frac{\mu\alpha\beta}{100}+\frac{\mu\alpha\beta}{640}
		\le\frac{\mu\alpha\beta}{32}.
	\end{align}
	Therefore $s\widetilde C_2\le\tfrac{\mu\alpha\beta}{16}+\tfrac{\mu\alpha\beta}{32}\le\tfrac{\mu\alpha\beta}{8}$, and
	\begin{align}
		r\widetilde B_2+s\widetilde C_2
		\le\frac{\mu\alpha\beta}{64}+\frac{\mu\alpha\beta}{8}
		=\frac{9\mu\alpha\beta}{64}<\frac{\mu\alpha\beta}{2}.
	\end{align}
	This proves \eqref{PL_cond_4} and completes the proof.
\end{proof}

\subsection{Proof of Theorem \ref{thm:pl_rate}}
The P\L~inequality \eqref{PL_ineq_assump} gives $
	\cG_k\ge2\mu\Ex\tilde f(\bar\phi_e^k)$. Using this in \eqref{AA2} gives
\begin{equation}\label{F_rec_PL}
	\Ex\tilde f(\bar\phi_e^{k+1})
	\;\le\;
	(1-\mu\alpha\beta)\,\Ex\tilde f(\bar\phi_e^k)
	+A_1\cD_k
	+A_2\cE_k
	+A_3,
\end{equation}
where we dropped the nonpositive term $-\tfrac{\alpha\beta}{4}\bar\cG_k$. 
Substituting \eqref{pl_smooth_upper2} into \eqref{BB2} and using \eqref{PL_cond_1},
\begin{equation}\label{D_rec_PL}
	\cD_{k+1}
	\;\le\;
	\widetilde B_2\,\Ex\tilde f(\bar\phi_e^k)
	+\Bigl(1-\tfrac{\Delta_\theta}{2}\Bigr)\cD_k
	+\widetilde B_1\,\cE_k
	+B_3 .
\end{equation}
Now, using \eqref{pl_smooth_upper} in the compression-error recursion \eqref{CC2}, we get
\begin{equation}\label{E_rec_PL}
	\cE_{k+1}
	\;\le\;
	\widetilde C_2\,\Ex\tilde f(\bar\phi_e^k)
	+C_1\cD_k
	+(1-\mu_\eta)\cE_k
	+C_3 .
\end{equation}
 Define
\begin{equation}\label{PL_Lyapunov_def}
	\cH_k
	\;\define\;
	\Ex\tilde f(\bar\phi_e^k)+r\,\cD_k+s\,\cE_k,
	\quad
	r=\frac{4A_1}{\Delta_\theta},
	\quad
	s=\frac{4(A_2+r\widetilde B_1)}{\mu_\eta}.
\end{equation}
Taking \eqref{F_rec_PL}$+r\cdot$\eqref{D_rec_PL}$+s\cdot$\eqref{E_rec_PL}, we get
\begin{align} \label{H_coeff}
\cH_{k+1} &\leq	
	[(1-\mu\alpha\beta)+r\widetilde B_2+s\widetilde C_2]\Ex\tilde f(\bar\phi_e^k)
	\nonumber\\
	&\quad +
	[A_1+r\Bigl(1-\tfrac{\Delta_\theta}{2}\Bigr)+sC_1]\cD_k
	+
	[A_2+r\widetilde B_1+s(1-\mu_\eta)]\cE_k \nonumber \\
	&\quad +A_3+rB_3+sC_3.
\end{align}
We bound each in turn.
\begin{itemize}
	\item \emph{$\cE_k$-coefficient.} By the choice of $s$ in \eqref{PL_Lyapunov_def} ($A_2+r\widetilde B_1=\tfrac{s\mu_\eta}{4}$), 
	\begin{equation}\label{coeffE_bound}
		A_2+r\widetilde B_1+s(1-\mu_\eta)
		=
		s(1-\tfrac{3\mu_\eta}{4})
		\le
		s(1-\tfrac{\mu_\eta}{2}).
	\end{equation}
	
	\item \emph{$\cD_k$-coefficient.} By \eqref{PL_cond_3} ($sC_1\le A_1/2$), and choice of $r$ in \eqref{PL_Lyapunov_def}  ($r\Delta_\theta=4A_1$),
	\begin{equation}\label{coeffD_bound}
		A_1+r(1-\tfrac{\Delta_\theta}{2})+sC_1
		\;\le\;
		A_1+r-\frac{r\Delta_\theta}{2}+\frac{A_1}{2}
		=
		r-\frac{A_1}{2}
		=
		r(1-\tfrac{\Delta_\theta}{8}).
	\end{equation}
	
	\item \emph{$\Ex\tilde f(\bar\phi_e^k)$-coefficient.} By \eqref{PL_cond_4}
	($r\widetilde B_2+s\widetilde C_2\le\mu\alpha\beta/2$),
	\begin{equation}\label{coeffF_bound}
		(1-\mu\alpha\beta)+r\widetilde B_2+s\widetilde C_2
		\;\le\;
		1-\frac{\mu\alpha\beta}{2}.
	\end{equation}
\end{itemize}
\noindent Collecting \eqref{coeffE_bound}--\eqref{coeffF_bound} and setting
\begin{equation}\label{rho_def}
	\rho
	\;\define\;
	\min\Bigl\{
	\tfrac{\mu\alpha\beta}{2},\;
	\tfrac{\Delta_\theta}{8},\;
	\tfrac{\mu_\eta}{2}
	\Bigr\},
\end{equation}
each coefficient is bounded by $(1-\rho)$ times the corresponding weight in $\cH_k$. Hence, from \eqref{H_coeff}, we get
\begin{equation}\label{PL_Lyapunov_rec}
	\cH_{k+1}
	\;\le\;
	(1-\rho)\,\cH_k
	+A_3+rB_3+sC_3 .
\end{equation}
Iterating \eqref{PL_Lyapunov_rec} from $0$ to $K-1$,
\begin{equation}\label{PL_Lyapunov_iterate}
	\cH_K
	\;\le\;
	(1-\rho)^K\cH_0
	+\frac{A_3+rB_3+sC_3}{\rho}.
\end{equation}
The explicit stepsize conditions ensure that the first term in
\eqref{rho_def} is the active contraction factor. Indeed,
\eqref{PL_exp_e} gives
\begin{align}
	\frac{\mu\alpha\beta}{2}
	\le
	\frac{\eta\delta}{8}
	=
	\frac{\mu_\eta}{2}.
\end{align}
Moreover, since \(\mu\le L\), the other bound in \eqref{PL_exp_e} gives
\begin{align}
	\frac{\mu\alpha\beta}{2}
\le
\frac{\theta\gamma\Delta_\lambda}{256}.
\end{align}
Using
\(\Delta_\theta\ge\theta\gamma\Delta_\lambda/4\) from
\eqref{Delta_lower}, we further obtain
\begin{align}
	\frac{\mu\alpha\beta}{2}
	\le
	\frac{\Delta_\theta}{64}
	\le
	\frac{\Delta_\theta}{8}.
\end{align}
Therefore, $
\rho=\frac{\mu\alpha\beta}{2}$.
Using
\(\mathbb E\tilde f(\bar\phi_e^K)\le\mathcal H_K\) and
\(1-\rho\le e^{-\rho}\), \eqref{PL_Lyapunov_iterate} yields
\begin{equation}\label{PL_function_gap_tau}
	\Ex\tilde f(\bar\phi_e^K)
	\le
	\exp\Bigl(-\tfrac{\mu\alpha\beta K}{2}\Bigr)\cH_0
	+\frac{2(A_3+rB_3+sC_3)}{\mu\alpha\beta}.
\end{equation}

\noindent\textbf{Noise floor.}
Using the definitions of $A_3,B_3,C_3$ in \eqref{ABC_constants_again}, the
definitions of $r,s,\widetilde B_1,\widetilde B_2,\widetilde C_2$ in
\eqref{BC_tilde_s_r}, and the bounds established in
Proposition~\ref{prop:pl_stepsize}, we first estimate
\begin{align}\label{PL_noise_floor_split}
	\frac{A_3+rB_3+sC_3}{\alpha\beta}
	&=
	\frac{A_3}{\alpha\beta}
	+
	\frac{rB_3}{\alpha\beta}
	+
	\frac{sC_3}{\alpha\beta}.
\end{align}
From \eqref{ABC_constants_again},
\begin{align}\label{A3_over_tau}
	\frac{A_3}{\alpha\beta}
	&=
	\frac{\alpha\beta L\sigma^2}{2N}.
\end{align}
Using the definitions of $r$ in \eqref{BC_tilde_s_r} and $B_3$ in
\eqref{ABC_constants_again}, together with
$\Delta_\theta\ge \theta\gamma\Delta_\lambda/4$ from \eqref{Delta_lower}, we get
\begin{align}\label{rB3_over_tau}
	\frac{rB_3}{\alpha\beta}
	&=
	\cO\left(
	\frac{\alpha^2\beta^2L^2\sigma^2}
	{\nu\theta\gamma\Delta_\lambda}
	\right)
	+
	\cO\left(
	\frac{\alpha^4\beta^4L^4\sigma^2}
	{N\nu\theta^2\beta\gamma^3\Delta_\lambda^3}
	\right).
\end{align}
Next, we bound $sC_3/(\alpha\beta)$. Using the definitions of $s$ and $r$ in
\eqref{BC_tilde_s_r}, we have
\begin{align}
	sC_3
	&=
	\frac{4C_3}{\mu_\eta}
	\left(
	A_2+r\frac{4\beta^2L^2B_2}{N}
	\right)
	+
	\frac{4rB_1C_3}{\mu_\eta}
	\nonumber\\
	&\overset{\eqref{ggghhh}}{\le}
	\frac{4A_1C_3}{3\mu_\eta}
	+
	\frac{16A_1B_1C_3}{\mu_\eta\Delta_\theta},
	\label{sC3_split}
\end{align}
where we used $r=4A_1/\Delta_\theta$.
We first bound the first term in \eqref{sC3_split} as follows
\begin{align}
	\frac{4A_1C_3}{3\mu_\eta\alpha\beta}
	&=
	\frac{4}{3\mu_\eta\alpha\beta}
	\cdot
	\frac{4\alpha\beta^3L^2}{N}
	\cdot
	\frac{6(1-\delta)}{\eta\delta}\alpha^2N\sigma^2
	\nonumber\\
	&=
	\frac{128(1-\delta)\alpha^2\beta^2L^2\sigma^2}
	{(\eta\delta)^2}
	\nonumber\\
	&=
	\cO\left(
	\alpha^2\beta^2L^2\sigma^2
	\frac{1-\delta}{\eta^2\delta^2}
	\right).
	\label{sC3_first_part}
\end{align}
We now bound the second term in \eqref{sC3_split} as follows
\begin{align}
	\frac{16A_1B_1C_3}{\mu_\eta\Delta_\theta\alpha\beta}
	&=
	\frac{16}{\mu_\eta\Delta_\theta\alpha\beta}
	\cdot
	\frac{4\alpha\beta^3L^2}{N}
	\cdot
	B_1
	\cdot
	\frac{6(1-\delta)}{\eta\delta}\alpha^2N\sigma^2
	\nonumber\\
	&=
	\frac{1536(1-\delta)\alpha^2\beta^2L^2\sigma^2}
	{(\eta\delta)^2\Delta_\theta}
	B_1 .
	\label{sC3_second_pre}
\end{align}
Substituting the definition $
B_1
=
\frac{52\beta\underline\Delta_\lambda}
{\nu\theta\Delta_\lambda}
+
\frac{96\alpha^2\beta^2L^2}
{\nu\theta\gamma\Delta_\lambda}$,
we get
\begin{align}
	\frac{16A_1B_1C_3}{\mu_\eta\Delta_\theta\alpha\beta}
	&=
	\cO\left(
	\frac{
		\alpha^2\beta^3L^2\sigma^2(1-\delta)\underline\Delta_\lambda
	}{
		\nu\eta^2\delta^2\theta\Delta_\lambda\Delta_\theta
	}
	\right)
	\nonumber\\
	&\quad+
	\cO\left(
	\frac{
		\alpha^4\beta^4L^4\sigma^2(1-\delta)
	}{
		\nu\eta^2\delta^2\theta\gamma\Delta_\lambda\Delta_\theta
	}
	\right).
	\label{sC3_second_expanded}
\end{align}
Using $\Delta_\theta\gtrsim \theta\gamma\Delta_\lambda$, the first term becomes
\begin{align}
	\frac{
		\alpha^2\beta^3L^2\sigma^2(1-\delta)\underline\Delta_\lambda
	}{
		\nu\eta^2\delta^2\theta\Delta_\lambda\Delta_\theta
	}
	&=
	\cO\left(
	\frac{
		\alpha^2\beta^3L^2\sigma^2(1-\delta)\underline\Delta_\lambda
	}{
		\nu\eta^2\delta^2\theta^2\gamma\Delta_\lambda^2
	}
	\right)
	\nonumber\\
	&=
	\cO\left(
	\frac{\alpha^2\beta^2L^2\sigma^2}
	{\nu\beta\gamma\underline\Delta_\lambda}
	\cdot
	\frac{(1-\delta)\beta^2\underline\Delta_\lambda^2}
	{\eta^2\delta^2\theta^2\Delta_\lambda^2}
	\right).
\end{align}
By the admissible scaling $\beta
\lesssim
\frac{\theta\eta\delta\Delta_\lambda}
{\underline\Delta_\lambda\sqrt{1-\delta}}$, we have
\begin{align}
	\frac{
		\alpha^2\beta^3L^2\sigma^2(1-\delta)\underline\Delta_\lambda
	}{
		\nu\eta^2\delta^2\theta\Delta_\lambda\Delta_\theta
	}
	&=
	\cO\left(
	\frac{\alpha^2\beta^2L^2\sigma^2}
	{\nu\beta\gamma\underline\Delta_\lambda}
	\right).
	\label{sC3_B1_first}
\end{align}
For the second term in \eqref{sC3_second_expanded}, again using
$\Delta_\theta\gtrsim\theta\gamma\Delta_\lambda$, we have
\begin{align}
	\frac{
		\alpha^4\beta^4L^4\sigma^2(1-\delta)
	}{
		\nu\eta^2\delta^2\theta\gamma\Delta_\lambda\Delta_\theta
	}
	&=
	\cO\left(
	\frac{
		\alpha^4\beta^4L^4\sigma^2(1-\delta)
	}{
		\nu\eta^2\delta^2\theta^2\gamma^2\Delta_\lambda^2
	}
	\right)
	\nonumber\\
	&=
	\cO\left(
	\frac{\alpha^2\beta^2L^2\sigma^2}
	{\nu\beta\gamma\underline\Delta_\lambda}
	\cdot
	\frac{(1-\delta)\alpha^2\beta^3L^2\underline\Delta_\lambda}
	{\eta^2\delta^2\theta^2\gamma\Delta_\lambda^2}
	\right).
\end{align}
Using the admissible condition $\alpha^2\beta^3
\lesssim
\frac{\mu\theta^2\gamma(\eta\delta)^2\Delta_\lambda^2}
{L^3\underline\Delta_\lambda}$,
and $\mu\le 2L$, the second factor is bounded by an absolute constant. Thus
\begin{align}
	\frac{
		\alpha^4\beta^4L^4\sigma^2(1-\delta)
	}{
		\nu\eta^2\delta^2\theta\gamma\Delta_\lambda\Delta_\theta
	}
	&=
	\cO\left(
	\frac{\alpha^2\beta^2L^2\sigma^2}
	{\nu\beta\gamma\underline\Delta_\lambda}
	\right).
	\label{sC3_B1_second}
\end{align}
Combining \eqref{sC3_first_part}, \eqref{sC3_B1_first}, and
\eqref{sC3_B1_second}, we obtain the direct bound
\begin{align}\label{sC3_over_tau}
	\frac{sC_3}{\alpha\beta}
	&=
	\cO\left(
	\alpha^2\beta^2L^2\sigma^2
	\left[
	\frac{1-\delta}{\eta^2\delta^2}
	+
	\frac{1}{\nu\beta\gamma\underline\Delta_\lambda}
	\right]
	\right).
\end{align}
Combining \eqref{PL_noise_floor_split}, \eqref{A3_over_tau},
\eqref{rB3_over_tau}, and \eqref{sC3_over_tau} yields
\begin{align}\label{PL_noise_floor_order}
	\frac{A_3+rB_3+sC_3}{\alpha\beta}
	&=
	\cO\left(
	\frac{\alpha\beta L\sigma^2}{N}
	\right)
	\nonumber\\
	&\quad
	+
	\cO\left(
	\alpha^2\beta^2L^2\sigma^2
	\left[
	\frac{1}{\nu\theta\gamma\Delta_\lambda}
	+
	\frac{1-\delta}{\eta^2\delta^2}
	+
	\frac{1}{\nu\beta\gamma\underline\Delta_\lambda}
	\right]
	\right)
	\nonumber\\
	&\quad
	+
	\cO\left(
	\frac{\alpha^4\beta^4L^4\sigma^2}
	{N\nu\theta^2\beta\gamma^3\Delta_\lambda^3}
	\right).
\end{align}
It follows from \eqref{PL_function_gap_tau} that
\begin{equation}\label{PL_function_gap_tau22}
\begin{aligned}
		\Ex\tilde f(\bar\phi_e^K)
	&\le
	\exp\Bigl(-\tfrac{\mu\alpha\beta K}{2}\Bigr)\cH_0 
	+\cO\left(
	\frac{\alpha\beta L\sigma^2}{\mu N}
	\right)
	\\
	&\quad
	+
	\cO\left(
	\frac{\alpha^2\beta^2L^2\sigma^2}{\mu}
	\left[
	\frac{1}{\nu\theta\gamma\Delta_\lambda}
	+
	\frac{1-\delta}{\eta^2\delta^2}
	+
	\frac{1}{\nu\beta\gamma\underline\Delta_\lambda}
	\right]
	\right)
	\\
	&\quad
	+
	\cO\left(
	\frac{\alpha^4\beta^4L^4\sigma^2}
	{\mu N\nu\theta^2\beta\gamma^3\Delta_\lambda^3}
	\right).
\end{aligned}
\end{equation}
For the initialization term, using \eqref{d0}, we have
\begin{align}\label{d0_pl_recall}
	\cH_0
	&=
	\tilde f(\bar\phi_e^0)+r\cD_0
	\le
	\tilde f(\bar\phi_e^0)
	+
	rv_r^2\alpha^2
	\|\Q^{\top}\grad\f(\x^{0})\|^2,
\end{align}
where $
	\|\V_r^{-1}\|^2
	\le
	\frac{1}{2\beta\gamma(1-\lambda)\nu}
	=
	\frac{1}{2\beta\gamma\Delta_\lambda\nu}
	\define v_r^2$.
Using the definition of $r=\frac{4A_1}{\Delta_\theta}$ from \eqref{BC_tilde_s_r}, the definition of $A_1$ in
\eqref{ABC_constants_again}, and $\tau=\alpha\beta$, we have
\begin{align}\label{rD0_exact}
	r\cD_0
	&\le
	\frac{4A_1}{\Delta_\theta}
	\alpha^2v_r^2
	\|\Q^{\top}\grad\f(\x^{0})\|^2
	\nonumber\\
	&=
	\frac{16\alpha\beta^3L^2}{N\Delta_\theta}
	\alpha^2v_r^2
	\|\Q^{\top}\grad\f(\x^{0})\|^2
	\nonumber\\
	&=
	\frac{16L^2v_r^2}{\Delta_\theta}
	\tau^3
	\varsigma_0^2
	\nonumber\\
	&\leq 
	\frac{8L^2}{\nu\beta \gamma \Delta_{\lambda}\Delta_\theta}
	\tau^3
	\varsigma_0^2		\nonumber\\
	&\leq 
	\frac{32L^2}{\nu \theta \beta \gamma^2 \Delta_{\lambda}^2}
	\tau^3
	\varsigma_0^2,
\end{align}
where $
\varsigma_0^2
\define
\frac1N
\|
\Q^\top\nabla\mathbf f(\mathbf x^0)
\|^2
=
\frac1N
\left\|
\nabla\mathbf f(\mathbf x^0)
-
\mathbf 1\otimes\nabla f(x^0)
\right\|^2
=
\frac1N
\sum_{i=1}^N
\left\|
\nabla f_i(x^0)-\nabla f(x^0)
\right\|^2$.
Using \eqref{PL_explicit}, we let $\bar\tau$ denote the largest admissible  value of $\tau=\alpha \beta$:
\begin{align} \label{tau_bar_pl}
	\tau
	\ \le\ \min\Bigg\{
	\frac{\theta\gamma\Delta_\lambda}{128L},
	\ \frac{\eta\delta}{4\mu},
	\ \Big(\frac{\mu\theta^2\gamma^3\Delta_\lambda^3\,\beta}{327680\,L^5}\Big)^{1/4},
	\ \Big(\frac{\mu(\eta\delta)^2}{8192\,L^3(1-\delta)}\Big)^{1/2},
	\ \Big(\frac{\mu\theta^2\gamma(\eta\delta)^2\Delta_\lambda^2}{2^{28}L^3\underline\Delta_\lambda\,\beta}\Big)^{1/2}
	\Bigg\} \define \bar{\tau}.
\end{align}
Since
$\tau\le\bar\tau$, \eqref{rD0_exact} implies
\begin{align}\label{rD0_quadratic_envelope}
	r\cD_0
	&\le
	\widetilde a_2\tau^2,\quad
	\widetilde a_2
	\define
	\frac{32\bar{\tau}L^2\varsigma_0^2}{\nu  \theta \beta \gamma^2 \Delta_{\lambda}^2}
	.
\end{align}
Therefore, \eqref{PL_function_gap_tau22} becomes
\begin{equation}\label{PL_function_gap_tau2}
	\begin{aligned}
		\Ex\tilde f(\bar\phi_e^K)
		&\le
		\exp\Bigl(-\tfrac{\mu\alpha\beta K}{2}\Bigr)\left(
		\tilde f(\bar\phi_e^0)+\widetilde a_2\tau^2
		\right)
		+\cO\left(
		\frac{\alpha\beta L\sigma^2}{\mu N}
		\right)
		\\
		&\quad
		+
		\cO\left(
		\frac{\alpha^2\beta^2L^2\sigma^2}{\mu}
		\left[
		\frac{1}{\nu\theta\gamma\Delta_\lambda}
		+
		\frac{1-\delta}{\eta^2\delta^2}
		+
		\frac{1}{\nu\beta\gamma\underline\Delta_\lambda}
		\right]
		\right)
		\\
		&\quad
		+
		\cO\left(
		\frac{\alpha^4\beta^4L^4\sigma^2}
		{\mu N\nu\theta^2\beta\gamma^3\Delta_\lambda^3}
		\right).
	\end{aligned}
\end{equation}
\subsection{Proof of Corollary \ref{corr:pl_rate}} \label{app_corr_step_rate_pl}
We first find the order of the noise terms in \eqref{PL_function_gap_tau2} by selecting  $\theta,\eta,\nu=\Theta(1)$ and 
\begin{align}
	\beta
	&\asymp
	\frac{\delta\Delta_\lambda}
	{\underline\Delta_\lambda\sqrt{1-\delta}},\quad
	\gamma
	\asymp
	\frac{\beta}{\underline\Delta_\lambda}
	\asymp
	\frac{\delta\Delta_\lambda}
	{\underline\Delta_\lambda^2\sqrt{1-\delta}}.
	\label{PL_scaling}
\end{align}
The constants hidden in $\asymp$ are chosen sufficiently small so that
\eqref{PL_exp_a}--\eqref{PL_exp_d} hold. First,
by \eqref{PL_scaling},
\begin{align}\label{bracket1_term1}
	\frac{1}{\nu\theta\gamma\Delta_\lambda}
	&=
	\cO\left(
	\frac{\underline\Delta_\lambda^2\sqrt{1-\delta}}
	{\delta\Delta_\lambda^2}
	\right)
	\le
	\cO\left(
	\frac{\underline\Delta_\lambda^2}
	{\delta^2\Delta_\lambda^2}
	\right),
\end{align}
where we used $0<\delta\le1$ and $1-\delta\le1$. Also,
\begin{align}\label{bracket1_term2}
	\frac{1-\delta}{\eta^2\delta^2}
	&=
	\cO\left(
	\frac{1}{\delta^2}
	\right)
	\le
	\cO\left(
	\frac{\underline\Delta_\lambda^2}
	{\delta^2\Delta_\lambda^2}
	\right),
\end{align}
because $\Delta_\lambda\le \underline\Delta_\lambda$. Finally, again by
\eqref{PL_scaling},
\begin{align}\label{bracket1_term3}
	\frac{1}{\nu\beta\gamma\underline\Delta_\lambda}
	&=
	\cO\left(
	\frac{(1-\delta)\underline\Delta_\lambda^2}
	{\delta^2\Delta_\lambda^2}
	\right)
	\le
	\cO\left(
	\frac{\underline\Delta_\lambda^2}
	{\delta^2\Delta_\lambda^2}
	\right).
\end{align}
Therefore, \eqref{bracket1_term1}--\eqref{bracket1_term3} imply
\begin{align}\label{bracket1}
	\frac{1}{\nu\theta\gamma\Delta_\lambda}
	+
	\frac{1-\delta}{\eta^2\delta^2}
	+
	\frac{1}{\nu\beta\gamma\underline\Delta_\lambda}
	&=
	\cO\left(
	\frac{\underline\Delta_\lambda^2}
	{\delta^2\Delta_\lambda^2}
	\right).
\end{align}
Similarly, using \eqref{PL_scaling},
\begin{align}\label{bracket2}
	\frac{1}{\nu\theta^2\beta\gamma^3\Delta_\lambda^3}
	&=
	\cO\left(
	\frac{(1-\delta)^2\underline\Delta_\lambda^7}
	{\delta^4\Delta_\lambda^7}
	\right)
	\le
	\cO\left(
	\frac{\underline\Delta_\lambda^7}
	{\delta^4\Delta_\lambda^7}
	\right).
\end{align}
Substituting \eqref{bracket1} and \eqref{bracket2} into
\eqref{PL_function_gap_tau2} gives the
rate as a function of the single stepsize parameter $\tau$:
\begin{align}\label{PL_rate_final_tau_with_initial}
	\Ex\tilde f(\bar\phi_e^K)
	&\le
	\exp\left(
	-\frac{\mu\tau K}{2}
	\right)
	\left(
	\tilde f(\bar\phi_e^0)+\widetilde a_2\tau^2
	\right)
	+
	\cO\left(
	\frac{\tau L\sigma^2}{\mu N}
	\right)
	\nonumber\\
	&\quad
	+
	\cO\left(
	\frac{\tau^2L^2\sigma^2}{\mu}
	\frac{\underline\Delta_\lambda^2}
	{\delta^2\Delta_\lambda^2}
	\right)
	+
	\cO\left(
	\frac{\tau^4L^4\sigma^2}{\mu N}
	\frac{\underline\Delta_\lambda^7}
	{\delta^4\Delta_\lambda^7}
	\right).
\end{align}
Under the heavy-compression regime \(1-\delta=\Theta(1)\), with
\(\theta,\eta,\nu=\Theta(1)\), define the condition number $
\kappa\define\frac{L}{\mu}\ge1$.
Substituting \eqref{PL_scaling} into \eqref{tau_bar_pl} gives
\begin{align}
	\bar\tau
	=
	\Theta\!\left(
	\frac{\delta}{L}
	\min\left\{
	\left(\frac{\Delta_\lambda}{\underline\Delta_\lambda}\right)^2,
	\;
	\kappa^{-1/4}
	\left(\frac{\Delta_\lambda}{\underline\Delta_\lambda}\right)^{7/4},
	\;
	\kappa^{-1/2},
	\;
	\kappa^{-1/2}
	\frac{\Delta_\lambda}{\underline\Delta_\lambda},
	\;
	\kappa
	\right\}
	\right).
	\label{tau_bar_order}
\end{align}
Since $
\frac{\Delta_\lambda}{\underline\Delta_\lambda}\le1$,
$\kappa\ge1$,
the terms of orders \(\kappa\) and \(\kappa^{-1/2}\) cannot be active:
the former is no smaller than
\((\Delta_\lambda/\underline\Delta_\lambda)^2\), whereas the latter
is no smaller than
\(\kappa^{-1/2}\Delta_\lambda/\underline\Delta_\lambda\).
Hence,
\begin{align}
	\bar\tau
	=
	\Theta\!\left(
	\frac{\delta}{L}
	\min\left\{
	\left(\frac{\Delta_\lambda}{\underline\Delta_\lambda}\right)^2,
	\;
	\kappa^{-1/4}
	\left(\frac{\Delta_\lambda}{\underline\Delta_\lambda}\right)^{7/4},
	\;
	\kappa^{-1/2}
	\frac{\Delta_\lambda}{\underline\Delta_\lambda}
	\right\}
	\right).
	\label{tau_bar_graph_limited_eta_const}
\end{align}
Consequently,
\begin{align}
	\frac{1}{\mu\bar\tau}
	=
	\Theta\!\left(
	\max\left\{
	\frac{\kappa\underline\Delta_\lambda^2}
	{\delta\Delta_\lambda^2},
	\;
	\frac{\kappa^{5/4}\underline\Delta_\lambda^{7/4}}
	{\delta\Delta_\lambda^{7/4}},
	\;
	\frac{\kappa^{3/2}\underline\Delta_\lambda}
	{\delta\Delta_\lambda}
	\right\}
	\right).
	\label{PL_inv_tau_bar}
\end{align}
For later use, \eqref{tau_bar_graph_limited_eta_const} also implies
\begin{align}
	\bar\tau
	&=
	\mathcal O\!\left(
	\frac{\delta\Delta_\lambda^2}
	{L\underline\Delta_\lambda^2}
	\right),
	\\
	\widetilde a_2
	&=
	\mathcal O\!\left(
	\frac{
		L\varsigma_0^2\underline\Delta_\lambda^3
	}{
		\delta^2\Delta_\lambda^3
	}
	\right),
\end{align}
where the second bound follows from \eqref{rD0_quadratic_envelope} and
\eqref{PL_scaling}. Therefore,
\begin{align}
	\widetilde a_2\bar\tau^2
	=
	\mathcal O\!\left(
	\frac{
		\varsigma_0^2\Delta_\lambda
	}{
		L\underline\Delta_\lambda
	}
	\right).
	\label{PL_initial_term_bound}
\end{align}
We now apply Lemma~\ref{lem:linear_stepsize_quartic} to
\eqref{PL_rate_final_tau_with_initial}. Comparing
\eqref{PL_rate_final_tau_with_initial} with \eqref{eq:quartic_recursion}, we choose
\begin{align}\label{PL_lemma_coefficients}
	a_0
	&=
	\tilde{f}(\bar\phi_e^0),\quad
	a_1
	=
	\frac{L\sigma^2}{\mu N},\quad
	a_2
	=
	\frac{L^2\sigma^2\underline\Delta_\lambda^2}
	{\mu\delta^2\Delta_\lambda^2},
	\nonumber\\
	a_3
	&=
	0,\quad
	a_4
	=
	\frac{L^4\sigma^2\underline\Delta_\lambda^7}
	{\mu N\delta^4\Delta_\lambda^7},\quad
	\widetilde a_2
	=
	\cO\left(
	\frac{\varsigma_0^2\underline\Delta_\lambda^3}
	{\delta^2\Delta_\lambda^3}
	\right).
\end{align}
The logarithmic stepsize is chosen as
\begin{align}\label{PL_log_tau_choice}
	\tau
	&=
	\min\left\{
	\frac{2}{\mu K}\ln T_K,
	\bar\tau
	\right\},
	\quad
	T_K
	\define
	\max\left\{
	\exp(1),
	\frac{\mu(\tilde{f}(\bar\phi_e^0)+\widetilde a_2\bar\tau^2)K}{a_1}
	\right\}.
\end{align}
Therefore, by Lemma~\ref{lem:linear_stepsize_quartic},
\begin{align}\label{PL_coeff_substitution}
	\frac{a_1}{\mu K}
	&=
	\frac{L\sigma^2}{\mu^2NK},
	\nonumber\\
	\frac{a_2}{\mu^2K^2}
	&=
	\frac{L^2\sigma^2\underline\Delta_\lambda^2}
	{\mu^3\delta^2\Delta_\lambda^2K^2},
	\nonumber\\
	\frac{a_3}{\mu^3K^3}
	&=
	0,
	\nonumber\\
	\frac{a_4}{\mu^4K^4}
	&=
	\frac{L^4\sigma^2\underline\Delta_\lambda^7}
	{\mu^5N\delta^4\Delta_\lambda^7K^4}.
\end{align}
Combining \eqref{PL_coeff_substitution} with \eqref{eq:quartic_bound} gives
\begin{align}\label{PL_final_rate_with_initial_order}
	\Ex\tilde f(\bar\phi_e^K)
	&=
	\widetilde\cO\left(
	\frac{L\sigma^2}{\mu^2NK}
	\right)
	+
	\widetilde\cO\left(
	\frac{L^2\sigma^2\underline\Delta_\lambda^2}
	{\mu^3\delta^2\Delta_\lambda^2K^2}
	\right)
	\nonumber\\
	&\quad
	+
	\widetilde\cO\left(
	\frac{L^4\sigma^2\underline\Delta_\lambda^7}
	{\mu^5N\delta^4\Delta_\lambda^7K^4}
	\right)
	+
	\left(
	\tilde{f}(\bar\phi_e^0)+\widetilde a_2\bar\tau^2
	\right)
	\exp\left(
	-\frac{\mu\bar\tau K}{2}
	\right).
\end{align}
It remains to derive the transient time. We compare each lower-order term in
\eqref{PL_final_rate_with_initial_order} with the leading term
\[
\widetilde\cO\left(
\frac{L\sigma^2}{\mu^2NK}
\right).
\]
First, the $K^{-2}$ term is dominated by the leading $K^{-1}$ term when
\begin{align}\label{PL_transient_K2}
	\frac{L^2\sigma^2\underline\Delta_\lambda^2}
	{\mu^3\delta^2\Delta_\lambda^2K^2}
	&\le
	\frac{L\sigma^2}{\mu^2NK}
	\Longleftrightarrow
	K
	\ge
	\frac{NL\underline\Delta_\lambda^2}
	{\mu\delta^2\Delta_\lambda^2}
	=
	\frac{N\kappa\underline\Delta_\lambda^2}
	{\delta^2\Delta_\lambda^2}.
\end{align}
 Second, the $K^{-4}$ term is dominated by the leading term when
\begin{align}\label{PL_transient_K4}
	\frac{L^4\sigma^2\underline\Delta_\lambda^7}
	{\mu^5N\delta^4\Delta_\lambda^7K^4}
	&\le
	\frac{L\sigma^2}{\mu^2NK}
	\nonumber\\
	&\Longleftrightarrow
	K^3
	\ge
	\frac{L^3\underline\Delta_\lambda^7}
	{\mu^3\delta^4\Delta_\lambda^7}
	\nonumber\\
	&\Longleftrightarrow
	K
	\ge
	\frac{\kappa\underline\Delta_\lambda^{7/3}}
	{\delta^{4/3}\Delta_\lambda^{7/3}} .
\end{align}
Finally, consider the exponentially decaying term in
\eqref{PL_final_rate_with_initial_order}. By the logarithmic stepsize
selection, this term is dominated by the leading stochastic term once
\[
K
\ge
\widetilde{\mathcal O}\!\left(
\frac{1}{\mu\bar\tau}
\right).
\]
Using \eqref{PL_inv_tau_bar}, this is equivalent to
\begin{align}\label{PL_transient_exp}
	K
	\ge
	\widetilde{\mathcal O}\left(
	\max\left\{
	\frac{\kappa\underline\Delta_\lambda^2}
	{\delta\Delta_\lambda^2},
	\;
	\frac{\kappa^{5/4}\underline\Delta_\lambda^{7/4}}
	{\delta\Delta_\lambda^{7/4}},
	\;
	\frac{\kappa^{3/2}\underline\Delta_\lambda}
	{\delta\Delta_\lambda}
	\right\}
	\right).
\end{align}
Here, the logarithmic factors may depend on
\(\tilde f(\bar\phi_e^0)+\widetilde a_2\bar\tau^2\) through \(T_K\);
by \eqref{PL_initial_term_bound}, the initialization contribution satisfies
\[
\widetilde a_2\bar\tau^2
=
\mathcal O\!\left(
\frac{
	\varsigma_0^2\Delta_\lambda
}{
	L\underline\Delta_\lambda
}
\right).
\]
 Taking the maximum of the thresholds in \eqref{PL_transient_K2},
\eqref{PL_transient_K4}, and \eqref{PL_transient_exp}, gives
\begin{align}
	K
	\ge
	\widetilde{\mathcal O}\left(
	\max\left\{
	\frac{N\kappa\underline\Delta_\lambda^2}
	{\delta^2\Delta_\lambda^2},
	\;
	\frac{\kappa\underline\Delta_\lambda^{7/3}}
	{\delta^{4/3}\Delta_\lambda^{7/3}},
	\;
	\frac{\kappa\underline\Delta_\lambda^2}
	{\delta\Delta_\lambda^2},
	\;
	\frac{\kappa^{5/4}\underline\Delta_\lambda^{7/4}}
	{\delta\Delta_\lambda^{7/4}},
	\;
	\frac{\kappa^{3/2}\underline\Delta_\lambda}
	{\delta\Delta_\lambda}
	\right\}
	\right).
\end{align}
Keeping only the compression and graph parameters  proves \eqref{PL_transient}.

\section{Stepsize selection}
\subsection{Nonconvex case}
	\begin{lemma}[\sc \small Sublinear rate stepsize selection]\label{lem:stepsize_selection}
	\rm	Let $K\ge 1$, let $\bar\tau>0$, and let $a_0>0$ and
	$a_1,a_2,\tilde{a}_2,a_3\ge 0$. Suppose that, for every
	$\tau\in(0,\bar\tau]$, the nonnegative sequence
	$\{\cE_k\}_{k=0}^{K-1}$ satisfies
	\begin{align}
		\frac1K\sum_{k=0}^{K-1}\cE_k
		\le
		\frac{a_0}{\tau K}
		+a_1\tau
		+a_2\tau^2
		+\frac{\tilde{a}_2}{K}\tau^2
		+a_3\tau^4.
		\label{eq:general_stepsize_recursion}
	\end{align}
	Define
	\begin{equation}\label{eq:tilde_a2_def}
		\tilde a_{2,K}\define a_2+\frac{\tilde{a}_2}{K}.
	\end{equation}
	Adopting the convention that any candidate with zero denominator is set to
	$+\infty$, choose
	\begin{equation}\label{eq:general_alpha_choice}
		\tau
		=
		\min\left\{
		\left(\frac{a_0}{a_1K}\right)^{1/2},
		\left(\frac{a_0}{\tilde a_{2,K}K}\right)^{1/3},
		\left(\frac{a_0}{a_3K}\right)^{1/5},
		\bar\tau
		\right\}.
	\end{equation}
	Then,
	\begin{align}
		\frac1K\sum_{k=0}^{K-1}\cE_k
		&\le
		2\left(\frac{a_1a_0}{K}\right)^{1/2}
		+
		2a_2^{1/3}
		\left(\frac{a_0}{K}\right)^{2/3}
		+
		2\tilde{a}_2^{1/3}\frac{a_0^{2/3}}{K}
		\nonumber\\
		&\quad
		+
		2a_3^{1/5}
		\left(\frac{a_0}{K}\right)^{4/5}
		+
		\frac{a_0}{\bar\tau K}.
		\label{eq:general_stepsize_bound}
	\end{align}
\end{lemma}
\begin{proof}
	By the definition of $\tilde a_{2,K}$, $a_2\tau^2+\frac{\tilde{a}_2}{K}\tau^2
	=
	\tilde a_{2,K}\tau^2$.
	Hence it is enough to bound
	\[
	\Psi_K(\tau)
	\define
	\frac{a_0}{\tau K}
	+
	a_1\tau
	+
	\tilde a_{2,K}\tau^2
	+
	a_3\tau^4 .
	\]
	The convention on zero denominators simply removes the corresponding
	candidate from the minimum.		We consider the possible active constraints in
	\eqref{eq:general_alpha_choice}, derived using term balancing.
	
	\begin{enumerate}
		\item \emph{Case $\tau=\bar\tau$.}
		Then
		\[
		a_1\bar\tau
		\le
		\left(\frac{a_1a_0}{K}\right)^{1/2},
		\quad
		\tilde a_{2,K}\bar\tau^2
		\le
		\tilde a_{2,K}^{1/3}
		\left(\frac{a_0}{K}\right)^{2/3},\quad
		a_3\bar\tau^4
		\le
		a_3^{1/5}
		\left(\frac{a_0}{K}\right)^{4/5}.
		\]
		Therefore,
		\begin{align}
			\Psi_K(\bar\tau)
			\le
			\frac{a_0}{\bar\tau K}
			+
			\left(\frac{a_1a_0}{K}\right)^{1/2}
			+
			\tilde a_{2,K}^{1/3}
			\left(\frac{a_0}{K}\right)^{2/3}
			+
			a_3^{1/5}
			\left(\frac{a_0}{K}\right)^{4/5}.
			\label{eq:case_alpha_bar}
		\end{align}
		
		\item \emph{Case $\tau=\left(a_0/(a_1K)\right)^{1/2}$.}
		Then
		\[
		\frac{a_0}{\tau K}
		=
		a_1\tau
		=
		\left(\frac{a_1a_0}{K}\right)^{1/2}, \quad
		\tilde a_{2,K}\tau^2
		\le
		\tilde a_{2,K}^{1/3}
		\left(\frac{a_0}{K}\right)^{2/3},
		\quad
		a_3\tau^4
		\le
		a_3^{1/5}
		\left(\frac{a_0}{K}\right)^{4/5}.
		\]
		Thus
		\begin{align}
			\Psi_K(\tau)
			\le
			2\left(\frac{a_1a_0}{K}\right)^{1/2}
			+
			\tilde a_{2,K}^{1/3}
			\left(\frac{a_0}{K}\right)^{2/3}
			+
			a_3^{1/5}
			\left(\frac{a_0}{K}\right)^{4/5}.
			\label{eq:case_alpha_a1}
		\end{align}
		
		\item \emph{Case $\tau=\left(a_0/(\tilde a_{2,K}K)\right)^{1/3}$.}
		Then
		\[
		\frac{a_0}{\tau K}
		=
		\tilde a_{2,K}\tau^2
		=
		\tilde a_{2,K}^{1/3}
		\left(\frac{a_0}{K}\right)^{2/3},\quad
		a_1\tau
		\le
		\left(\frac{a_1a_0}{K}\right)^{1/2},
		\quad
		a_3\tau^4
		\le
		a_3^{1/5}
		\left(\frac{a_0}{K}\right)^{4/5}.
		\]
		Therefore,
		\begin{align}
			\Psi_K(\tau)
			\le
			2\tilde a_{2,K}^{1/3}
			\left(\frac{a_0}{K}\right)^{2/3}
			+
			\left(\frac{a_1a_0}{K}\right)^{1/2}
			+
			a_3^{1/5}
			\left(\frac{a_0}{K}\right)^{4/5}.
			\label{eq:case_alpha_a2}
		\end{align}
		
		\item \emph{Case $\tau=\left(a_0/(a_3K)\right)^{1/5}$.}
		Then
		\[
		\frac{a_0}{\tau K}
		=
		a_3\tau^4
		=
		a_3^{1/5}
		\left(\frac{a_0}{K}\right)^{4/5},\quad
		a_1\tau
		\le
		\left(\frac{a_1a_0}{K}\right)^{1/2},
		\quad
		\tilde a_{2,K}\tau^2
		\le
		\tilde a_{2,K}^{1/3}
		\left(\frac{a_0}{K}\right)^{2/3}.
		\]
		Thus
		\begin{align}
			\Psi_K(\tau)
			\le
			2a_3^{1/5}
			\left(\frac{a_0}{K}\right)^{4/5}
			+
			\left(\frac{a_1a_0}{K}\right)^{1/2}
			+
			\tilde a_{2,K}^{1/3}
			\left(\frac{a_0}{K}\right)^{2/3}.
			\label{eq:case_alpha_a3}
		\end{align}
	\end{enumerate}
	\noindent	Combining \eqref{eq:case_alpha_bar}--\eqref{eq:case_alpha_a3}, and using
	nonnegativity of all terms, gives
	\begin{align}
		\Psi_K(\tau)
		\le
		2\left(\frac{a_1a_0}{K}\right)^{1/2}
		+
		2\tilde a_{2,K}^{1/3}
		\left(\frac{a_0}{K}\right)^{2/3}
		+
		2a_3^{1/5}
		\left(\frac{a_0}{K}\right)^{4/5}
		+
		\frac{a_0}{\bar\tau K}.
		\label{eq:general_stepsize_bound_compact}
	\end{align}
	Substituting this bound into \eqref{eq:general_stepsize_recursion} yields
	\[
	\frac1K\sum_{k=0}^{K-1}\cE_k
	\le
	2\left(\frac{a_1a_0}{K}\right)^{1/2}
	+
	2\tilde a_{2,K}^{1/3}
	\left(\frac{a_0}{K}\right)^{2/3}
	+
	2a_3^{1/5}
	\left(\frac{a_0}{K}\right)^{4/5}
	+
	\frac{a_0}{\bar\tau K}.
	\]
	Finally, by \eqref{eq:tilde_a2_def} and the subadditivity of
	$t\mapsto t^{1/3}$ on $\mathbb R_+$,
	\[
	\tilde a_{2,K}^{1/3}
	=
	\left(a_2+\frac{\tilde a_2}{K}\right)^{1/3}
	\le
	a_2^{1/3}
	+
	\left(\frac{\tilde a_2}{K}\right)^{1/3}.
	\]
	Hence
	\[
	\tilde a_{2,K}^{1/3}
	\left(\frac{a_0}{K}\right)^{2/3}
	\le
	a_2^{1/3}
	\left(\frac{a_0}{K}\right)^{2/3}
	+
	\tilde a_2^{1/3}\frac{a_0^{2/3}}{K}.
	\]
	This proves \eqref{eq:general_stepsize_bound}.
\end{proof}
\subsection{P\L~case}
\begin{lemma}[\sc\small Logarithmic stepsize for linear rate]
	\label{lem:linear_stepsize_quartic}
	\rm
	Let $K\ge1$, $\mu>0$, $\bar\tau>0$, and let
	$\widetilde a_2,a_2,a_3,a_4\ge0$, $a_0>0$, and $a_1>0$.
	Suppose that for every $\tau\in(0,\bar\tau]$ a nonnegative quantity
	$\cR_K$ satisfies
	\begin{align}\label{eq:quartic_recursion}
		\cR_K
		&\le
		\exp\left(
		-\frac{\tau\mu K}{2}
		\right)
		\left(
		a_0+\widetilde a_2\tau^2
		\right)
		+a_1\tau+a_2\tau^2+a_3\tau^3+a_4\tau^4 .
	\end{align}
	With the choice
	\begin{align}\label{eq:quartic_alpha}
		\tau
		&=
		\min\left\{
		\frac{2}{\mu K}\ln T_K,
		\bar\tau
		\right\},
		\quad
		T_K
		\define
		\max\left\{
		\exp(1),
		\frac{\mu (a_0+\widetilde a_2\bar\tau^2)K}{a_1}
		\right\},
	\end{align}
	we have
	\begin{align}\label{eq:quartic_bound}
		\cR_K
		&\le
		\widetilde\cO\left(
		\frac{a_1}{\mu K}
		+
		\frac{a_2}{\mu^2K^2}
		+
		\frac{a_3}{\mu^3K^3}
		+
		\frac{a_4}{\mu^4K^4}
		\right)
		+
		\left(
		a_0+\widetilde a_2\bar\tau^2
		\right)
		\exp\left(
		-\frac{\mu\bar\tau K}{2}
		\right).
	\end{align}
\end{lemma}

\begin{proof}
	We distinguish the two active cases in \eqref{eq:quartic_alpha}.
	
	\begin{itemize}
		\item \emph{Case 1: $\tau=\frac{2}{\mu K}\ln T_K\le\bar\tau$.}
		By the definition of $\tau$ in \eqref{eq:quartic_alpha},
		\begin{align}\label{eq:exp_equals_T}
			\exp\left(
			-\frac{\tau\mu K}{2}
			\right)
			&=
			\frac{1}{T_K}
			\le
			\frac{a_1}
			{\mu (a_0+\widetilde a_2\bar\tau^2)K},
		\end{align}
		where the inequality uses the definition of $T_K$ in
		\eqref{eq:quartic_alpha}. Since $\tau\le\bar\tau$, the entire
		exponentially weighted term collapses to:
		\begin{align}\label{eq:exp_collapse}
			\bigl(a_0+\widetilde a_2\tau^2\bigr)
			\exp\left(
			-\frac{\tau\mu K}{2}
			\right)
			&=
			\frac{a_0+\widetilde a_2\tau^2}{T_K}
			\;\overset{\tau\le\bar\tau}{\le}\;
			\frac{a_0+\widetilde a_2\bar\tau^2}{T_K}
			\;\overset{\eqref{eq:exp_equals_T}}{\le}\;
			\frac{a_1}{\mu K}.
		\end{align}
		The polynomial terms in \eqref{eq:quartic_recursion} satisfy
		\begin{align}\label{eq:poly_case1}
			a_1\tau
			&=
			\frac{2a_1\ln T_K}{\mu K},
			\quad
			a_2\tau^2
			=
			\frac{4a_2\ln^2 T_K}{\mu^2K^2},
			\nonumber\\
			a_3\tau^3
			&=
			\frac{8a_3\ln^3 T_K}{\mu^3K^3},
			\quad
			a_4\tau^4
			=
			\frac{16a_4\ln^4 T_K}{\mu^4K^4}.
		\end{align}
		Substituting \eqref{eq:exp_collapse} and \eqref{eq:poly_case1} into
		\eqref{eq:quartic_recursion} gives
		\begin{align}\label{eq:quartic_case1}
			\cR_K
			&\le
			\frac{(1+2\ln T_K)a_1}{\mu K}
			+
			\frac{4a_2\ln^2 T_K}{\mu^2K^2}
			+
			\frac{8a_3\ln^3 T_K}{\mu^3K^3}
			+
			\frac{16a_4\ln^4 T_K}{\mu^4K^4}.
		\end{align}
		
		\item \emph{Case 2: $\tau=\bar\tau\le\frac{2}{\mu K}\ln T_K$.}
		Here $\tau=\bar\tau$, so the exponentially weighted term in
		\eqref{eq:quartic_recursion} is exactly
		\begin{align}\label{eq:case2_exp}
			\bigl(
			a_0+\widetilde a_2\bar\tau^2
			\bigr)
			\exp\left(
			-\frac{\mu\bar\tau K}{2}
			\right).
		\end{align}
		Since $\bar\tau\le\frac{2}{\mu K}\ln T_K$ by the active-case assumption,
		the polynomial terms satisfy
		\begin{align}\label{eq:poly_case2}
			a_1\bar\tau
			&\le
			\frac{2a_1\ln T_K}{\mu K},
			\quad
			a_2\bar\tau^2
			\le
			\frac{4a_2\ln^2 T_K}{\mu^2K^2},
			\nonumber\\
			a_3\bar\tau^3
			&\le
			\frac{8a_3\ln^3 T_K}{\mu^3K^3},
			\quad
			a_4\bar\tau^4
			\le
			\frac{16a_4\ln^4 T_K}{\mu^4K^4}.
		\end{align}
		Combining \eqref{eq:case2_exp} and \eqref{eq:poly_case2} gives
		\begin{align}\label{eq:quartic_case2}
			\cR_K
			&\le
			\bigl(
			a_0+\widetilde a_2\bar\tau^2
			\bigr)
			\exp\left(
			-\frac{\mu\bar\tau K}{2}
			\right)
			+
			\frac{2a_1\ln T_K}{\mu K}
			\nonumber\\
			&\quad
			+
			\frac{4a_2\ln^2 T_K}{\mu^2K^2}
			+
			\frac{8a_3\ln^3 T_K}{\mu^3K^3}
			+
			\frac{16a_4\ln^4 T_K}{\mu^4K^4}.
		\end{align}
	\end{itemize}
	Combining \eqref{eq:quartic_case1} and \eqref{eq:quartic_case2}, and
	absorbing logarithmic factors in $T_K$ into $\widetilde\cO(\cdot)$, yields
	\eqref{eq:quartic_bound}.
\end{proof}
		\bibliographystyle{ieeetr}
		\bibliography{myref}

\end{document}